\documentclass[11pt,reqno]{amsart}
\usepackage[utf8]{inputenc}
\usepackage[T1]{fontenc}
\usepackage{lmodern}
\usepackage{amsmath}
\usepackage{amssymb}
\usepackage{amsfonts}
\usepackage{mathrsfs}
\usepackage{amsthm}
\usepackage{commath}
\usepackage{dsfont}
\usepackage{hyperref}
\usepackage{doi}
\usepackage{enumitem}

\usepackage{mathtools}

\newcommand{\R}{\mathbb{R}}

\newcommand{\D}{\mathcal{D}}

\newcommand{\HD}{\mathsf{HD}}
\newcommand{\LD}{\mathsf{LD}}
\newcommand{\BA}{\mathsf{BA}}
\newcommand{\BS}{\mathsf{BS}}
\newcommand{\F}{\mathsf{F}}

\newcommand{\Stop}{\mathsf{Stop}}
\newcommand{\Tree}{\mathsf{Tree}}

\newcommand{\RFar}{R_{\mathsf{Far}}}
\newcommand{\one}{\mathds{1}}
\DeclareMathOperator{\supp}{supp}
\DeclareMathOperator{\graph}{graph}
\newcommand\restr[2]{{
		\left.\kern-\nulldelimiterspace 
		#1 
		\right|_{#2} 
}}

\newcommand{\Hn}[1]{\restr{\mathcal{H}^n}{#1}}

\DeclareMathOperator{\diam}{diam}
\DeclareMathOperator{\dist}{dist}
\DeclareMathOperator{\lip}{\text{Lip}}

\newtheorem{theoremalph}{Theorem}
\newtheorem{theorem}{Theorem}[section]
\newtheorem{lemma}[theorem]{Lemma}
\newtheorem{cor}[theorem]{Corollary}
\theoremstyle{definition}
\newtheorem{remark}[theorem]{Remark}

\numberwithin{equation}{section}

\newcounter{AbsConstants}

\title{Sufficient condition for rectifiability involving Wasserstein distance $W_2$}

\author{Damian D\k{a}browski}
\address{Damian D\k{a}browski\newline\indent Departament de Matem\`atiques, Universitat Aut\`onoma de Barcelona; Barcelona Graduate School of Mathematics (BGSMath)\newline\indent Edifici C Facultat de Ci\`encies, 08193 Bellaterra (Barcelona, Catalonia) }
\email{ddabrowski ``at'' mat.uab.cat}

\begin{document}
	\begin{abstract}
		A Radon measure $\mu$ is $n$-rectifiable if it is absolutely continuous with respect to $\mathcal{H}^n$ and $\mu$-almost all of $\supp\mu$ can be covered by Lipschitz images of $\R^n$. In this paper we give two sufficient conditions for rectifiability, both in terms of square functions of flatness-quantifying coefficients. The first condition involves the so-called $\alpha$ and $\beta_2$ numbers. The second one involves $\alpha_2$ numbers -- coefficients quantifying flatness via Wasserstein distance $W_2$. Both conditions are necessary for rectifiability, too -- the first one was shown to be necessary by Tolsa, while the necessity of the $\alpha_2$ condition is established in our recent paper. Thus, we get two new characterizations of rectifiability.
	\end{abstract}
	\maketitle
	
	\section{Introduction}
	Let $1\le n\le d$ be integers. We say that a Radon measure $\mu$ on $\R^d$ is $n$-rectifiable if there exist countably many Lipschitz maps $f_i:\R^n\rightarrow\R^d$ such that 
	\begin{equation}\label{eq:measure concentrated on images}
	\mu(\R^d\setminus\bigcup_i f_i(\R^n))=0,
	\end{equation}
	and moreover $\mu$ is absolutely continuous with respect to $n$-dimensional Hausdorff measure $\mathcal{H}^n$. A set $E\subset \R^d$ is $n$-rectifiable if the measure $\Hn{E}$ is $n$-rectifiable. We will often omit $n$ and just write ``rectifiable''.
	
	Rectifiable sets and measures have been studied for many decades now, and the first results in this area are due to Besicovitch. See \cite[Chapters 15--18]{mattila1999geometry} for some classical characterizations of rectifiability involving densities, tangent measures, and projections. The aim of this paper is to prove a sufficient condition for rectifiability involving the so-called $\alpha_2$ coefficients. In fact, we will show a bit more: a sufficient condition for rectifiability involving $\alpha$ and $\beta_2$ numbers. We introduce the whole menagerie of flatness-quantifying coefficients below.
	
	\subsection{\texorpdfstring{$\beta$ and $\alpha$}{beta and alpha} coefficients}
	$\beta$ numbers were introduced by Jones in \cite{jones1990rectifiable}, where they were used to characterize subsets of rectifiable curves. They were further developed by David and Semmes in \cite{david1991singular,david1993analysis}. For $1\le p <\infty$ and a Radon measure $\mu$ on $\R^d$ we define
	\begin{equation}\label{eq:homogeneous betas}
	\beta^h_{\mu,p}(x,r)=\inf_L\left(\frac{1}{r^n}\int_{B(x,r)}\left(\frac{\dist(y,L)}{r}\right)^p\ d\mu(y)\right)^{1/p},
	\end{equation}
	where the infimum runs over all $n$-planes $L$ intersecting $B(x,r)$. The letter $h$ in the superscript stands for \emph{homogeneous} and refers to the normalizing factor $r^{n}$. In our setting it will be more convenient to normalize by $\mu(B(x,3r))$ instead, and so we define
	\begin{equation*}
	\beta_{\mu,p}(x,r)=\inf_L\left(\frac{1}{\mu(B(x,3r))}\int_{B(x,r)}\left(\frac{\dist(y,L)}{r}\right)^p\ d\mu(y)\right)^{1/p}.
	\end{equation*}
	For a ball $B=B(x,r)$ we will sometimes write $\beta_{\mu,p}(B)$ instead of $\beta_{\mu,p}(x,r)$, and we will do the same with all the other coefficients introduced below.
	
	Another way to quantify flatness of measures is offered by $\alpha$ numbers, introduced by Tolsa in \cite{tolsa2008uniform}. 
	To define them, we need a distance on the space of measures. Given Radon measures $\mu$ and $\nu$, and an open ball $B$, we set
	\begin{equation*}
	F_B(\mu, \nu)=\sup\left\{\left\lvert\int\phi\ d\mu - \int\phi\ d\nu \right\rvert\ :\ \phi\in\text{Lip}_1(B)\right\},
	\end{equation*}
	where
	\begin{equation*}
	\lip_1(B)=\{\phi\ :\ \lip(\phi)\le 1,\ \supp \phi\subset B \}.
	\end{equation*}	
	The coefficient $\alpha$ of a measure $\mu$ in $B$ is defined as
	\begin{equation*}
	\alpha_{\mu}(B)=\inf_{c,L}\frac{1}{r(B)\mu(3B)} F_B(\mu,c\Hn{L}),
	\end{equation*}
	where the infimum runs over all $c\ge 0$ and all $n$-planes $L$.
	
	We prove the following sufficient condition for rectifiability in terms of $\alpha$ and $\beta_2$ square functions.
	\begin{theorem}\label{thm:sufficient condition}
		Let $\mu$ be a Radon measure on $\R^d$. Suppose that
		\begin{equation}\label{eq:main theorem assumption on alphas}
		\int_0^{1}\alpha_{\mu}(x,r)^2\ \frac{dr}{r}<\infty \qquad\text{for $\mu$-a.e. $x\in\R^d,$}
		\end{equation}
		and
		\begin{equation}\label{eq:main theorem assumption on betas}
		\int_0^{1}\beta_{\mu,2}(x,r)^2\ \frac{dr}{r}<\infty \qquad\text{for $\mu$-a.e. $x\in\R^d.$}
		\end{equation}
		Then $\mu$ is $n$-rectifiable.
	\end{theorem}
	
	Since Tolsa has shown in \cite{tolsa2015characterization} that \eqref{eq:main theorem assumption on alphas} and \eqref{eq:main theorem assumption on betas} are also necessary conditions for rectifiability, we immediately get the following characterization.
	
	\begin{cor}
		Let $\mu$ be a Radon measure on $\R^d$. Then, $\mu$ is $n$-rectifiable if and only if \eqref{eq:main theorem assumption on alphas} and \eqref{eq:main theorem assumption on betas} hold for $\mu$-a.e. $x\in\R^d$.
	\end{cor}
	
	A number of similar characterizations has been shown in recent years. First, recall that upper and lower $n$-dimensional densities of a Radon measure $\mu$ at $x\in\R^d$ are
	\begin{equation*}
	\Theta^{n,*}(x,\mu)=\limsup_{r\to 0^+}\frac{\mu(B(x,r))}{r^n},\qquad \Theta_{*}^n(x,\mu)=\liminf_{r\to 0^+}\frac{\mu(B(x,r))}{r^n},
	\end{equation*}
	respectively. If they are equal, we set $\Theta^{n}(x,\mu) = \Theta^{n,*}(x,\mu) = \Theta_{*}^n(x,\mu)$ and we call it $n$-dimensional density of $\mu$ at $x$.
	
	In \cite{tolsa2015characterization} it was shown that a rectifiable measure $\mu$ satisfies
	\begin{equation}\label{eq:Jones square function finite}
	\int_0^1 \beta_{\mu,2}^h(x,r)^2\ \frac{dr}{r}<\infty\qquad\text{for $\mu$-a.e. $x\in\R^d.$}\tag{$\ref*{eq:main theorem assumption on betas}^h$}
	\end{equation}
	On the other hand, Azzam and Tolsa proved in \cite{azzam2015characterization} that if a Radon measure $\mu$ satisfies \eqref{eq:Jones square function finite} and 
	$0<\Theta^{n,*}(x,\mu)<\infty$ for $\mu$-a.e. $x\in\R^d$,
	then $\mu$ is $n$-rectifiable. More recently, Edelen, Naber and Valtorta \cite{edelen2016quantitative} managed to weaken the assumption on densities to
	\begin{equation}\label{eq:weakened upper density bdd}
	\Theta^{n,*}(x,\mu)>0\quad \text{and} \quad \Theta_*^{n}(x,\mu)<\infty\qquad\text{for $\mu$-a.e. $x\in\R^d$.}
	\end{equation}
	An alternative proof showing that \eqref{eq:Jones square function finite} and \eqref{eq:weakened upper density bdd} are sufficient for rectifiability was later given in \cite{tolsa2017rectifiability}.
	
	\begin{theoremalph}[\cite{tolsa2015characterization,azzam2015characterization,edelen2016quantitative}]\label{thm:beta characterization}
		Let $\mu$ be a Radon measure on $\R^d$. Then, $\mu$ is $n$-rectifiable if and only if \eqref{eq:Jones square function finite} and \eqref{eq:weakened upper density bdd} hold for $\mu$-a.e. $x\in\R^d$.
	\end{theoremalph}
	
	There is a reason why the theorems above are stated for $\beta_2$ numbers (as opposed to $\beta_p$ numbers with $p\neq 2$). The conditions \eqref{eq:main theorem assumption on betas} and \eqref{eq:Jones square function finite} with $\beta_{2}$ numbers replaced by $\beta_p$ are necessary for rectifiability only for $1\le p\le 2$. On the other hand, conditions \eqref{eq:Jones square function finite} and $0<\Theta^{n,*}(\mu,x)<\infty$ $\mu$-almost everywhere are sufficient for rectifiability only if $p\ge 2$. Relevant counterexamples were given in \cite{tolsa2017rectifiability}. However, if we assume more on densities (namely that $\Theta_*^{n}(\mu,x)>0$ and $\Theta^{n,*}(\mu,x)<\infty$ $\mu$-almost everywhere), then the finiteness of $\beta_p$ square function for certain $p<2$ becomes sufficient for rectifiability, see \cite{pajot1997conditions,badger2016two}.
	
	It is also worth mentioning that appropriate versions of $\beta$ numbers give rise to various necessary and/or sufficient conditions for the so called \emph{Federer rectifiability} of measures. We say that a measure is $n$-rectifiable in the sense of Federer if it satisfies \eqref{eq:measure concentrated on images}, and no absolute continuity with respect to $\mathcal{H}^n$ is required. This notion is more difficult to characterize than the one we work with, as illustrated by the surprising example of Garnett, Killip and Schul \cite{garnett2010doubling}. In the case $n=1$ significant progress has been achieved in \cite{lerman2003quantifying,badger2015multiscale,badger2016two,azzam2016characterization,badger2017multiscale,martikainen2018boundedness}. An excellent overview of the problem is given in the survey \cite{badger2018generalized}.
	
	Concerning $\alpha$ numbers, as we already mentioned, Tolsa showed in \cite{tolsa2015characterization} that \eqref{eq:main theorem assumption on alphas} is necessary for rectifiability. 
	Is it also sufficient? Azzam, David, and Toro proved in \cite{azzam2016wasserstein} that if $\mu$ is doubling, then some condition related to \eqref{eq:main theorem assumption on alphas} is sufficient for rectifiability. In \cite{orponen2018absolute} Orponen showed that for $n=d=1$ a variant of \eqref{eq:main theorem assumption on alphas} is sufficient for rectifiability (which in this case is equivalent to absolute continuity with respect to $\mathcal{H}^1$).
	Finally, Azzam, Tolsa and Toro \cite{azzam2018characterization} proved that a measure $\mu$ satisfying \eqref{eq:main theorem assumption on alphas} which is also pointwise doubling, i.e. such that
	\begin{equation}\label{eq:pointwise doubling}
	\limsup_{r\to 0^+}\frac{\mu(B(x,2r))}{\mu(B(x,r))}<\infty\qquad\text{for $\mu$-a.e. $x\in\R^d$,}
	\end{equation}
	is rectifiable. 
	\begin{theoremalph}[{\cite{tolsa2015characterization,azzam2018characterization}}]\label{thm:alpha doubling characterization}
		Let $\mu$ be a Radon measure on $\R^d$. Then, $\mu$ is $n$-rectifiable if and only if \eqref{eq:main theorem assumption on alphas} and \eqref{eq:pointwise doubling} hold for $\mu$-a.e. $x\in\R^d$.
	\end{theoremalph}
	Also in \cite{azzam2018characterization}, the authors construct a purely $1$-unrectifiable measure on $\R^2$ satisfying \eqref{eq:main theorem assumption on alphas}. This shows that, for general $n$ and $d$, \eqref{eq:main theorem assumption on alphas} on its own is \emph{not} a sufficient condition for rectifiability.
	
	Finally, let us mention that in \cite{tolsa2014rectifiability,tolsa2017rectifiable} rectifiable sets and measures are characterized using yet another kind of square functions. They involve the so-called $\Delta$ numbers, defined as $\Delta_{\mu}(x,r)=|\frac{\mu(B(x,r))}{r^n} - \frac{\mu(B(x,2r))}{(2r)^n}|$. The results from \cite{tolsa2014rectifiability}, valid for arbitrary $n$, require $0<\Theta^n_*(\mu,x)\le\Theta^{n,*}(\mu,x)<\infty$ for $\mu$-a.e $x\in\R^d$. On the other hand, in \cite{tolsa2017rectifiable} it was shown that for $n=1$ analogous results hold under weaker assumption $0<\Theta^{1,*}(x,\mu)<\infty$ for $\mu$-a.e. $x\in\R^d$.
%
%
	\subsection{\texorpdfstring{$\alpha_p$}{alpha p} coefficients}
	Coefficients $\alpha_p$ were introduced by Tolsa in \cite{tolsa2012mass}. They can be thought of as a generalization of $\alpha$ numbers -- in fact, under relatively mild assumptions, one has $\alpha_{\mu}(B)\approx \alpha_{\mu,1}(B)$, see \lemref{lem:alpha2 controls alpha and beta2} and \cite[Lemma 5.1]{tolsa2012mass}. As in the case of $\alpha$ coefficients, in order to define $\alpha_p$ numbers we need a metric on the space of measures.
	
	Let $1\le p <\infty$, and let $\mu, \nu$ be two probability Borel measures on $\R^d$ satisfying $\int |x|^p\ d\mu<\infty,\ \int |x|^p\ d\nu<\infty$. The Wasserstein distance $W_p$ between $\mu$ and $\nu$ is defined as 
	\begin{equation*}
	W_p(\mu,\nu)=\left(\inf_{\pi}\int_{\R^d\times\R^d} |x-y|^p\ d\pi(x,y)\right)^{1/p},
	\end{equation*}
	where the infimum is taken over all transport plans between $\mu$ and $\nu$, i.e. Borel probability measures $\pi$ on $\R^d\times\R^d$ satisfying $\pi(A\times\R^d)=\mu(A)$ and $\pi(\R^d\times A)=\nu(A)$ for all measurable $A\subset\R^d$. The same definition makes sense if instead of probability measures we consider $\mu,\ \nu,$ and $\pi$ of the same total mass. For more information on Wasserstein distance see for example \cite[Chapter 7]{villani2003topics} or \cite[Chapter 6]{villani2008optimal}.
	
	Similarly as $\alpha$ numbers, $\alpha_p$ numbers quantify how far is a given measure from being a flat measure, that is, from being of the form $c\Hn{L}$ for some constant $c>0$ and some $n$-plane $L$. In order to measure it locally (say, in a ball $B$), we introduce the following auxiliary function. 
	
	Let $\varphi:\R^d\rightarrow [0,1]$ be a radial Lipschitz function satisfying $\varphi \equiv 1$ in $B(0,2)$, $\supp\varphi \subset B(0,3)$, and for all $x\in B(0,3)$
	\begin{gather*}
	c^{-1}\dist(x,\partial B(0,3))^2\le \varphi(x)\le c\dist(x,\partial B(0,3))^2,\\
	|\nabla\varphi(x)|\le c \dist(x,\partial B(0,3)),
	\end{gather*}
	for some constant $c>0$. Given a ball $B = B(x,r)\subset\R^d$ we set
	\begin{equation*}
	\varphi_{{B}}(y) = \varphi\left(\frac{y-x}{r}\right).
	\end{equation*}
	$\varphi_B$ should be thought as a regularized characteristic function of $B$. For $1\le p <\infty$, a Radon measure $\mu$ on $\R^d$, and a ball $B\subset\R^d,$ we define
	\begin{equation*}
	\alpha_{\mu,p}(B) = \inf_L \frac{1}{r(B)\mu(3B)^{1/p}}W_p(\varphi_B\mu,a_{B,L}\varphi_B\Hn{L}),
	\end{equation*}
	where the infimum is taken over all $n$-planes $L$ intersecting $B$, and
	\begin{equation*}
	a_{{B,L}} = \frac{\int \varphi_B\ d\mu}{\int \varphi_B\ d\Hn{L}}.
	\end{equation*}
	
	Coefficients $\alpha_p$ were introduced in \cite{tolsa2012mass} with the aim of characterizing \emph{uniformly} rectifiable measures. Uniform rectifiability, introduced by David and Semmes in \cite{david1991singular,david1993analysis}, is a stronger, more quantitative version of rectifiability. 
%
	Can $\alpha_{p}$ numbers be used to characterize rectifiability also in the non-uniform case? Driven by this question, our main motivation for proving \thmref{thm:sufficient condition} was to get the following sufficient condition for rectifiability.
	\begin{theorem}\label{thm:sufficient condition with alpha2}
		Let $\mu$ be a Radon measure on $\R^d$. Suppose that
		\begin{equation}\label{eq:alpha2 square function}
		\int_0^{1}\alpha_{\mu,2}(x,r)^2\ \frac{dr}{r}<\infty \qquad\text{for $\mu$-a.e. $x\in\R^d.$}
		\end{equation}
		Then $\mu$ is $n$-rectifiable.
	\end{theorem}
	\thmref{thm:sufficient condition with alpha2} follows immediately from \thmref{thm:sufficient condition} because, as shown in \lemref{lem:alpha2 controls alpha and beta2}, $\alpha_2$ numbers bound from above both $\alpha$ and $\beta_2$ numbers.	
	In \cite{dabrowski2019necessary} we show that \eqref{eq:alpha2 square function} is also a necessary condition for rectifiability, and so we get the following characterization.
	\begin{cor}\label{thm:thm equivalence}
		Let $\mu$ be a Radon measure on $\R^d$. Then, $\mu$ is $n$-rectifiable if and only if for $\mu$-a.e. $x\in\R^d$ we have
		\begin{equation*}
		\int_0^{1}\alpha_{\mu,2}(x,r)^2\ \frac{dr}{r}<\infty.
		\end{equation*}
	\end{cor}	
	
	We would like to stress that, compared to \thmref{thm:beta characterization} and \thmref{thm:alpha doubling characterization}, the characterization above does not make any additional assumptions on densities or on doubling properties of the measure.
	
	The organization of the paper, as well as the general strategy of the proof, are outlined in Section \ref{sec:sketch of the proof}. For now, let us just say that \lemref{lem:main lemma}, our main lemma, can be seen as a technical, more quantitative version of \thmref{thm:sufficient condition}. If one prefers working with homogeneous coefficients $\beta_{\mu,2}^h$ and $\alpha_{\mu}^h$ (where $\alpha_{\mu}^h(x,r)=\frac{\mu(B(x,3r))}{r^n}\alpha_{\mu}(x,r)$), then a possible ``homogenized'' modification of \lemref{lem:main lemma} is discussed in \remref{rem:homogeneous version of main lemma}. However, it is clear that 
	``homogenized'' (i.e. with $\alpha$ and $\beta_2$ numbers replaced by their homogeneous counterparts) versions of \thmref{thm:sufficient condition} and \thmref{thm:sufficient condition with alpha2} are not true (unless we assume more about densities) -- think of Lebesgue measure on $\R^d$.
	
	\subsection*{Acknowledgements} The author would like to express his deep gratitude to Xavier Tolsa for all his help and guidance. He acknowledges the support from the Spanish Ministry of Economy and Competitiveness, through the María de Maeztu Programme for Units of Excellence in R\&D (MDM-2014-0445), and also partial support from the Catalan Agency for Management of University and Research Grants (2017-SGR-0395), and from the Spanish Ministry of Science, Innovation and Universities (MTM-2016-77635-P). 
\section{Sketch of the proof}\label{sec:sketch of the proof}
The proof of \thmref{thm:sufficient condition} is organized as follows. In Section \ref{sec:estimates of alpha beta coeffs} we provide basic estimates of $\alpha$ and $\beta$ coefficients, while in Section \ref{sec:DM cubes} we recall the definition and some properties of the David-Mattila lattice, which will be used further on. In Section \ref{sec:main lemma} we formulate the main lemma. Given an appropriate David-Mattila cube $R_0$, the main lemma provides us with a Lipschitz graph $\Gamma$ such that we have $\mu\ll \Hn{\Gamma}$ on a large chunk of $\Gamma\cap R_0$, and $\mu(\Gamma\cap R_0)\ge \frac{1}{2}\mu(R_0)$. In the same section we show how to use the main lemma to prove \thmref{thm:sufficient condition}. Everything that follows is dedicated to proving the main lemma.

In Section \ref{sec:stopping cubes} we perform the usual stopping time argument. We define the family of stopping cubes $\Stop$, comprising high density cubes $\HD$, low density cubes $\LD$, big angle cubes $\BA$ (cubes whose best approximating planes form a big angle with $L_0$, the best approximating plane of $R_0$), big square function cubes $\BS$ (cubes with a big portion of points for which the square functions are larger than a certain threshold), and far cubes $\F$ (cubes with a big portion of $\RFar$, points that are far from certain best approximating planes). Cubes not contained in any of the stopping cubes form the $\Tree$. Next, we show various good properties of cubes from the $\Tree$, as well as estimate the measure of cubes from $\BS$ and $\F$ (it is easy).

Section \ref{sec:construction of graph} is devoted to constructing the Lipschitz graph $\Gamma$. One possible way to do it would be to use the tools from \cite{david2012reifenberg} -- this was done for example in \cite{azzam2015characterization,azzam2018characterization}. In this paper we decided to use another well-known method, dating back at least to \cite{david1991singular} and \cite{leger1999menger}. We follow the way it was applied in \cite{chunaev2018family} and \cite{tolsa2014analytic}. It consists of showing that $R_0\setminus\bigcup_{Q\in\Stop}Q$ forms a graph of a Lipschitz map $F$ defined on a subset of $L_0$, and then carefully extending $F$ to the whole $L_0$. The remaining part of the paper is dedicated to showing that the measure of stopping cubes is small.

In Section \ref{sec:LD} we first show that cubes from $\Tree$ lie close to $\Gamma$ (the graph of $F$), and then use this property to estimate the measure of low density cubes. Roughly speaking, we may cover (almost all) $\LD$ cubes with a family of (almost) disjoint balls satisfying $B\cap\Gamma\approx r(B)^n$, and such that the densities $\Theta_{\mu}(B)$ are low. Small measure of $\LD$ easily follows. It is crucial that we have finiteness of the $\beta_2$ square function \eqref{eq:main theorem assumption on betas}, as it lets us estimate the size of $\RFar$ (see  \lemref{lem:small measure of RFar}). This approach to bounding the measure of low density cubes comes from \cite{azzam2015characterization}.

In Section \ref{sec:approximating measure} we define a measure $\nu$ supported on $\Gamma$. We show that $\nu$ is very close to $\mu$ in the sense of distance $F_B(\mu,\nu)$, so that the $\alpha_{\nu}$ numbers are close to $\alpha_{\mu}$. The measure $\nu$ is then used in Section \ref{sec:HD} to estimate the size of high density set. The general idea is to consider $f$ -- the density of $\nu$ with respect to $\Hn{\Gamma}$, and then to bound the $L^2$ norm of $|f-c_0|$, where $c_0$ is a certain constant. We do it using the smallness of $\alpha_{\mu}$ square function \eqref{eq:main theorem assumption on alphas}, the fact that $\nu$ approximates $\mu$ well, and an appropriate type of Paley-Littlewood result (see \eqref{eq:classical harmonic analysis result}). Estimating $\lVert f- c_0\rVert_{L^2}$ requires a lot of work, but once we have it, it is not very difficult to bound the measure of $\HD$ cubes. Roughly speaking, high density cubes correspond to big values of $f$, and those we can control since $\lVert f- c_0\rVert_{L^2}$ is small. This method of estimating $\HD$ is due to \cite{azzam2018characterization}, where a similar approach from \cite{tolsa2017rectifiable} was refined and simplified.

Finally, in Section \ref{sec:BS} we bound the size of big angle cubes $\BA$. First, we show that this amounts to estimating $\lVert \nabla F\rVert_{L^2}$ (recall that $F$ is the Lipschitz map whose graph is $\Gamma$). Using Dorronsoro's theorem, this reduces to estimating the $\beta_{\sigma,1}$ square function, where $\sigma$ is the surface measure on $\Gamma$. This could be done using the smallness of either $\beta_{\mu,2}$ or $\alpha_{\mu}$ square functions. For us it was easier to deal with $\alpha_{\mu}$, due to all the estimates from Section \ref{sec:approximating measure}.

Thus, having estimated the measure of the stopping region, the proof of the main lemma is finished.
	\subsection*{Notation}
Throughout the paper we will write $A\lesssim B$ whenever $A\le CB$ for some constant $C$. All such implicit constants may depend on dimensions $n, d$, and on constants $A_0, C_0$, which will be fixed in Section \ref{sec:DM cubes}. If the implicit constant depends also on some other parameter $t$, we will write $A\lesssim_{t} B$. The notation $A\approx B$ means $A\lesssim B\lesssim B$, and $A\approx_t B$ means $A\lesssim_t B\lesssim_t B$.

We denote by $B(z,r)\subset\R^d$ an open ball with center at $z\in\R^d$ and radius $r>0$. Given a ball $B$, its center and radius are denoted by $z(B)$ and $r(B)$, respectively. If $\lambda>0$, then $\lambda B$ is defined as a ball centered at $z(B)$ of radius $\lambda r(B)$.

For a ball $B$ and measure $\mu$, we define the $n$-dimensional density of $\mu$ at $B$ as
\begin{equation*}
\Theta_{\mu}(B) = \frac{\mu(B)}{r(B)^n}.
\end{equation*}

Given two $n$-planes $L_1, L_2$, let $L_1'$ and $L_2'$ be the respective parallel $n$-planes passing through $0$. Then,
\begin{equation*}
\measuredangle(L_1,L_2)=\dist_H(L_1'\cap B(0,1),\ L_2'\cap B(0,1)),
\end{equation*}
where $\dist_H$ stands for Hausdorff distance between two sets. $\measuredangle(L_1,L_2)$ can be seen as an angle between $L_1$ and $L_2$.

Given an affine subspace $L\subset \R^d$, we will denote the orthogonal projection onto $L$ by $\Pi_L$. The orthogonal projection onto $L^{\perp}$ will be denoted by $\Pi^{\perp}_L$.

Given a set $A\subset\R^d$, we denote by $\one_A:\R^d\to\{0,1\}$ a characteristic function of $A$.

Finally, for sets $A,B\subset\R^d$ we define
\begin{equation*}
	\dist(A,B) = \inf_{a\in A}\inf_{b\in B} |a-b|.
\end{equation*}
\section{Estimates of \texorpdfstring{$\alpha$}{alpha} and \texorpdfstring{$\beta$}{beta} coefficients}\label{sec:estimates of alpha beta coeffs}
In this section we provide some basic estimates of $\alpha$ and $\beta$ coefficients.
We begin by showing the relationship between $\alpha, \beta_2$, and $\alpha_2$ coefficients.
\begin{lemma}\label{lem:alpha2 controls alpha and beta2}
	Suppose that $\mu$ is a Radon measure, and $B$ is a ball intersecting $\supp\mu$. Then
	\begin{equation*}
	\beta_{\mu,2}(B)\le \alpha_{\mu,2}(B),
	\end{equation*}
	and
	\begin{equation*}
	\alpha_{\mu}(B)\le\alpha_{\mu,1}(B)\le\alpha_{\mu,2}(B).
	\end{equation*}
\end{lemma}

\begin{proof}
	To see $\beta_{\mu,2}(B)\le\alpha_{\mu,2}(B)$, let $L$ be a minimizing plane for $ \alpha_{\mu,2}(B)$ and $\pi$ be a minimizing transport plan between $\varphi_{{B}}\mu$ and $a_{B,L}\varphi_{{B}}\Hn{L}$, where $a_{B,L} = ({\int \varphi_B\ d\mu})/({\int \varphi_B\ d\Hn{L}})$ is as in the definition of $\alpha_{\mu,2}(B)$. Then, by the definition of a transport plan, and the fact that $\varphi_B\equiv 1$ on $B$, 
	\begin{equation*}
	\alpha_{\mu,2}(B)^2r(B)^2\mu(3B) = \int |x-y|^2\ d\pi(x,y)\ge \int_B \dist(x,L)^2\ d\mu \ge \beta_{\mu,2}(B)^2\mu(3B)r(B)^{2}.
	\end{equation*}
	
	For the estimate involving $\alpha$ numbers we will use the so-called Kantorovich duality for $W_1$ Wasserstein distance. It states that
	\begin{equation*}
	W_1(\mu,\nu) = \sup_{\lip(f)\le 1} \left\lvert \int f\ d\mu - \int f\ d\nu\right\rvert,
	\end{equation*}
	see \cite[Remark 6.5]{villani2008optimal} for more information. 
	
	Let $L$ be a minimizing plane for $\alpha_{\mu,1}(B)$, and let $a_{B,L}$ be as in the definition of $\alpha_{\mu,1}(B)$. Since $\varphi_B\equiv 1$ in $B$, it follows from the definition of $\alpha_{\mu}$ that
	\begin{multline*}
	\alpha_{\mu}(B)r(B)\mu(3B) \le F_B(\mu, a_{B,L}\Hn{L})=\sup_{\substack{\lip(f)\le 1\\\supp(f)\subset B}}
	\left\lvert \int f\ d\mu - \int f\ a_{B,L}d\Hn{L}\right\rvert
	\\= \sup_{\substack{\lip(f)\le 1\\\supp(f)\subset B}} \left\lvert \int f\varphi_B\ d\mu - \int f\varphi_B\ a_{B,L}d\Hn{L}\right\rvert
	\le \sup_{\lip(f)\le 1} \left\lvert \int f\varphi_B\ d\mu - \int f\varphi_B\ a_{B,L}d\Hn{L}\right\rvert
	\\= W_1(\varphi_B\mu, a_{B,L}\varphi_B\Hn{L}) = \alpha_{\mu,1}(B)r(B)\mu(3B).
	\end{multline*}
	The estimate $\alpha_{\mu,1}(B)\le\alpha_{\mu,2}(B)$ follows immediately by the Cauchy-Schwarz inequality and the fact that $\int\varphi_{{B}}\ d\mu \le \mu(3B)$.
\end{proof}

\begin{lemma}\label{lem:betas are smaller than alphas}
	Suppose that $\mu$ is a Radon measure on $\R^d$, and that $B\subset\R^d$ is a ball satisfying $\mu(3B)\approx\mu(6B)$. Then
	\begin{equation}\label{eq:beta1 beta2 estimate}
	\beta_{\mu,1}(B)\le\beta_{\mu,2}(B),
	\end{equation}
	and
	\begin{equation}\label{eq:beta1 alpha estimate}
	\beta_{\mu,1}(B)\lesssim\alpha_{\mu}(2B).
	\end{equation}
	Moreover, given balls $B_1\subset B_2$ such that $r(B_1)\approx r(B_2)$ and $\mu(3B_1)\approx\mu(3B_2)$ we have
	\begin{align}
	\beta_{\mu,2}(B_1)&\lesssim\beta_{\mu,2}(B_2),\label{eq:beta2 scaling estimate}\\
	\alpha_{\mu}(B_1)&\lesssim\alpha_{\mu}(B_2).\label{eq:alpha scaling estimate}
	\end{align}
\end{lemma}
\begin{proof}
	%
	The first estimate is a direct consequence of the Cauchy-Schwarz inequality.
	
	To prove the second estimate, let $L_B$ be the minimizing plane for $\beta_{\mu,1}(B)$. The estimate follows if we consider the 1-Lipschitz function $\phi = \psi \dist(x,L_B),$ where $\psi$ is $r(B)^{-1}$-Lipschitz, $\psi\equiv 1$ on $B$, and $\supp(\psi)\subset 2B$.
	
	The last two inequalities follow immediately from the definitions of $\beta_{\mu,2}$ and $\alpha_{\mu}$.
\end{proof}

\begin{lemma}\label{lem:c comparable to density}
	Suppose that $\mu$ is a Radon measure, $B$ is a ball with $\mu(B)>0$, $L$ an $n$-plane intersecting $0.9B$, and assume that $c$ minimizes $F_B(\mu,\, c\Hn{L})$. Then
	\begin{equation}\label{eq:c smaller then mu over Hn}
	c\lesssim \frac{\mu(B)}{r(B)^n}.
	\end{equation}
	Furthermore, there exists $\varepsilon>0$ such that if $\mu(0.9B)\approx\mu(3B)$, and   $F_B(\mu,\, c\Hn{L})\le\varepsilon\mu(3B)r(B),$ then 
	\begin{equation}\label{eq:c greater then mu over Hn}
	c\gtrsim \frac{\mu(3B)}{r(B)^n}.
	\end{equation}
\end{lemma}

\begin{proof}
	Let $r=r(B)$ and consider $\Phi(x) = (r-|x-z(B)|)_{+}\in\lip_1(B)$. It is not difficult to see that on a significant portion (say, a half) of the $n$-dimensional ball $L\cap B$ we have $\Phi(x)\approx r$, and so
	\begin{equation*}
	c\int \Phi(x)\ d\Hn{L}(x) \approx cr^{n+1}.
	\end{equation*}
	If we had $c\gg \mu(B)r^{-n}$, then
	\begin{equation*}
	F_B(\mu,c\Hn{L})\ge c\int \Phi(x)\ d\Hn{L}(x) - \int \Phi(x)\ d\mu(x)\ge Ccr^{n+1}-\mu(B)r\gg\mu(B)r.
	\end{equation*}
	But in that case the constant $\tilde{c}=0$ would be better than $c$, since we always have $F_B(\mu,0)\le \mu(B)r,$ and thus we reach a contradiction with optimality of $c$.
	
	Now, assume further that $F_B(\mu,\, c\Hn{L})\le\varepsilon\mu(3B)r,$ and $\mu(0.9B)\approx\mu(3B),$ so that $\int \Phi(x)\ d\mu(x)\approx\mu(3B)r$. If we had $c\ll \mu(3B)r^{-n}$, then
	\begin{multline*}
	F_B(\mu,c\Hn{L})\ge \int \Phi(x)\ d\mu(x) - c\int \Phi(x)\ d\Hn{L}(x) \ge C\mu(3B)r - \widetilde{C}cr^{n+1}\\
	\ge C\mu(3B)r - \frac{C}{2}\mu(3B)r = \frac{C}{2}\mu(3B)r.
	\end{multline*}
	Thus, we reach a contradiction with $F_B(\mu,\, c\Hn{L})\le\varepsilon\mu(3B)r.$
\end{proof}

\begin{lemma}\label{lem:plane from beta good for alpha}
	Suppose that $\mu$ is a Radon measure on $\R^d$, and that $B_1, B_2\subset\R^d$ are concentric balls satisfying $B_1\subset0.9B_2,\ \mu(B_1)\approx\mu(3B_2)\approx r(B_1)^n\approx r(B_2)^n$. Let $L_{\beta}$ be the $n$-plane minimizing $\beta_{\mu,2}(B_2)$, and $L_{\alpha},\ c>0,$ be the $n$-plane and constant minimizing $\alpha_{\mu}(B_2)$. Suppose further that $L_{\alpha}, L_{\beta}$ intersect $0.9B_1$. Then
	\begin{equation}\label{eq:Lbeta good for alpha}
	\frac{1}{\mu(B_1)r(B_1)}F_{B_1}(\mu, c\Hn{L_{\beta}})\lesssim \beta_{\mu,2}(B_2)+\alpha_{\mu}(B_2).
	\end{equation}
\end{lemma}
\begin{proof}
	Set $r=r(B_1)$. It follows easily by \eqref{eq:c smaller then mu over Hn} that $c\lesssim\mu(B_2)r(B_2)^{-n}\approx 1$, and so $F_{B_1}(\mu, c\Hn{L_{\beta}})\lesssim r\mu(B_1)$. Thus, without loss of generality, we may assume that $\beta_{\mu,2}(B_2)+\alpha_{\mu}(B_2)<\varepsilon$ for some small $\varepsilon>0$.
	
	By the triangle inequality, we have
	\begin{multline*}
	F_{B_1}(\mu, c\Hn{L_{\beta}}) \le F_{B_1}(\mu, c\Hn{L_{\alpha}}) +  F_{B_1}(c\Hn{L_{\alpha}}, c\Hn{L_{\beta}})\\
	\le F_{B_2}(\mu, c\Hn{L_{\alpha}}) +  F_{B_1}(c\Hn{L_{\alpha}}, c\Hn{L_{\beta}}),
	\end{multline*}
	and so our aim is to estimate the second term from the right hand side.
	
	Let $x_{\alpha}\in L_{\alpha}\cap \bar{B_1}$ and $x_{\beta}\in L_{\beta}$ be such that
	\begin{equation*}
	|x_{\alpha}-x_{\beta}| = \dist(x_{\alpha}, L_{\beta})=\inf_{x\in L_{\alpha}\cap B_1}\dist(x, L_{\beta}).
	\end{equation*}
	Without loss of generality we may assume that $x_{\alpha}=0$, so that $L_{\alpha}$ is a linear subspace. Denote $L'_{\beta} = L_{\beta}-x_{\beta}$. It follows by basic linear algebra that for $x\in L_{\alpha}\cap B_1$
	\begin{equation}\label{eq:1}
	\dist(x, L_{\beta})  = |x-\Pi_{L_{\beta}}(x)| = |x-x_{\beta}-\Pi_{L'_{\beta}}(x-x_{\beta})| = |\Pi^{\perp}_{L'_{\beta}}(x-x_{\beta})| = |\Pi^{\perp}_{L'_{\beta}}(x) - x_{\beta}|.
	\end{equation}
	Note that by the above and the triangle inequality
	\begin{equation*}
		\dist(x, L_{\beta}) = |\Pi^{\perp}_{L'_{\beta}}(x) - x_{\beta}|\le |x_{\beta}| + |\Pi^{\perp}_{L'_{\beta}}(x)|.
	\end{equation*}	
	On the other hand, by our choice of $x_{\beta}$, $|x_{\beta}|\le \dist(x, L_{\beta})$ for all $x\in L_{\alpha}\cap B_1$. Together with the triangle inequality and the identity \eqref{eq:1}  this gives
	\begin{equation*}
		|x_{\beta}| + |\Pi^{\perp}_{L'_{\beta}}(x)| \le 2|x_{\beta}| + |\Pi^{\perp}_{L'_{\beta}}(x) - x_{\beta}| \le 3\dist(x, L_{\beta}).
	\end{equation*}
	We put the two estimates above together to get
	\begin{equation}\label{eq:dist x to L beta}
		|x_{\beta}| + |\Pi^{\perp}_{L'_{\beta}}(x)|\approx \dist(x, L_{\beta}).
	\end{equation}

	Now, observe that, by the definition of $\measuredangle(L_{\alpha}, L_{\beta})$, for every $x\in L_{\alpha}$ we have $|\Pi^{\perp}_{L'_{\beta}}(x)| \le|x|\measuredangle(L_{\alpha},L_{\beta})$. Moreover, there exists a subspace $\ell\subset L_{\alpha}$ on which the equality is achieved, i.e. for all $x\in\ell$ we have $|\Pi^{\perp}_{L'_{\beta}}(x)| =|x|\measuredangle(L_{\alpha},L_{\beta})$.
	Consider a cone around $\ell$:
	\begin{equation*}
	K = \bigg\{ x\in\R^d\ :\ |\Pi_{\ell}(x)|\ge\frac{4}{5}|x| \bigg\}.
	\end{equation*}
	Since $0\in B_1\cap K\cap L_{\alpha}$, it is easy to see that $\mathcal{H}^n(B_1\cap K\cap L_{\alpha})\gtrsim r^n$, which in turn implies that for some small constant $0<\delta\ll 1$ (depending on the implicit constant in the previous inequality and dimension) we have 
	\begin{equation}\label{eq:a lot of K in B1}
	\mathcal{H}^n(B_1\cap K\cap L_{\alpha}\setminus B(0,\delta r))\gtrsim r^n.
	\end{equation}
	Moreover, for $x\in B_1\cap K\cap L_{\alpha}\setminus B(0,\delta r)$ we have
	\begin{multline*}
	|\Pi^{\perp}_{L'_{\beta}}(x)| = |\Pi^{\perp}_{L'_{\beta}}(\Pi_{\ell}(x))+\Pi^{\perp}_{L'_{\beta}}(\Pi_{\ell}^{\perp}(x))|\ge |\Pi^{\perp}_{L'_{\beta}}(\Pi_{\ell}(x))| - |\Pi^{\perp}_{L'_{\beta}}(\Pi_{\ell}^{\perp}(x))|\\
	\ge |\Pi_{\ell}(x)|\measuredangle(L_{\alpha},L_{\beta}) - |\Pi^{\perp}_{\ell}(x)|\measuredangle(L_{\alpha},L_{\beta})
	\overset{x\in K}{\ge}\frac{4}{5}|x|\measuredangle(L_{\alpha},L_{\beta}) - \frac{3}{5}|x|\measuredangle(L_{\alpha},L_{\beta})\\
	=\frac{1}{5} |x|\measuredangle(L_{\alpha},L_{\beta})\approx r \measuredangle(L_{\alpha},L_{\beta}).
	\end{multline*}	
	Hence, using the above, \eqref{eq:a lot of K in B1}, and \eqref{eq:dist x to L beta} yields	
	\begin{equation}\label{eq:dist of xalpha to xbeta}
	|x_{\beta}|r^{n-1} +  r^{n}\measuredangle(L_{\alpha},L_{\beta})\lesssim 
	\int_{B_1}\frac{\dist(x,L_{\beta})}{r}\ d\Hn{L_{\alpha}}(x)\overset{\eqref{eq:c greater then mu over Hn}}{\lesssim}c \int_{B_1}\frac{\dist(x,L_{\beta})}{r}\ d\Hn{L_{\alpha}}(x).
	\end{equation}
	
	Now, consider $\phi\in\lip_1(B_2)$ such that $\phi(x)\approx \dist(x,L_{\beta})$ in $B_1$, and $\phi(x)\lesssim \dist(x,L_{\beta})$ in $B_2$. Then,
	\begin{multline}\label{eq:dist of xalpha to xbeta 2}
	c\int_{B_1}\frac{\dist(x,L_{\beta})}{r}\ d\Hn{L_{\alpha}}(x)\lesssim c\int_{B_2}\frac{\phi(x)}{r}\ d\Hn{L_{\alpha}}(x)\\
	\lesssim  \int_{B_2}\frac{\phi(x)}{r}\ d\mu(x) + r^{-1}F_{B_2}(\mu,c\Hn{L_{\alpha}})
	\lesssim (\beta_{\mu,2}(B_2)+\alpha_{\mu}(B_2))\mu(3B_2).
	\end{multline}
	\eqref{eq:dist of xalpha to xbeta} and the calculation above give $\measuredangle(L_{\alpha},L_{\beta})\lesssim\beta_{\mu,2}(B_2)+\alpha_{\mu}(B_2)<\varepsilon$. Let $\Pi: L_{\alpha}\to \R^d$ be the orthogonal projection onto $L'_{\beta}$, and $i:L_{\alpha}\to{\R^d}$ an embedding. We have
	\begin{equation}\label{eq:Pi close to Id}
	\lVert \Pi - i\rVert_{op}=\lVert \Pi - i \rVert_{L^{\infty}(L_{\alpha}\cap B(0,1))}\lesssim \beta_{\mu,2}(B_2)+\alpha_{\mu}(B_2)<\varepsilon.
	\end{equation} 
	Thus, $\Pi$ is a linear isomorphism onto $L'_{\beta}$, with a bound on Jacobian
	\begin{equation}\label{eq:bound on jacobian}
	\left|1-|J\Pi|\right|\lesssim\beta_{\mu,2}(B_2)+\alpha_{\mu}(B_2).
	\end{equation}		
	It follows that for any $f\in\lip_1(B_1)$ we have
	\begin{multline*}
	\left|\int f(x)\ d\Hn{L_{\alpha}}(x) - \int f(y)\ d\Hn{L_{\beta}}(y)\right|\\ = \left|\int f(x)\ d\Hn{L_{\alpha}}(x) - \int f(x_{\beta}+\Pi(x))|J\Pi(x)|\ d\Hn{L_{\alpha}}(x)\right|\\
	\le \int |f(x) - f(x_{\beta}+\Pi(x))|\ d\Hn{L_{\alpha}}(x) + \int |f(x_{\beta}+\Pi(x))|\left|1-|J\Pi(x)|\right|\ d\Hn{L_{\alpha}}(x)\\
	\le \int_{B_1\cup \Pi^{-1}(B_1-x_{\beta})} |x_{\beta}| + |x-\Pi(x)|\ d\Hn{L_{\alpha}}(x)\\ + \int_{\Pi^{-1}(B_1-x_{\beta})} \lVert f\rVert_{L^{\infty}}\left|1-|J\Pi(x)|\right|\ d\Hn{L_{\alpha}}(x)\\
	\overset{\eqref{eq:Pi close to Id},\eqref{eq:bound on jacobian}}{\lesssim} |x_{\beta}|r^n + (\beta_{\mu,2}(B_2)+\alpha_{\mu}(B_2))r^{n+1}.
	\end{multline*}
	Taking supremum over all $f\in\lip_1(B_1)$, dividing by $r^{n+1}$, using \eqref{eq:dist of xalpha to xbeta}, \eqref{eq:dist of xalpha to xbeta 2}, the fact that $\mu(B_1)\approx r^n$, and that $c\lesssim 1$, yields the desired inequality:
	\begin{equation*}
	\frac{1}{\mu(B_1)r}F_{B_1}(c\Hn{L_{\alpha}}, c\Hn{L_{\beta}})\lesssim \beta_{\mu,2}(B_2)+\alpha_{\mu}(B_2).
	\end{equation*}
\end{proof}
%

\section{The David-Mattila cubes}\label{sec:DM cubes}
In the proof of \thmref{thm:sufficient condition} we will use the lattice of ``dyadic cubes'' constructed by David and Mattila \cite{david2000removable}.
\begin{lemma}[{\cite[Theorem 3.2, Lemma 5.28]{david2000removable}}]\label{lem:DM lattice}
	Let $\mu$ be a Radon measure on $\R^d$, $E=\supp\mu.$ For any constants $C_0>1$, $A_0>5000C_0$ there exists a sequence of partitions of $E$ into Borel subsets $Q,\ Q\in\D_k$, with the following properties:
	\begin{itemize}
		\item[(a)] For each integer $k\ge 0$, $E$ is the disjoint union of the ``cubes'' $Q$, $Q\in \D_k$, and if $k<l$, $Q\in\D_l$, and $R\in\D_k$, then either $Q\cap R = \varnothing$ or else $R\subset Q$.
		\item[(b)] The general position of the cubes $Q$ can be described as follows. For each $k\ge0$ and each cube $Q\in\D_k$, there is a ball $B(Q)=B(z_Q, r(Q)),$ such that 
		\begin{gather*}
		z_Q\in Q,\quad A_0^{-k}\le r(Q)\le C_0 A_0^{-k},\\
		E\cap B(Q)\subset Q\subset E\cap 28 B(Q) = E\cap B(z_Q,28 r(Q)),
		\end{gather*}
		and the balls $5B(Q), Q\in \D_k,$ are disjoint.
	\end{itemize}		
\end{lemma}
\begin{remark}
	The cubes of David and Mattila have many other useful properties, most notably the so-called \emph{small boundaries}. We will not need them, however.
\end{remark}

For any $Q\in\D:= \bigcup_{k\ge 0}\D_k$ we denote by $\D(Q)$ the family of $P\in\D$ such that $P\subset Q$. Given $Q\in \D_k$ we set $J(Q) = k$ and $\ell(Q)=56 C_0 A_0^{-k}$. Note that $r(Q)\approx\ell(Q).$
We define also $B_Q = 28B(Q) = B(z_Q, 28\, r(Q)),$ so that
\begin{equation*}
E\cap \tfrac{1}{28}B_Q\subset Q\subset B_Q.
\end{equation*} 
Denote by $\D^{db}$ the family of doubling cubes, i.e. $Q\in\D$ satisfying
\begin{equation}\label{eq:doubling_balls}
\mu(100 B(Q))\le C_0 \mu(B(Q)).
\end{equation}

If the constants $C_0,\, A_0$ in \lemref{lem:DM lattice} are chosen of the form $A_0=C(C_0)^K,\ K\ge 3n+5$, and $C_0=C_0(K,n)$ is large enough, then the following lemma holds.
\begin{lemma}[{\cite[Lemma 5.31]{david2000removable}}]
	 Let $R\in\D$ and $Q\subset R$ be cubes such that all the intermediate cubes $S,\ Q\subsetneq S\subsetneq R,$ are non-doubling, i.e. $S\in\D\setminus\D^{db}$. Then,
	\begin{equation}\label{eq:density drops for nondoubling cubes}
	\mu(100B(Q))\le A_0^{-2d(J(Q)-J(R)-1)}\mu(100B(R)).
	\end{equation}
\end{lemma}
\begin{remark}
	The constant $2d$ in \eqref{eq:density drops for nondoubling cubes} can be replaced by any positive constant if $C_0,\, A_0$ are chosen suitably. See \cite[(5.30)]{david2000removable} for details.
\end{remark}

It will be convenient for us to work with cubes satisfying a doubling condition stronger than \eqref{eq:doubling_balls}. To introduce them we need a version of \cite[Lemma 2.8]{tolsa2014analytic}. For reader's convenience, we provide the proof below.
\begin{lemma}\label{lem:lemma 2.8 from the green book}
	Let $\mu$ be a Radon measure on $\R^d$ and $\alpha>1$ be some constant. Then, for $\mu$-a.e. $x\in\R^d$ there exists a sequence $r_j\to 0$ such that for every $j$ we have
	\begin{equation}\label{eq:2}
	\mu(B(x,\alpha\, r_j))\le 2\,\alpha^d\mu(B(x,r_j)).
	\end{equation}
\end{lemma}
\begin{proof}
	Consider the set $Z\subset\supp\mu$ of points such that for $x\in Z$ there does not exist a sequence of radii $r_j\to 0$ satisfying \eqref{eq:2}. We want to show that $\mu(Z)=0$. Let
	\begin{equation*}
		Z_j = \{x\in\supp\mu\, :\, \mu(B(x,\alpha\, r))> 2\,\alpha^d\mu(B(x,r))\ \text{for all}\ r\le 2^{-j} \}.
	\end{equation*}
	Clearly $Z=\bigcup_j Z_j$, and so it suffices to prove $\mu(Z_j)=0$ for all $j\ge 0$.
	
	Let $B_0$ be an arbitrary ball of radius $2^{-j}$ centered at $Z_j$, and choose some integer $k\ge d$. For each $x\in Z_j\cap B_0$ we set $B_x = B(x,\alpha^{-k} 2^{-j})$. Observe that, by the definition of $Z_j$, for $h=0,\dots, k-1$ we have
	\begin{equation*}
		\mu(\alpha^{h+1} B_x)> 2\,\alpha^d\mu(\alpha^{h}B_x).
	\end{equation*}
	Thus,
	\begin{multline}\label{eq:3}
		\mu(B_x)< (2\,\alpha^d)^{-1}\mu(\alpha B_x)<\dots< (2\,\alpha^d)^{-k}\mu(\alpha^{k} B_x) = (2\,\alpha^d)^{-k} \mu(B(x,2^{-j}))\\
		\le (2\,\alpha^d)^{-k}\mu(2B_0).
	\end{multline}

	Now, we use Besicovitch covering theorem to choose points $\{x_m\}\subset Z_j\cap B_0$ such that $\bigcup_m B_{x_m}$ covers $Z_j\cap B_0$, and moreover $\sum_m \one_{B_{x_m}}\le C_d$. The bounded intersection property implies that $N:=\#\{x_m\}<\infty$, and more precisely
	\begin{equation*}
		N\omega_d(\alpha^{-k} 2^{-j})^d = \sum_{m} \mathcal{H}^d(B_{x_m})\le C_d \mathcal{H}^d(2B_0) = C_d2^d\omega_d2^{-jd},
	\end{equation*}
	where $\omega_d$ stands for the volume of a $d$-dimensional ball. Hence,
	\begin{equation*}
		N\le C_d2^d\alpha^{kd}.
	\end{equation*}
	Consequently, we may use \eqref{eq:3} and the fact that $\bigcup_m B_{x_m}\supset Z_j\cap B_0$ to obtain
	\begin{equation*}
		\mu(Z_j\cap B_0)\le \sum_m \mu(B_{x_m}) \le N (2\,\alpha^d)^{-k}\mu(2B_0)\le C_d2^{d-k}\mu(2B_0).
	\end{equation*}
	Since $k$ can be chosen as large as we wish, this gives $\mu(Z_j\cap B_0)=0$. But $B_0$ was an arbitrary ball, and so $\mu(Z_j)=0$.
\end{proof}

We may use the lemma above to show the following.

\begin{lemma}\label{lem:sequence of dcrsing db cubes}
	There exists a constant $C=C(d,C_0,A_0)$ such that for $\mu$-a.e. $x\in\R^d$ there exists a sequence of cubes $Q_j\in\D^{db}$ satisfying $x\in Q_j$, $\ell(Q_j)\to 0$, and
	\begin{equation}\label{eq:very doubling cubes}
	\mu(100B_{Q_j})\le C\, \mu(B(Q_j)).
	\end{equation}
\end{lemma}
\begin{proof}
	Let $\alpha=2\,C_0^2\,A_0^{k+1}$, where $k$ is a constant that will be fixed later on. Consider a sequence of balls $B(x,r_j)$ given by \lemref{lem:lemma 2.8 from the green book}. Fix some $j$. Let $Q$ be the smallest cube satisfying $x\in Q$ and $B(x,r_j)\subset 100B(Q).$ We have
	\begin{equation*}
		72\,A_0^{-1}C_0^{-1}\,r(Q)\le r_j\le 100\,r(Q).
	\end{equation*}
	It is easy to check that, with the choice of $\alpha$ we made at the beginning, we have
	\begin{equation*}
	B(x,\alpha\,r_j)\supset 100B(R),
	\end{equation*}
	where $R$ is the $k$-th ancestor of $Q$, i.e. $Q\subset R$ and $J(Q)-J(R)=k$.
	
	Now, if all the intermediate cubes $S,\ Q\subsetneq S\subsetneq R,$ were non-doubling, then by \eqref{eq:density drops for nondoubling cubes} and \lemref{lem:lemma 2.8 from the green book} we would have
	\begin{multline*}
	\mu(B(x,r_j))\le \mu(100B(Q))\overset{\eqref{eq:density drops for nondoubling cubes}}{\le} A_0^{-2d(k-1)}\mu(100B(R))\le A_0^{-2d(k-1)}\mu(B(x,\alpha\, r_j))\\
	\le	 A_0^{-2d(k-1)} 2\, (2\,C_0^2\,A_0^{k+1})^d\mu(B(x,r_j)) = 2^{d+1}C_0^{2d}A_0^{-dk + 3d}\mu(B(x,r_j)).
	\end{multline*}
	For $k=k(d,C_0,A_0)$ big enough the constant on the right hand side is smaller than $1$, and so we reach a contradiction. It follows that one of the intermediate cubes $S$ is doubling. Thus,
	\begin{multline*}
	\mu(100B_S)\le\mu(100B(R))\le\mu(B(x,\alpha\, r_j))\le 2\,\alpha^d\mu(B(x,r_j))\\
	\le 2\,\alpha^d\mu(100B(Q))\le 2\,\alpha^d\mu(100B(S))\le 2\, C_0\, \alpha^d\mu(B(S)).
	\end{multline*}
	Setting $Q_j = S$ finishes the proof.
\end{proof}
We will call the cubes satisfying \eqref{eq:very doubling cubes} \emph{strongly doubling}, and the family of all such cubes will be denoted by $\D^{sdb}$. We fix constants $C_0$ and $A_0$ so that all of the above holds, and from now on we will treat them as absolute constants. We will not mention dependence on them in our estimates.

Finally, for each $Q\in\D$ we define $L_Q$ to be the $n$-plane minimizing $\beta_{\mu,2}(3B_Q)$, and $c_Q\ge 0$ to be the constant minimizing $\alpha_{\mu}(3B_Q)$.

\section{Main Lemma}\label{sec:main lemma}
Given $\varepsilon>0$ and $r>0$ let us define the set of ``good points'':
\begin{equation}\label{eq:def of G}
G^{\varepsilon}_{r} = \left\{x\in\supp\mu\ :\ \int_0^{1000r} \left(\alpha_{\mu}(x,s)^2 + \beta_{\mu,2}(x,s)^2 \right)\ \frac{ds}{s}<\varepsilon^2 \right\}.
\end{equation}
Using this notation we may formulate our main lemma.
\begin{lemma}\label{lem:main lemma}
	Let $\mu$ be a finite Radon measure on $\R^d$. There exists a small dimensional constant $\varepsilon_0>0$ such that the following holds: suppose that $R_0\in\D^{sdb}$ satisfies
	\begin{equation}\label{eq:set Gr is small}
	\mu\left(R_0\setminus G^{\varepsilon_0}_{r(R_0)}\right)\le \varepsilon_0\,\mu(3B_{R_0}).
	\end{equation}
	Then, there exists a set $R_G\subset R_0$, and a Lipschitz map $F:L_{R_0}\rightarrow L_{R_0}^{\perp} $ (recall that $L_{R_0}$ denotes the $n$-dimensional plane minimizing $\beta_{\mu,2}(3B_{R_0})$), such that for 
	\begin{equation*}
	\Gamma = \big\{(x,F(x))\ :\ x\in L_{R_0} \big\}
	\end{equation*}
	we have $R_G\subset \Gamma$,
	\begin{equation}\label{eq:RG bigger than half of R0}
	\mu(R_G)\ge \frac{\mu(R_0)}{2},
	\end{equation}
	and $\restr{\mu}{R_G}$ is absolutely continuous with respect to $\mathcal{H}^n$.
\end{lemma}
Several remarks are in order.
\begin{remark}
	Assumption \eqref{eq:set Gr is small} is implied by a somewhat more natural condition
	\begin{equation*}
	\int_{R_0}\int_0^{1000r(R_0)} \left( \alpha_{\mu}(x,s)^2 + \beta_{\mu,2}(x,s)^2\right)\ \frac{ds}{s}d\mu(x)<\varepsilon_0^3\,\mu(3B_{R_0}).
	\end{equation*}
\end{remark}
\begin{remark}
	The constant $\frac{1}{2}$ in \eqref{eq:RG bigger than half of R0} can be replaced by any $\delta\in (0,1)$, as long as we allow $\varepsilon_0$ to depend on $\delta$. Naturally, $\varepsilon_0(\delta)\to 0$ as $\delta\to 1$.
\end{remark}
\begin{remark}\label{rem:homogeneous version of main lemma}
	Recall that we defined homogeneous $\beta$ numbers $\beta_{\mu,2}^h$  in \eqref{eq:homogeneous betas}. One could similarly define $\alpha_{\mu}^h(x,r)=\frac{\mu(B(x,3r))}{r^n}\alpha_{\mu}(x,r).$ Careful inspection of the proof of \lemref{lem:main lemma} (see \remref{rem:homogeneous reduction}) shows the following. If instead of \eqref{eq:def of G} we define for $Q\in\D$
	\begin{multline*}
	G^{\varepsilon}_{Q} = \bigg\{x\in Q\ :\ \int_0^{1000r(Q)} \alpha_{\mu}^h(x,s)^2\  \frac{ds}{s}<\varepsilon^2\,\Theta_{\mu}(3B_Q)^2\quad\text{and}\\
	\int_0^{1000r(Q)} \beta_{\mu,2}^h(x,s)^2\ \frac{ds}{s}<\varepsilon^2\,\Theta_{\mu}(3B_Q) \bigg\},
	\end{multline*}
	and we replace the assumption \eqref{eq:set Gr is small} by $\mu\left(R_0\setminus G^{\varepsilon_0}_{R_0}\right)\le \varepsilon_0\,\mu(3B_{R_0}),$ then the conclusion of \lemref{lem:main lemma} still holds. In other words, if the homogeneous square functions in some initial cube $R_0$ are small \emph{relative to density of $\mu$ in the initial cube}, then $\mu$ is rectifiable on a large chunk of $R_0$.
\end{remark}
Let us show how \lemref{lem:main lemma} may be used to prove \thmref{thm:sufficient condition}.
\begin{proof}[Proof of \thmref{thm:sufficient condition} using \lemref{lem:main lemma}]
	To show that $\mu$ is $n$-rectifiable it suffices to show that for any $E\subset\R^d$ satisfying $\mu(E)>0$ there exists $F\subset E$ with $\mu(F)>0$ and such that $\restr{\mu}{F}$ is rectifiable. Let us fix $E\subset\R^d$ with $\mu(E)>0$.
	
	Let $\varepsilon_0>0$ be so small that \lemref{lem:main lemma} holds. Note that by the assumption on the finiteness of $\alpha$ and $\beta$ square functions \eqref{eq:main theorem assumption on alphas}, \eqref{eq:main theorem assumption on betas}, we have 
	\begin{equation*}
	\mu(\R^d\setminus G^{\varepsilon_0}_r)\xrightarrow{r\rightarrow 0} 0.
	\end{equation*}
	In particular, $\mu$-almost all of $E$ is contained in $\bigcup_{r>0}G^{\varepsilon_0}_r$. By Lebesgue differentiation theorem, for $\mu$-almost every $x\in E\cap G^{\varepsilon_0}_r$ 
	\begin{equation*}
	\frac{\mu(B(x,s)\cap E\cap G^{\varepsilon_0}_r)}{\mu(B(x,s))}\xrightarrow{s\rightarrow 0} 1.
	\end{equation*}
	Taking into account that for $s<r$ we have $G^{\varepsilon_0}_s\supset G^{\varepsilon_0}_r$, it follows that for $\mu$-almost every $x\in E$
	\begin{equation*}
	\frac{\mu(B(x,r)\cap E\cap G^{\varepsilon_0}_r)}{\mu(B(x,r))}\xrightarrow{r\rightarrow 0} 1.
	\end{equation*}
	
	Choose some $x\in E$ such that the above and \lemref{lem:sequence of dcrsing db cubes} hold. Let $r_0>0$ be so small that $\mu(B(x,r)\cap E\cap G^{\varepsilon_0}_{r})>(1-\varepsilon_0)\mu(B(x,r))$ for all $r<r_0$.
	
	Using \lemref{lem:sequence of dcrsing db cubes} we may choose $R_0\in\D^{sdb}$ such that $x\in R_0$ and $\tilde{r}:=2r(B_{R_0})<r_0$. We have $R_0\subset B(x,\tilde{r})\subset 3B_{R_0}$, and so
	\begin{equation*} 
	\mu(R_0\setminus G^{\varepsilon_0}_{r(R_0)})\le\mu(R_0\setminus G^{\varepsilon_0}_{\tilde{r}})\le \mu(B(x,\tilde{r})\setminus G^{\varepsilon_0}_{\tilde{r}})\le \varepsilon_0 \mu(B(x,\tilde{r}))\le \varepsilon_0\mu(3B_{R_0}).
	\end{equation*}		
	Hence, $R_0$ satisfies the assumptions of \lemref{lem:main lemma}. We obtain a Lipschitz graph $\Gamma$ and a set $R_G\subset R_0\cap\Gamma$ such that $\mu(R_G)\ge 0.5\mu(R_0),$ and $\restr{\mu}{R_G}$ is absolutely continuous with respect to $\mathcal{H}^n$. On the other hand, arguing as above, and using the fact that $R_0$ is doubling, we see that $\mu(R_0\setminus E)\le\varepsilon_0\mu(3B_{R_0})\le C_0\varepsilon_0\mu(R_0)<0.5\mu(R_0)$, assuming $\varepsilon_0<0.5\, C_0^{-1}.$
	
	It follows that $\mu(R_G\cap E)\ge\mu(R_G) -\mu(R_0\setminus E)>0,$ and $\restr{\mu}{R_G\cap E}$ is $n$-rectifiable. Setting $F = R_G\cap E$ concludes the proof.

\end{proof}
The rest of the paper is dedicated to proving \lemref{lem:main lemma}. We fix $R_0\in\D^{sdb}$ satisfying \eqref{eq:set Gr is small}. The constant $\varepsilon_0$ will be chosen later on. 
To simplify notation, we set $G=G^{\varepsilon_0}_{r(R_0)},\ B_{0}=B_{R_0},\ r_0=r(B_0),\ z_0=z_{R_0},\ c_0=c_{R_0}$ (where $c_{R_0}$ is a constant minimizing $\alpha_{\mu}(3B_0)$), $L_0 = L_{R_0},$  (where $L_{R_0}$ is an $n$-plane minimizing $\beta_{\mu,2}(3B_0)$), and $\Pi_0=\Pi_{L_0}$.

\begin{remark}\label{rem:WLOG density of BR0 is 1}
	Without loss of generality we may (and will) assume that
	\begin{equation*}
	\Theta_{\mu}(3B_0)=1,
	\end{equation*}
	so that (using the strong doubling property of $R_0$ \eqref{eq:very doubling cubes})
	\begin{equation}
	\mu(100B_0)\approx\mu(R_0)\approx r_0^n\approx \ell(R_0)^n.
	\end{equation}
	Indeed, if we consider the normalized measure $\nu = \mu/\Theta_{\mu}(3B_0)$, then: $\Theta_{\nu}(3B_0)=1$; for any ball $B$ with $\mu(B)>0$ we have $\alpha_{\mu}(B) = \alpha_{\nu}(B),\ \beta_{\mu,2}(B) = \beta_{\nu,2}(B)$; and if the assumptions of \lemref{lem:main lemma} were satisfied for $\mu$, then they are also satisfied  for $\nu$. Sets $\Gamma$ and $R_G$ constructed for $\nu$ will also have all the desired properties when applied to $\mu$.
\end{remark}
\begin{remark}\label{rem:homogeneous reduction}
	The reduction to case $\Theta_{\mu}(3B_0)=1$ performed above is one of the main reasons why we decided to work with non-homogeneous (i.e. normalized by $\mu(3B)$) $\alpha$ and $\beta$ coefficients. If we assumed \textit{a priori} that $\Theta_{\mu}(3B_0)=1$, then we could replace $\alpha_{\mu}$ and $\beta_{\mu,2}$ numbers in \eqref{eq:def of G} by $\alpha^h_{\mu}$ and $\beta^h_{\mu,2}$, and then carry on with the proof without making \emph{any} changes. Roughly speaking, throughout most of the proof we work with cubes $Q$ satisfying $\mu(3B_Q)\approx \ell(Q)^n\Theta_{\mu}(3B_0), $ so that $\alpha_{\mu}^h(3B_Q)\approx \alpha_{\mu}(3B_Q)\Theta_{\mu}(3B_0)$ and $\beta_{\mu,2}^h(3B_Q)\approx \beta_{\mu,2}(3B_Q)\Theta_{\mu}(3B_0)^{1/2}$ -- see \remref{rem:homogeneous comparable on tree}.
	
	Now, the claim we made in \remref{rem:homogeneous version of main lemma} follows because the modified assumption (involving $G^{\varepsilon}_Q$) allows us to make the reduction $\Theta_{\mu}(3B_0)=1$.
\end{remark}
%
%

\section{Stopping cubes}\label{sec:stopping cubes}
This section is dedicated to performing the stopping time argument. We will show basic properties of the resulting tree of cubes, and estimate the size of two families of stopping cubes.

The stopping conditions involve parameters $A\gg 1,\ \tau\ll 1,\ \theta\ll 1$, which depend on dimension and which will be fixed later on. The constant $\varepsilon_0$ is fixed at the very end of the proof, and depends on $A,\ \tau,\ \theta$.

We define the following subfamilies of $\D(R_0)$:
\begin{itemize}
	\item $\HD_0$ (``high density''), which contains cubes $Q\in \D(R_0)$ satisfying 
	\begin{equation*}
	\mu(3B_Q)>A\ell(Q)^n,
	\end{equation*}
	\item $\LD_0$ (``low density''), which contains cubes $Q\in \D(R_0)$ satisfying 
	\begin{equation*}
	\mu(1.5B_Q)<\tau\ell(Q)^n,
	\end{equation*}
	\item $\BS_0$ (``big square functions''), which contains cubes $Q\in \D(R_0)\setminus (\LD_0\cup\HD_0)$ satisfying 
	\begin{equation}\label{eq:BA}
	\mu(Q\setminus G) > \frac{1}{2} \mu(Q).
	\end{equation}
\end{itemize}
Let $\Stop_0$ be the family of maximal (and thus disjoint) cubes from $\HD_0\cup\LD_0\cup\BS_0$, and let $\Tree_0\subset\D(R_0)$ be the family of cubes that are not contained in any $Q\in\Stop_0$. In particular, $\Stop_0\not\subset\Tree_0$. 

Recall that $L_Q$ is an $n$-plane minimizing $\beta_{\mu,2}(3B_Q)$. We define
\begin{equation*}
\RFar = \{x\in 3B_0\ :\ \dist(x, L_Q)\ge\, \sqrt{\varepsilon_0}\ell(Q)\quad \text{for some $Q\in\Tree_0$ s.t. $x\in 3B_Q$} \}.
\end{equation*}
We introduce two more families of stopping cubes:
\begin{itemize}
	\item $\BA_0$ (``big angles''), which contains cubes $Q\in \D(R_0)\setminus \Stop_0$ satisfying 
	\begin{equation}\label{eq:def of BS}
	\measuredangle(L_Q,L_0)>\theta,
	\end{equation}
	\item $\F_0$ (``far''), which consists of $Q\in \D(R_0)\setminus (\Stop_0\cup\BA_0)$ satisfying 
	\begin{equation}\label{eq:def of F}
	\mu(3B_Q\cap \RFar)>\varepsilon_0^{1/4}\mu(3B_Q).
	\end{equation}
\end{itemize}
Let $\Stop\subset\D(R_0)$ be the family of maximal (and thus disjoint) cubes from $\Stop_0\cup \BA_0\cup\F_0$. Set $\HD = \HD_0\cap\Stop,\ \LD = \LD_0\cap\Stop,\ \BS = \BS_0\cap\Stop,\ \BA = \BA_0\cap\Stop,\ \F = \F_0\cap\Stop.$ 
We define $\Tree\subset\Tree_0$ as the family of cubes that are not contained in any $Q\in\Stop$. Note that $\Stop\not\subset\Tree$.
For $P\in\D$ we set $\Tree_0(P)=\Tree_0\cap\D(P),\ \Tree(P)=\Tree\cap\D(P)$. 
\subsection{Properties of cubes in \texorpdfstring{$\Tree$}{Tree}}

\begin{lemma}
	The following estimates hold:
	\begin{align}
	&\mu(1.5B_Q)\ge \tau\ell(Q)^n   &\forall\ Q\in\Tree_0\cup\Stop_0\setminus\LD_0,\label{eq:notLD}\\
	&{\mu}(100B_Q)\lesssim A\ell(Q)^n &\forall\ Q\in\Tree_0\cup\Stop_0,\label{eq:notHD}\\
	&\mu(Q\setminus G) \le \frac{1}{2} \mu(Q)& \forall\ Q\in\Tree_0,\label{eq:notBA}\\
	&\measuredangle(L_Q,L_0)\le \theta& \forall\ Q\in\Tree,\label{eq:notBS}\\
	&\mu(3B_Q\cap \RFar)\le\varepsilon_0^{1/4}\, \mu(3B_Q) & \forall\ Q\in\Tree\label{eq:notF}.
	\end{align}
\end{lemma}
\begin{proof}
	All estimates except for \eqref{eq:notHD} follow immediately from the stopping time conditions. \eqref{eq:notHD} holds for $R_0$ because $R_0\in\D^{sdb}$. To see it for $Q\in\Tree_0\cup\Stop_0,\ Q\neq R_0$, note that the parent of $Q$, denoted by $R$, satisfies $R\in\Tree_0$, and so $\mu(100B_Q)\le\mu(3B_R)\le A\ell(R)^n\approx A\ell(Q)^n$.
\end{proof}
\begin{remark}\label{rem:homogeneous comparable on tree}
	Note that, by \eqref{eq:notLD} and \eqref{eq:notHD}, for $Q\in\Tree_0\cup\Stop_0\setminus\LD_0$ we have $\beta_{\mu,2}(3B_Q)\approx_{A,\tau} \beta^h_{\mu,2}(3B_Q)$ and $\alpha_{\mu}(3B_Q)\approx_{A,\tau} \alpha^h_{\mu}(3B_Q)$.
\end{remark}


\begin{lemma}\label{lem:beta and alpha estimates}
	Let $R\in\Tree_0$. Then
	\begin{align}
	\sum_{Q\in\Tree_0(R)}\alpha_{\mu}(3B_Q)^2\ell(Q)^n&\lesssim_{A,\tau}\varepsilon_0^2\ell(R)^n,\label{eq:alpha sum estimate}\\
	\sum_{Q\in\Tree_0(R)}\beta_{\mu,2}(3B_Q)^2\ell(Q)^n&\lesssim_{A,\tau}\varepsilon_0^2\ell(R)^n.\label{eq:beta sum estimate}
	\end{align}
	Moreover, for any $x\in 3B_0$
	\begin{align}
	\sum_{\substack{Q\in\Tree_0\\ x\in 3B_Q}}\alpha_{\mu}(3B_Q)^2&\lesssim_{A,\tau}\varepsilon_0^2,\label{eq:sum of alphas containing x estimate}\\
	\sum_{\substack{Q\in\Tree_0\\ x\in 3B_Q}}\beta_{\mu,2}(3B_Q)^2&\lesssim_{A,\tau}\varepsilon_0^2.\label{eq:sum of betas containing x estimate}
	\end{align}
\end{lemma}
\begin{proof}
%
	Let $Q\in \Tree_0(R)$. By the definition of $G$, for any $z\in 4B_Q\cap G$ we have
	\begin{equation}\label{eq:def of G revisited}
	\int_0^{1000r(R_0)} \alpha_{\mu}(z,r)^2\ \frac{dr}{r}<\varepsilon_0^2.
	\end{equation}
	It is easy to see that 
	for $300 r(Q)\le r\le 400 r(Q)$ we have $3B_Q\subset B(z,r)\subset 25B_Q$, and that $\mu(9B_Q)\approx_{A,\tau}\mu(B(z,3r))\approx_{A,\tau} \mu(100B_Q)$. Using \eqref{eq:alpha scaling estimate} with $B_1 = 3B_Q$ and $B_2=B(z,r)$ yields
	\begin{equation*}
	\alpha_{\mu}(3B_Q)\lesssim_{A,\tau} \alpha_{\mu}(B(z,r)).
	\end{equation*}
	Integrating with respect to $r$ gives us for every $z\in 4B_Q\cap G$
	\begin{equation}\label{eq:estimate alpha of a ball by a bigger ball}
	\int_{300r(Q)}^{400r(Q)} \alpha_{\mu}(z,r)^2\ \frac{dr}{r}\gtrsim_{A,\tau}\alpha_{\mu}(3B_Q)^2.
	\end{equation}		
	To see \eqref{eq:sum of alphas containing x estimate}, let $x\in 3B_0$ and choose some $P\in \Tree_0$ satisfying $x\in 3B_P$. By \eqref{eq:notBA} we may pick $z\in P\cap G$. It is clear that for all cubes $Q\in \Tree_0$ such that $\ell(Q)>\ell(P)$ and $x\in 3B_Q$ we have $z\in 4B_Q\cap G$. Thus, summing \eqref{eq:estimate alpha of a ball by a bigger ball} over all such $Q\subset R_0$, and noticing that for any fixed sidelength $\ell(Q_0)>\ell(P)$ there are only boundedly many $Q$ with $\ell(Q)=\ell(Q_0)$ and $3B_Q\ni x$, yields
	\begin{equation*}
	\sum_{\substack{Q\in \Tree_0\\x\in 3B_Q,\ \ell(Q)>\ell(P)}} \alpha_{\mu}(3B_Q)^2
	\lesssim_{A,\tau} \int_0^{1000r(R_0)}\alpha_{\mu,2}(z,r)^2\ \frac{dr}{r}\lesssim \varepsilon_0^2.
	\end{equation*}
	Since the estimate holds for arbitrary $P\in \Tree_0$  with  $x\in 3B_P$, \eqref{eq:sum of alphas containing x estimate} follows.
	
	To see \eqref{eq:alpha sum estimate}, we integrate \eqref{eq:sum of alphas containing x estimate} over $x\in 3B_R$ to get
	\begin{multline*}
	\varepsilon_0^2\ell(R)^n\gtrsim_{A,\tau} \int_{3B_R}\sum_{Q\in \Tree_0} \alpha_{\mu}(3B_Q)^2\, \one_{3B_{Q}}(x)\ d\mu(x)\\
	= \sum_{Q\in \Tree_0} \alpha_{\mu}(3B_Q)^2 \mu(3B_Q\cap 3B_{R})\gtrsim_{A,\tau} \sum_{Q\in \Tree_0(R)} \alpha_{\mu}(3B_Q)^2\ell(Q)^n.
	\end{multline*}
	%
	
	The estimates for $\beta_{\mu,2}(3B_Q)$ can be shown in the same way.
\end{proof}

\begin{cor}
	We have
	\begin{align}
	\sum_{Q\in\Tree_0(R)}F_{2.5B_Q}(\mu, c_Q\Hn{L_Q})^2\ell(Q)^{-(n+2)}&\lesssim_{A,\tau}\varepsilon_0^2\ell(R)^n,\label{eq:F sum estimate}\\
	\sum_{\substack{Q\in\Tree_0\\ x\in 3B_Q}}F_{2.5B_Q}(\mu, c_Q\Hn{L_Q})^2\ell(Q)^{-(2n+2)}&\lesssim_{A,\tau}\varepsilon_0^2.\label{eq:sum of F containing x estimate}
	\end{align}
\end{cor}
\begin{proof}
	Let $Q\in \Tree_0$. Recall that by \eqref{eq:notLD}, \eqref{eq:notHD}, we have $\mu(2.5B_Q)\approx_{A,\tau}\mu(3B_Q)\approx_{A,\tau} \ell(Q)^n.$ Moreover, it follows easily by \eqref{eq:notLD} and the smallness of $\alpha$ and $\beta$ numbers \eqref{eq:sum of alphas containing x estimate},\eqref{eq:sum of betas containing x estimate}, that the best approximating planes for $\beta_{\mu,2}(3B_Q)$ and $\alpha_{\mu}(3B_Q)$ intersect $2B_Q$. 
	
	Hence, by \lemref{lem:plane from beta good for alpha} applied to $B_1=2.5B_Q$ and $B_2=3B_Q$, and by \lemref{lem:beta and alpha estimates}, we get the desired estimates.
\end{proof}
\begin{cor}
	For every $Q\in\Tree_0$
	\begin{equation}\label{eq:cBQ estimate}
	c_{Q}\approx_{A,\tau}1.
	\end{equation}
\end{cor}
\begin{proof}
	By \eqref{eq:notLD}, \eqref{eq:notHD}, we have $\mu(1.5B_Q)\approx_{A,\tau}\mu(9B_Q)\approx_{A,\tau} \ell(Q)^n.$ Together with the smallness of $\alpha_{\mu}(3B_Q)$ \eqref{eq:sum of alphas containing x estimate}, this implies that the best approximating plane for $\alpha_{\mu}(3B_Q)$ intersects $2B_Q$. Thus, \lemref{lem:c comparable to density} yields
	\begin{equation*}
	c_Q\approx_{A,\tau} 1.
	\end{equation*}
\end{proof}

\begin{lemma}\label{lem:small measure of RFar}
	We have
	\begin{equation}\label{eq:small measure of RFar}
	\mu(\RFar)\lesssim_{A,\tau} \sqrt{\varepsilon_0}\mu(R_0)^n.
	\end{equation}
\end{lemma}
\begin{proof}
	We begin by using Chebyshev and Cauchy-Schwarz inequalities to obtain
	\begin{multline*}
	\sqrt{\varepsilon_0}\mu(\RFar)\le \int_{3B_0} \Bigg(\sum_{\substack{Q\in\Tree_0\\ x\in 3B_Q}}\bigg(\frac{\dist(x,L_Q)}{\ell(Q)}\bigg)^2\Bigg)^{1/2}\ d\mu(x)\\
	\le \Bigg(\int_{3B_0} \sum_{\substack{Q\in\Tree_0\\ x\in 3B_Q}}\bigg(\frac{\dist(x,L_Q)}{\ell(Q)}\bigg)^2\ d\mu(x)\Bigg)^{1/2}\mu(3B_0)^{1/2}.
	\end{multline*}
	By Fubini, the right hand side is equal to
	\begin{multline*}
	\Bigg(\sum_{Q\in\Tree_0} \int_{3B_Q} \bigg(\frac{\dist(x,L_Q)}{\ell(Q)}\bigg)^2\ d\mu(x)\Bigg)^{1/2}\mu(3B_0)^{1/2}\\
	\lesssim_{A,\tau} \bigg(\sum_{Q\in\Tree_0}\beta_{\mu,2}(3B_Q)^2\ell(Q)^n\bigg)^{1/2}\mu(R_0)^{n/2}.
	\end{multline*}
	We can estimate this using the smallness of $\beta$-numbers \eqref{eq:beta sum estimate}, and thus
	\begin{equation*}
	\sqrt{\varepsilon_0}\mu(\RFar)\lesssim_{A,\tau}\varepsilon_0\mu(R_0)^n.
	\end{equation*}
\end{proof}


\subsection{Balanced balls}
\begin{lemma}[{\cite[Lemma 3.1, Remark 3.2]{azzam2015characterization}}]\label{lem:balanced balls}
	Let $\mu$ be a Radon measure on $\R^d$, and let $B\subset\R^d$ be some ball with radius $r>0$ such that $\mu(B)>0$. Let $0<\gamma<1$. Then there exist constants $\rho_1=\rho_1(\gamma)>0$ and $\rho_2=\rho_2(\gamma)>0$ such that one of the following alternatives holds:
	\begin{itemize}
		\item[(a)] There are points $x_0, \dots, x_n\in B$ such that
		\begin{equation*}
		\mu(B(x_k,\rho_1 r)\cap B)\ge \rho_2\mu(B)\quad \text{for $0\le k\le n,$}
		\end{equation*}
		and for any $y_k\in B(x_k,\rho_1 r),\ k=1,\dots,n$, if we denote by $L^y_k$ the $k$-plane passing through $y_0,\dots,y_k$, then we have
		\begin{equation}\label{eq:yk far from Lk}
		\dist(y_k, L^y_{k-1})\ge\gamma r.
		\end{equation}
		\item[(b)] There exists a family of balls $\{B_i\}_{i\in I_B},$ with radii $r(B_i)=4\gamma r$, centered on $B$, so that the balls $\{10B_i\}_{i\in I_B}$ are pairwise disjoint,
		\begin{equation}\label{eq:unbalanced condition 1}
		\sum_{i\in I_B} \mu(B_i)\gtrsim \mu(B),
		\end{equation}
		and
		\begin{equation}\label{eq:unbalanced condition 2}
		\Theta_{\mu}(B_i)\gtrsim\gamma^{-1}\Theta_{\mu}(B).
		\end{equation}
	\end{itemize}
\end{lemma}
We will say that a ball $B$ is $\gamma$-balanced if the alternative (a) holds.	
\begin{lemma}\label{lem:small alpha gives balanced ball}
	Let $\mu$ be a Radon measure on $\R^d$, $B\subset\R^d$ be a ball such that $\mu(B)\approx \mu(1.1B)>0$. Suppose $L$ is the $n$-plane minimizing $\alpha_{\mu}(1.1B)$ and that $L$ intersects $0.9B$.
	There exist $C=C(n,d)<1,\,\gamma=\gamma(n,d)<1$ such that if $\alpha_{\mu}(1.1B)\le C\gamma$, then $B$ is $\gamma$-balanced.
\end{lemma}
\begin{proof}
	Proof by contradiction. Suppose that $B$ is not $\gamma$-balanced, i.e. that the alternative (b) in \lemref{lem:balanced balls} holds. 
	
	We will estimate $\alpha_{\mu}(1.1B)$ from below. Let $c$ be the constant minimizing $\alpha_{\mu}(1.1B)$, so that by \eqref{eq:c smaller then mu over Hn}
	\begin{equation*}
	c\lesssim\Theta_{\mu}(1.1B)\approx\Theta_{\mu}(B).
	\end{equation*}
	Let balls $\{B_i\}_{i\in I_B}$ be as in  \lemref{lem:balanced balls} (b), with $r(B_i)=r_i=4\gamma r(B)$. Let $f\in\lip_1(1.1B)$ be defined in such a way that $f\equiv r_i$ on each $B_i$ and $\supp f\subset \bigcup_{i\in I_B} 2B_i\subset 1.1B$. Then,
	\begin{equation*}
	\int f\ d\mu \ge \sum_{i\in I_B} \mu(B_i)r_i\overset{\eqref{eq:unbalanced condition 1}}{\gtrsim} \gamma r(B)\mu(B).
	\end{equation*}
	On the other hand,
	\begin{multline*}
	c\int f\ d\Hn{L}\overset{\eqref{eq:c smaller then mu over Hn}}{\lesssim} \Theta_{\mu}(B) \sum_{i\in I_B} r_i^{n+1}= \Theta_{\mu}(B) \sum_{i\in I_B}\Theta_{\mu}(B_i)^{-1}\mu(B_i) r_i\\
	\overset{\eqref{eq:unbalanced condition 2}}{\lesssim}\gamma \sum_{i\in I_B} \mu(B_i)r_i\lesssim \gamma^2 r(B)\mu(B).
	\end{multline*}
	The two estimates above imply that for some dimensional constants $C_1, C_2$
	\begin{equation*}
	\alpha_{\mu}(1.1B)\ge C_1\gamma - C_2\gamma^2> C\gamma,
	\end{equation*}
	if we take $\gamma$ and $C=C(C_1,C_2)$ small enough. We reach a contradiction with the assumption $\alpha_{\mu}(1.1B)\le C\gamma$.
\end{proof}
\begin{cor}\label{cor:doubling cubes are balanced}
	Let $Q\in\Tree_0$. Then $2.5B_Q$ is $\gamma$-balanced, where $\gamma=\gamma(n,d)$.
\end{cor}
\begin{proof}
	We know that $\mu(1.5B_Q)\approx_{A,\tau} \mu(9B_Q)$, and that
	\begin{equation*}
	\alpha_{\mu}(3B_Q)\overset{\eqref{eq:sum of alphas containing x estimate}}{\lesssim_{A,\tau}}\varepsilon_0,
	\end{equation*}
	which implies (for $\varepsilon_0$ small enough) that the best approximating plane for $3B_Q$ intersects $2B_Q$. Applying \lemref{lem:small alpha gives balanced ball} to $B=2.5B_Q$ finishes the proof.
	%
	
\end{proof}
\subsection{Small measure of cubes from \texorpdfstring{$\BS$}{BA} and \texorpdfstring{$\F$}{F}}
\begin{lemma}\label{lem:small measure of BA and F}
	We have
	\begin{gather*}
	\sum_{Q\in\BS}\mu(Q)\lesssim\varepsilon_0\mu(R_0),\\
	\sum_{Q\in\F}\mu(Q)\lesssim_{A,\tau} {\varepsilon_0}^{1/4}\,\mu(R_0).
	\end{gather*}
\end{lemma}

\begin{proof}
	We start by estimating the measure of cubes from $\BS$. 
	%
	We use the definition of $\BS$ \eqref{eq:BA} to get
	\begin{equation*}
	\sum_{Q\in\BS}\mu(Q)\le 2\sum_{Q\in\BS}\mu(Q\setminus G)\le 2\mu(R_0\setminus G)\overset{\eqref{eq:set Gr is small}}{\le}2\varepsilon_0\mu(3B_0)\approx\varepsilon_0\mu(R_0).
	\end{equation*}
	
	Concerning $\F$, we use the $5R$-covering lemma to get a countable family of pairwise disjoint balls $B_i:=3B_{Q_i}$, $Q_i\in\F$, such that $\bigcup_i 5B_i\supset\bigcup_{Q\in\F} Q$. For every $i$ we have
	\begin{equation*}
	\mu(5B_i)=\mu(15B_{Q_i})\overset{\eqref{eq:notHD}}{\lesssim} A\ell(Q_i)^n\overset{\eqref{eq:notLD}}{\le} \frac{A}{\tau}\mu(B_i).
	\end{equation*} 
	Then
	\begin{multline*}
	\sum_{Q\in\F}\mu(Q)\lesssim\sum_i\mu(5B_i)\lesssim_{A,\tau} \sum_i\mu(B_i)\\
	\overset{\eqref{eq:def of F}}{\le}\frac{1}{\varepsilon_0^{1/4}}\sum_i\mu(B_i\cap \RFar)\le \frac{1}{\varepsilon_0^{1/4}}\mu(\RFar)\overset{\eqref{eq:small measure of RFar}}{\lesssim_{A,\tau}}\varepsilon_0^{1/4}\mu(R_0).
	\end{multline*}
\end{proof}

\section{Construction of the Lipschitz graph}\label{sec:construction of graph}
In this section we construct the Lipschitz graph $\Gamma$. At the beginning of Subsection \ref{sec:F on good part} we define also the good set $R_G\subset\Gamma\cap R_0$, and we show that $\restr{\mu}{R_G}\ll\mathcal{H}^n$. We start by proving some auxiliary estimates.
\subsection{Estimates involving best approximating planes}
\begin{lemma}{{\cite[Lemma 6.4]{azzam2015characterization}}}\label{lem:geometrical lemma}
	Suppose $P_1, P_2$ are $n$-planes in $\R^d$, $X=\{x_0,\dots, x_n \}$ is a collection of $n$ points, and
	\begin{gather}
	d_1=d_1(X)=\dfrac{1}{\diam(X)}\min_i\big\{\dist\big(x_i,\mathrm{span}(X\setminus\{x_i\}) \big) \big\}\in(0,1),\tag{a} \\
	\dist(x_i,P_j)<d_2\,\diam(X)\quad \text{for}\quad i=0,\dots,n\quad\text{and}\quad j=1,2, \tag{b}
	\end{gather}
	where $d_2<d_1/(2d)$. Then for $y\in P_2$
	\begin{equation}
	\dist(y,P_1)\le d_2\left(\frac{2d}{d_1}\dist(y,X) + \diam(X)\right).
	\end{equation}
\end{lemma}

\begin{lemma}\label{lem:planes close to each other for similar cubes}
	Suppose $Q_1, Q_2\in\Tree_0$ are such that $\dist(Q_1,Q_2)\lesssim\ell(Q_1)\approx\ell(Q_2)$. Let $P\in\Tree_0$ be the smallest cube such that $3B_P\supset 3B_{Q_1}\cup 3B_{Q_2}.$ Then $\ell(P)\approx\ell(Q_1),$ and for all $y\in L_{Q_2}$
	\begin{equation*}
	\dist(y,L_{Q_1})\lesssim_{A,\tau}\beta_{\mu,2}(3B_P)(\dist(y,Q_2)+ \ell(Q_2)).
	\end{equation*}
	In particular,
	\begin{equation}
	\measuredangle(L_{Q_1},L_{Q_2})\lesssim_{A,\tau}\beta_{\mu,2}(3B_P)\lesssim_{A,\tau}\varepsilon_0.
	\end{equation}
\end{lemma}

\begin{proof}
	Since $3B_0\supset 3B_{Q_1}\cup 3B_{Q_2}$ and $R_0\in \Tree_0$, the cube $P$ is well-defined. The comparability $\ell(P)\approx\ell(Q_2)$ holds due to the assumption $\dist(Q_1,Q_2)\lesssim\ell(Q_1)\approx\ell(Q_2)$.
	
	Since $Q_1\in\Tree_0$, Corollary \ref{cor:doubling cubes are balanced} tells us that $2.5B_{Q_1}$ is $\gamma$-balanced. Let $x_0,\dots, x_n\in 2.5B_{Q_1}$ be the points from alternative (a) in \lemref{lem:balanced balls}. Thus, we have a family of balls $\{B_k \coloneqq B(x_k,\rho_1 r(2.5B_{Q_1}))\}_{k=0,\dots,n}$, such that $\mu(B_k\cap 2.5B_{Q_1})\ge \rho_2\mu(2.5B_{Q_1})\approx_{A,\tau} \rho_2\ell(Q_1).$ 
	
	Since $r(B_k) = \rho_1 r(2.5B_{Q_1})\approx \ell(P)$, and $B_k\subset 3B_{Q_1}\subset 3B_P$, it is clear that
	\begin{equation*}
	\frac{1}{\mu(B_k)}\int_{B_k}\left(\frac{\dist(x,L_{Q_1})}{r(B_k)}\right)^2\ d\mu(x)\lesssim_{\rho_2,A,\tau} \beta_{\mu,2}(3B_{Q_1})^2\lesssim_{A,\tau} \beta_{\mu,2}(3B_P)^2,
	\end{equation*}
	and
	\begin{equation*}
	\frac{1}{\mu(B_k)}\int_{B_k}\left(\frac{\dist(x,L_P)}{r(B_k)}\right)^2\ d\mu(x)\lesssim_{\rho_2,A,\tau} \beta_{\mu,2}(3B_P)^2.
	\end{equation*}
	Keeping in mind that $\rho_2$ is a dimensional constant, we will not signal dependence on it in further computations. We use the above estimates and the Chebyshev inequality to find points $y_k\in B_k$ such that
	\begin{gather*}
	\dist(y_k, L_{Q_1})\lesssim_{A,\tau} \beta_{\mu,2}(3B_P)\ell(P),\\
	\dist(y_k, L_{P})\lesssim_{A,\tau} \beta_{\mu,2}(3B_P)\ell(P).
	\end{gather*}
	We would like to apply \lemref{lem:geometrical lemma} to $n$-planes $L_{Q_1}, L_P$ and points $X=\{y_0,\dots, y_n\}$. We have $d_1\gtrsim \gamma$ thanks to \eqref{eq:yk far from Lk}. Furthermore, due to estimate \eqref{eq:sum of betas containing x estimate} we know that $\beta_{\mu,2}(3B_P)\lesssim_{A,\tau}\varepsilon_0$, and so $\beta_{\mu,2}(3B_P)\approx_{A,\tau}d_2<d_1/(2d)$ for $\varepsilon_0$ small enough. Thus, 
	\begin{align}\label{eq:LP close to LQ1}
	\dist(y,L_{Q_1})\lesssim_{A,\tau}\beta_{\mu,2}(3B_P)(\dist(y,Q_1)+ \ell(Q_1))\quad&\text{for $y\in L_P$},\\
	\dist(y,L_{P})\lesssim_{A,\tau}\beta_{\mu,2}(3B_P)(\dist(y,Q_1)+ \ell(Q_1))\quad&\text{for $y\in L_{Q_1}$}.\notag
	\end{align}
	Since the assumptions about cubes $Q_1$ and $Q_2$ are identical, it turns out that the estimates above are also valid if we replace $Q_1$ with $Q_2$, i.e.
	\begin{align}\label{eq:LQ2 close to LP}
	\dist(y,L_{Q_2})\lesssim_{A,\tau}\beta_{\mu,2}(3B_P)(\dist(y,Q_2)+ \ell(Q_2))\quad&\text{for $y\in L_P$}\notag,\\
	\dist(y,L_{P})\lesssim_{A,\tau}\beta_{\mu,2}(3B_P)(\dist(y,Q_2)+ \ell(Q_2))\quad&\text{for $y\in L_{Q_2}$}.
	\end{align}
	Using the triangle inequality, estimates \eqref{eq:LQ2 close to LP}, \eqref{eq:LP close to LQ1}, and the fact that $(\dist(y,Q_1)+ \ell(Q_1))\approx (\dist(y,Q_2)+ \ell(Q_2))$ we finally reach the desired inequality
	\begin{equation*}
	\dist(y,L_{Q_1})\lesssim_{A,\tau}\beta_{\mu,2}(3B_P)(\dist(y,Q_2)+ \ell(Q_2)) \quad\text{for $y\in L_{Q_2}$}.
	\end{equation*}	
\end{proof}

\begin{lemma}\label{lem:planes close to each other for ancestors}
	Let $Q, P\in\Tree$ be such that $\ell(Q)\lesssim\ell(P)$ and $\dist(Q,P)\lesssim\ell(P)$. Then for any $x\in L_{Q}\cap CB_{Q}$ we have
	\begin{equation*}
	\dist(x,L_P)\lesssim_{A,\tau,C}\sqrt{\varepsilon_0}\ell(P).
	\end{equation*}
\end{lemma}
\begin{proof}
	Consider first the special case $Q\subset P$.
	
	By Corollary \ref{cor:doubling cubes are balanced}, there exist balls $B_k=B(x_k,\rho_1 r(Q)),\ k=0,\dots,n,$ such that $\mu(B_k\cap 2.5B_{Q})\ge \rho_2\mu(2.5B_{Q})$, and  $\dist(y_k, L^y_{k-1})\gtrsim\gamma \ell(Q)$ for $y_k\in B_k$ (see \eqref{eq:yk far from Lk}).
	
	It follows by \eqref{eq:notF} that, for $\varepsilon_0$ small enough, $B_i\setminus\RFar\not =\varnothing$. Fix some $y_i\in B_i\setminus\RFar$ for every $i=0,\dots, n$, so that
	\begin{gather*}
	\dist(y_i,L_Q)\lesssim \sqrt{\varepsilon_0}\ell(Q),\\
	\dist(y_i,L_P)\lesssim \sqrt{\varepsilon_0}\ell(P).
	\end{gather*}
	Let $z_i$ be the orthogonal projection of $y_i$ onto $L_Q$. Since $\ell(Q)\lesssim\ell(P)$, the triangle inequality yields
	\begin{equation}\label{eq:zi close to Lp}
	\dist(z_i,L_P)\le |y_i-z_i|+\dist(y_i,L_P)\lesssim\sqrt{\varepsilon_0}\ell(P).
	\end{equation}
	Furthermore, if $\varepsilon_0$ is small enough, $|y_i-z_i|\lesssim \sqrt{\varepsilon_0}\ell(Q)$ and $\dist(y_k,L^y_{k-1})\gtrsim \ell(Q)$ imply that $\dist(z_k,L^z_{k-1})\gtrsim \ell(Q)$, and that $z_i\in 3B_Q$. Since $L_Q = \text{span} (z_0,\dots,z_n)$, it follows by elementary geometry and \eqref{eq:zi close to Lp} that for any $x\in L_{Q}\cap CB_{Q}$
	\begin{equation*}
	\dist(x,L_P)\lesssim_C\sqrt{\varepsilon_0}\ell(P),
	\end{equation*}
	which concludes the proof in the case $Q\subset P$.
	
	Now, the general case follows by the above and \lemref{lem:planes close to each other for similar cubes}. Indeed, take a cube $R\in \Tree$ such that $R\supset Q$ and $\ell(R)=\ell(P)$. The assumption $\dist(Q,P)\lesssim\ell(P)$ gives us $\dist(R,P)\lesssim\ell(P)$, and so we can apply \lemref{lem:planes close to each other for similar cubes} to get
	\begin{equation*}
	\dist(y,L_P)\lesssim_{A,\tau,C}\varepsilon_0\ell(P),\qquad y\in L_R\cap CB_R.
	\end{equation*}
	On the other hand, since $Q\subset R$, we already know that for $x\in L_Q\cap CB_Q$ we have
	\begin{equation*}
	\dist(x,L_R)\lesssim_C\sqrt{\varepsilon_0}\ell(R)=\sqrt{\varepsilon_0}\ell(P).
	\end{equation*}
	Putting together the two inequalities above yields the desired result.
\end{proof}

\begin{lemma}\label{lem:cBQ similar for similar cubes}
	Suppose the cubes $Q_1, Q_2\in \Tree_0$ satisfy $2.5B_{Q_1}\subset 2.5B_{Q_2},\ \ell(Q_1)\approx \ell(Q_2).$ Then
	\begin{equation*}
	|c_{Q_1}-c_{Q_2}|\lesssim_{A,\tau} \varepsilon_0.
	\end{equation*}
\end{lemma}
\begin{proof}
	Set $B_i = 2.5B_{Q_i}, r_i=r(B_i), z_i=z(B_i), c_i=c_{Q_i}, L_i=L_{Q_i}$ for $i=1,2.$ Let $\phi(z)=(r_1-|z_1-z|)_+\in\lip_1(B_1).$ Then
	\begin{align*}
	r_1^n|c_1-c_2|&\lesssim \left\lvert \int\phi \ c_1d\Hn{L_1} -  \int\phi \ c_2d\Hn{L_1}\right\rvert\\
	&\le \left\lvert \int\phi \ c_1d\Hn{L_1} - \int\phi \ d\mu \right\rvert + \left\lvert \int\phi\ d\mu -  \int\phi \ c_2d\Hn{L_2}\right\rvert\\
	&\quad\quad\quad\quad\quad\qquad\qquad\qquad\qquad+ c_2\left\lvert\int\phi\ d\Hn{L_2} - \int\phi \ d\Hn{L_1} \right\rvert\\
	&\le F_{B_1}(\mu, c_1\Hn{L_1})+F_{B_2}(\mu, c_2\Hn{L_2}) + c_2\left\lvert\int\phi\ d\Hn{L_2} - \int\phi \ d\Hn{L_1} \right\rvert\\
	&\overset{\eqref{eq:sum of F containing x estimate}, \eqref{eq:cBQ estimate}}{\lesssim_{A,\tau}} \varepsilon_0r_1^n + \left\lvert\int\phi\ d\Hn{L_2} - \int\phi \ d\Hn{L_1} \right\rvert.
	\end{align*}
	The fact that the last term above can also be estimated by $\varepsilon_0r_1^n$ follows easily by the fact that $L_1$ and $L_2$ are close to each other, see \lemref{lem:planes close to each other for similar cubes}.
\end{proof}

\subsection{Lipschitz function \texorpdfstring{$F$}{F} corresponding to the good part of \texorpdfstring{$R_0$}{R0}}\label{sec:F on good part}
Consider an auxiliary function
\begin{equation}\label{eq:definition d}
d(x) = \inf_{Q\in\Tree} \big(\dist(x,Q)+\diam(B_Q)\big),\quad x\in\R^d.
\end{equation}
Let 
\begin{equation*}
R_G = \{x\in\R^d\ :\ d(x)=0 \}.
\end{equation*}
Observe that, by the definition of function $d$, we have $R_0\setminus\bigcup_{Q\in\Stop}Q\subset R_G$. 
\begin{lemma}\label{lem:mu on RG absolutely continuous wrto Hn}
	We have $\restr{\mu}{R_G}\ll \mathcal{H}^n$, and for $x\in R_G$
	\begin{equation*}
	\Theta^{n}_*(\mu,x)\approx_{A,\tau}\Theta^{*n}(\mu,x)\approx_{A,\tau}1.
	\end{equation*}
	In consequence, $d\restr{\mu}{R_G}=g\,d\Hn{R_G}$ with $g\approx_{A,\tau} 1$.
\end{lemma}
\begin{proof}
	Let $x\in R_G$. Given some small $h>0$ we use the fact that $d(x)=0$ to find $Q\in\Tree$  such that $B(x,h)\subset 3B_Q$ and $\ell(Q)\approx h$. Then
	\begin{equation*}
	\mu(B(x,h))\le\mu(3B_Q)\overset{\eqref{eq:notHD}}{\lesssim_{A}}\ell(Q)^n\approx h^n.
	\end{equation*}
	Now, let $P\in\Tree$ be such that $3B_P\subset B(x,h)$ and $\ell(P)\approx h$. Then
	\begin{equation*}
	\mu(B(x,h))\ge\mu(3B_P)\overset{\eqref{eq:notLD}}{\gtrsim_{\tau}}\ell(P)^n\approx h^n.
	\end{equation*}
	Letting $h\rightarrow0$ we get $1\lesssim_{\tau}\Theta^{n}_*(\mu,x)\le\Theta^{*n}(\mu,x)\lesssim_{A}1$ for $x\in R_G$. The upper density estimate and \cite[Theorem 6.9 (1)]{mattila1999geometry} imply $\restr{\mu}{R_G}\ll \restr{\mathcal{H}^n}{R_G}$ and $\restr{\mu}{R_G}(B)\lesssim_A \restr{\mathcal{H}^n}{R_G}(B)$ for all $B\subset\R^d$ Borel. The lower density estimate together with \cite[Theorem 6.9 (2)]{mattila1999geometry} give $\restr{\mathcal{H}^n}{R_G}\ll\restr{\mu}{R_G}$ and $\restr{\mathcal{H}^n}{R_G}(B)\lesssim_\tau \restr{\mu}{R_G}(B)$ (in particular, $\restr{\mathcal{H}^n}{R_G}$ is a finite Radon measure). Putting it all together, we use Radon-Nikodym theorem to get $d\restr{\mu}{R_G}=g\,d\Hn{R_G}$, with $g\approx_{A,\tau} 1$.
\end{proof}
In this subsection we will define function $F(x)$ for $x\in \Pi_0(R_G)\subset L_0.$ 

\begin{lemma}\label{lem:Lip property involving d(x)}
	If $\varepsilon_0$ and $\theta$ are small enough, then for any $x_1,x_2\in\R^d$
	\begin{equation}\label{eq:assertion of Lipschitz property involving d(x)}
	|\Pi_0^{\perp}(x_1)-\Pi_0^{\perp}(x_2)|\lesssim \theta|\Pi_0(x_1)-\Pi_0(x_2)| +d(x_1) + d(x_2).
	\end{equation}
\end{lemma}
\begin{proof}
	Fix some small $h>0$. Let $Q_1, Q_2\in\Tree$ be such that
	\begin{equation*}
	\dist(x_i,Q_i)+\diam(B_{Q_i})\le d(x_i)+h,\quad i=1,2.
	\end{equation*}
	Take any $y_i\in Q_i$. Note that $|x_i-y_i|\le d(x_i)+h$. The triangle inequality gives us
	\begin{align*}
	|\Pi_0^{\perp}(x_1)-\Pi_0^{\perp}(x_2)|&\le |\Pi_0^{\perp}(y_1)-\Pi_0^{\perp}(y_2)| + |\Pi_0^{\perp}(x_1)-\Pi_0^{\perp}(y_1)| + |\Pi_0^{\perp}(x_2)-\Pi_0^{\perp}(y_2)|\\
	&\le |\Pi_0^{\perp}(y_1)-\Pi_0^{\perp}(y_2)| + d(x_1)+d(x_2)+2h,
	\end{align*}
	and similarly
	\begin{equation*}
	|\Pi_0(y_1)-\Pi_0(y_2)|\le|\Pi_0(x_1)-\Pi_0(x_2)| + d(x_1)+d(x_2)+2h.
	\end{equation*}
	Hence, if we show that 
	\begin{equation}\label{eq:reduction from Lip property involving d(x)}
	|\Pi_0^{\perp}(y_1)-\Pi_0^{\perp}(y_2)|\lesssim\theta |\Pi_0(y_1)-\Pi_0(y_2)| + d(x_1) + d(x_2)+2h,
	\end{equation}
	use the two former inequalities, and let $h\rightarrow 0$, we will get \eqref{eq:assertion of Lipschitz property involving d(x)}.
	
	Let $P_i\in\Tree$ be the smallest cubes such that $3B_{P_i}\supset B_{Q_i}$ and 
	\begin{equation*}
	\ell(P_i)\approx\varepsilon_0^{1/n} |y_1 - y_2| + \sum_i \ell(Q_i).
	\end{equation*}
	We also take the smallest cube $R\in\Tree$ such that $3B_{R}\supset 3B_{P_1}\cup 3B_{P_2}$ and
	\begin{equation}\label{eq:ellR is like y1-y2}
	\ell(R)\approx  |y_1 - y_2|+ \sum_i \ell(Q_i).
	\end{equation}
	We use the fact that $3B_{R}\supset 3B_{P_1}\cup 3B_{P_2}$, the estimates \eqref{eq:notLD}, \eqref{eq:notHD}, the smallness of $\beta$ numbers \eqref{eq:sum of betas containing x estimate}, and the bound $\varepsilon_0\ell(R)^n\lesssim\ell(P_i)^n$, to get		
	\begin{equation*}
	\frac{1}{\mu(9B_{P_i})}\int_{3B_{P_i}}\left(\frac{\dist(w,L_R)}{\ell(R)}\right)^2\ d\mu(w)\lesssim_{A,\tau} \frac{\ell(R)^n\beta_{\mu,2}(3B_R)^2}{\ell(P_i)^n}\lesssim_{A,\tau}\frac{\ell(R)^n\varepsilon_0^2}{\ell(P_i)^n}\lesssim\varepsilon_0.
	\end{equation*}
	Hence, by Chebyshev's inequality, there exist some $z_i\in 3B_{P_i}$ such that
	\begin{equation}\label{eq:zi close to pizi}
	\dist(z_i,L_R) = |z_i-\pi(z_i)|\lesssim_{A,\tau}\sqrt{\varepsilon_0}\ell(R)\lesssim\sqrt{\varepsilon_0}(|y_1-y_2|+d(x_1) + d(x_2) + 2h),
	\end{equation}
	where $\pi$ denotes orthogonal projection onto $L_R$, and the second inequality is due to \eqref{eq:ellR is like y1-y2}. 
	Note also that, since $y_i, z_i\in 3B_{P_i}$, we have
	\begin{equation}\label{eq:yi close to zi}
	|y_i - z_i|\lesssim \ell(P_i)\lesssim \varepsilon_0^{1/n} |y_1 - y_2| + d(x_1) + d(x_2) + 2h.
	\end{equation}
	
	Now, the triangle inequality and 1-Lipschitz property of $\Pi_0^{\perp}$ give us
	\begin{equation*}
	|\Pi_0^{\perp}(y_1)-\Pi_0^{\perp}(y_2)|\le |\Pi_0^{\perp}(\pi(z_1))-\Pi_0^{\perp}(\pi(z_2))| + \sum_{i=1}^2 \big(|z_i-\pi(z_i)| + |y_i - z_i|\big).
	\end{equation*}		
	To estimate the first term from the right hand side we use the fact that projections onto $L_R$ and $L_0$ are close to each other \eqref{eq:notBS}, the triangle inequality, and 1-Lipschitz property of $\Pi$:
	\begin{align*}
	|\Pi_0^{\perp}(\pi(z_1))-\Pi_0^{\perp}(\pi(z_2))|&\lesssim \theta|\pi(z_1))-\pi(z_2)|\lesssim \theta|\Pi_0(\pi(z_1))-\Pi_0(\pi(z_2))|\\ &\le \theta\big(|\Pi_0(y_1)-\Pi_0(y_2)| + \sum_{i=1}^2 \big(|z_i-\pi(z_i)| + |y_i - z_i|\big) \big).
	\end{align*}
	Putting together the two estimates above, as well as \eqref{eq:zi close to pizi}, \eqref{eq:yi close to zi}, yields
	\begin{multline*}
	|\Pi_0^{\perp}(y_1)-\Pi_0^{\perp}(y_2)|\lesssim \theta|\Pi_0(y_1)-\Pi_0(y_2)| + \sum_{i=1}^2 \big(|z_i-\pi(z_i)| + |y_i - z_i|\big)\\
	\lesssim \theta|\Pi_0(y_1)-\Pi_0(y_2)| + C(A,\tau)\sqrt{\varepsilon_0}\big(|y_1-y_2| + d(x_1) + d(x_2) + 2h\big)\\
	 + \varepsilon_0^{1/n} |y_1 - y_2| + d(x_1) + d(x_2) + 2h.
	\end{multline*}
	Since $|y_1 - y_2|\approx |\Pi_0(y_1)-\Pi_0(y_2)| + |\Pi_0^{\perp}(y_1)-\Pi_0^{\perp}(y_2)|$, we may take $\varepsilon_0=\varepsilon_0(A,\tau,\theta)$ so small that 
	\begin{equation*}
	\big(C(A,\tau)\sqrt{\varepsilon_0}+ \varepsilon_0^{1/n}\big) |y_1 - y_2|\le \theta\big(|\Pi_0(y_1)-\Pi_0(y_2)| + |\Pi_0^{\perp}(y_1)-\Pi_0^{\perp}(y_2)|\big).
	\end{equation*}
	Then, for $\theta$ small enough, we obtain the desired inequality \eqref{eq:reduction from Lip property involving d(x)}:
	\begin{equation*}
	|\Pi_0^{\perp}(y_1)-\Pi_0^{\perp}(y_2)|\lesssim\theta |\Pi_0(y_1)-\Pi_0(y_2)| + d(x_1)+d(x_2) + 2h.
	\end{equation*}
\end{proof}

The lemma above gives us for any $x,y\in R_G$
\begin{equation*}
|\Pi_0^{\perp}(x)-\Pi_0^{\perp}(y)|\lesssim \theta|\Pi_0(x)-\Pi_0(y)|.
\end{equation*}
This allows us to define a function $F$ on $\Pi_0(R_G)\subset L_0$ as
\begin{equation}\label{eq:def of F on RG}
F(\Pi_0(x)) = \Pi_0^{\perp}(x),\qquad x\in R_G,
\end{equation}
with $\lip(F)\lesssim \theta$. Note that the graph of such $F$ is precisely $R_G$.

\subsection{Extension of \texorpdfstring{$F$}{F} to the whole \texorpdfstring{$L_0$}{L0}}\label{sec:extension of F}
For any $z\in L_0$ let us define
\begin{equation}\label{eq:definition D}
D(z) = \inf_{x\in \Pi_0^{-1}(z)} d(x) = \inf_{Q\in\Tree} \big( \dist(z,\Pi_0(Q)) + \diam(B_Q)\big).
\end{equation}

For each $z\in L_0$ with $D(z)>0$, i.e.  $z\in L_0\setminus\Pi_0(R_G)$, we define $J_z$ as the largest dyadic cube from $L_0$ such that $z\in J_z$ and
\begin{equation*}
\diam(J_z)\le \frac{1}{20}\inf_{u\in J_z} D(u).
\end{equation*}
Let $J_i, i\in I,$ be a relabeling of the set of all such cubes $J_z$, without repetition. 

\begin{lemma}\label{lem:properties of Ji}
	The cubes $\{J_i\}_{i\in I}$ are disjoint and satisfy the following:
	\begin{itemize}
		\item[(a)] If $z\in 15J_i$, then $5\diam(J_i)\le D(z)\le 50\diam(J_i)$.
		\item[(b)] If $15 J_i\cap 15 J_{i'}\not = \varnothing$, then
		\begin{equation*}
		\ell(J_i)\approx \ell(J_{i'}).
		\end{equation*}
		\item[(c)] For each interval $J_i$ there are at most $N$ intervals $J_{i'}$ such that $15J_i\cap 15J_{i'}\not = \varnothing$.
		\item[(d)] $L_0\setminus \Pi_0(R_G) = \bigcup_{i\in I} J_i = \bigcup_{i\in I} 15J_i.$
	\end{itemize}
\end{lemma}
The proof is straightforward and follows directly from the definition of $J_i$, see \cite[Lemma 7.20]{tolsa2014analytic}.

Note that, since $\beta_{\mu,2}(3B_0)$ is very small \eqref{eq:sum of betas containing x estimate} and $R_0$ is doubling, we have 
$\dist(z_0, L_0)\le 2r(R_0)=\frac{1}{14}r_0$. It follows that
\begin{equation}\label{eq:BR0 subset B0}
\Pi_0(R_0)\subset\Pi_0(B_0)\subset\Pi_0(1.01B_0)\subset 1.1B_0\cap L_0.
\end{equation}
We define the set of indices  
\begin{equation}\label{eq:definition I0}
I_0 = \{i\in I\ :\ J_i\cap 1.5B_0\not= \varnothing \}.
\end{equation}
\begin{lemma}\label{lem:dist of Ji to B0}
	The following holds:
	\begin{itemize}
		\item[(a)] If $i\in I_0$, then $\diam(J_i)\le0.2r_0$, and $3J_i\subset L_0\cap 1.9B_0$.
		\item[(b)]If $J_i\cap 1.4B_0 = \varnothing$ (in particular if $i\not\in I_0$), then
		\begin{equation*}
		\ell(J_i)\approx\dist(z_0,J_i)\approx |z_0-z|\gtrsim\ell(R_0)\quad \text{for all $z\in J_i$}.
		\end{equation*}
	\end{itemize}
\end{lemma}
\begin{proof}
	We begin by proving (a). Suppose $i\in I_0$. Then $J_i\cap 1.5B_0\not=\varnothing$ and
	\begin{equation*}
	3J_i\subset L_0\cap B(z_0, 1.5r_0+2\diam(J_i)).
	\end{equation*}
	We need to estimate $\diam(J_i)$. By the definition of $J_i$, we have
	\begin{equation*}
	\diam(J_i)\le\frac{1}{20}\inf_{u\in J_i}D(u).
	\end{equation*}
	Since $J_i\cap 1.5B_0\not=\varnothing$ we have $\inf_{u\in J_i}D(u)\le\max_{u\in L_0\cap 1.5B_0}D(u)$, and so it suffices to estimate the latter quantity. Note that the definition of $d$ \eqref{eq:definition d} gives for $x\in 1.5B_0$
	\begin{equation*}
	d(x)\le\dist(x,R_0)+\diam(B_0)\le 1.5r_0+2r_0=3.5r_0.
	\end{equation*}
	Hence, by the definition of $D$ \eqref{eq:definition D}
	\begin{equation*}
	\max_{u\in L_0\cap 1.5B_0}D(u)\le \max_{x\in 1.5B_0}d(x)\le 3.5r_0.
	\end{equation*}
	It follows that $\diam(J_i)\le \frac{7}{40}r_0$, and 
	\begin{equation*}
	3J_i\subset L_0\cap B(z_0, 1.85r_0).
	\end{equation*}
	
	Now, let us prove (b). Suppose $J_i\cap 1.4B_0 = \varnothing$ and $z\in J_i$. Clearly, $|z_0-z|\ge 1.4r_0$. Together with the definition of $D$ \eqref{eq:definition D} this gives
	\begin{equation*}
	D(z)\le |\Pi_0(z_0)-z|+\diam(B_0)\le 3|z_0-z|.
	\end{equation*}
	On the other hand, by \eqref{eq:BR0 subset B0} we have
	\begin{equation*}
	D(z)\ge \dist(z,\Pi_0(R_0))\ge \dist(z,1.1B_0)=|z_0-z|-1.1r_0\ge\frac{3}{14}|z_0-z|.
	\end{equation*}
	Putting together the two estimates above gives for $z\in J_i$
	\begin{equation*}
	\frac{1}{5}|z_0-z|\le D(z)\le 3|z_0-z|.
	\end{equation*}
	Applying \lemref{lem:properties of Ji} (a) yields
	\begin{equation*}
	\frac{5}{3}\diam(J_i)\le |z_0-z|\le 250\diam(J_i).
	\end{equation*}
	Moreover, since
	\begin{equation*}
	|z_0-z|-\diam(J_i)\le\dist(z_0,J_i)\le |z_0-z|,
	\end{equation*}
	we finally obtain
	\begin{equation*}
	\frac{2}{3}\diam(J_i)\le\dist(z_0,J_i)\le 250\diam(J_i).
	\end{equation*}
\end{proof}

\begin{lemma}\label{lem:Qi corresposnding to Ji}
	Given $i\in I_0$, there exists a cube $Q_i\in\Tree$ such that
	\begin{gather*}
	\ell(J_i)\approx \ell(Q_i),\label{eq:Ji and Qi comparable}\\
	\dist(J_i,\Pi_0(Q_i))\lesssim \ell(J_i).\label{eq:Ji close to Qi}
	\end{gather*}
\end{lemma}

\begin{proof}
	Let $i\in I_0$ and $z\in J_i$. We know by \lemref{lem:properties of Ji} (a) that $D(z)\approx \ell(J_i)$. Thus, by the definition of $D$ \eqref{eq:definition D} we may find $Q\in\Tree$ such that
	\begin{equation*}
	\dist(z,\Pi_0(Q)) + \diam(B_Q)\approx \ell(J_i).
	\end{equation*}
	Clearly, $\ell(Q)\lesssim \ell(J_i)$, and $\dist(J_i,\Pi_0(Q))\lesssim \ell(J_i)$. If $\ell(Q)\gtrsim \ell(J_i)$, we set $Q_i=Q$ and we are done. If that is not the case, then we define $Q_i$ as the ancestor $P\supset Q$ satisfying $\ell(P)\gtrsim \ell(J_i)$ (we can always do that because $\ell(J_i)\lesssim\ell(R_0)$ by \lemref{lem:dist of Ji to B0} (a)).
\end{proof}

For all $i\in I_0$ we define $F_i: L_0\rightarrow L_0^{\perp}$ as the affine function whose graph is the $n$-plane $L_{Q_i}$. Since $\measuredangle(L_{Q_i},L_0)\le \theta$ by \eqref{eq:notBS}, we have $\lip(F_i)\lesssim \theta$. For $i\not\in I_0$ set $F_i\equiv 0$, so that the graph of $F_i$ is the plane $L_0$.

\begin{lemma}\label{lem:Fi estimates}
	Suppose $10J_i\cap 10 J_{i'}\not = \varnothing$. We have:
	\begin{itemize}
		\item[(a)]if  $ i,i'\in I_0$, then
		\begin{equation*}
		\dist(Q_i,Q_{i'})\lesssim \ell(J_i),
		\end{equation*}
		\item[(b)]for $x\in 100J_i$
		\begin{equation*}
		|F_i(x)-F_{i'}(x)|\lesssim\sqrt{\varepsilon_0}\ell(J_i),
		\end{equation*}
		\item[(c)] $\lVert\nabla F_i - \nabla F_{i'}\rVert_{\infty}\lesssim \sqrt{\varepsilon_0}.$	
	\end{itemize}
\end{lemma}
\begin{proof}
	Let us start with (a). We know by Lemma \ref{lem:properties of Ji}(b) and \lemref{lem:Qi corresposnding to Ji} that $\ell(Q_i)\approx\ell(Q_{i'})\approx\ell(J_i)\approx\ell(J_{i'})$. Let $z_1\in Q_i, z_2\in Q_{i'}$ be such that $|\Pi_0(z_1)-\Pi_0(z_2)|\approx \dist(\Pi_0(Q_i),\Pi_0(Q_{i'}))$. Note that $d(z_1)\lesssim \ell(Q_i),\ d(z_2)\lesssim \ell(Q_{i'})$. It follows that
	\begin{multline*}
	\dist(Q_i,Q_{i'})\le|z_1-z_2|\le |\Pi_0^{\perp}(z_1)-\Pi_0^{\perp}(z_2)| +|\Pi_0(z_1)-\Pi_0(z_2)|\\
	\overset{\eqref{eq:assertion of Lipschitz property involving d(x)}}{\lesssim} |\Pi_0(z_1)-\Pi_0(z_2)|+d(z_1)+d(z_2)\lesssim \dist(\Pi_0(Q_i),\Pi_0(Q_{i'})) + \ell(J_i).
	\end{multline*}
	On the other hand, we have by \lemref{lem:Qi corresposnding to Ji}
	\begin{multline*}
	\dist(\Pi_0(Q_i),\Pi_0(Q_{i'}))\le \dist(\Pi_0(Q_i),J_i)) + \dist(J_i,J_{i'}) +\\ \dist(J_{i'},\Pi_0(Q_{i'})) + \diam(J_i)+\diam(J_{i'})
	\lesssim \ell(J_i).
	\end{multline*}
	The two estimates together give us (a). 
	
	Now, (b) and (c) for $i,\ i'\in I_0$ follow immediately because we can apply \lemref{lem:planes close to each other for similar cubes} to $Q_i$ and $Q_{i'}$.	
	If $i,\ i'\notin I_0$, then (b) and (c) are trivially true, since $F_i=F_{i'}\equiv 0$. The only remaining case is $i\in I_0,\ i'\notin I_0$.
	
	Since $10J_i\cap 10 J_{i'}\not = \varnothing$, we know by \lemref{lem:properties of Ji} (b) and \lemref{lem:dist of Ji to B0} that $\ell(J_i)\approx\ell(J_{i'})\approx\ell(R_0).$ We apply \lemref{lem:planes close to each other for similar cubes} to $Q_i$ and $R_0$, and the result follows..
\end{proof}

Now, to define function $F$ on $L_0\setminus\Pi_0(R_G)$ we consider the following partition of unity: for each $i\in I$ let $\widetilde{\varphi}_i\in C^{\infty}(L_0)$ be such that $\widetilde{\varphi}_i\equiv 1$ on $2J_i$, $\supp\widetilde{\varphi}_i\subset 3J_i$, and
\begin{gather*}
\lVert\nabla\widetilde{\varphi}_i\rVert_{\infty}\lesssim \ell(J_i)^{-1},\\ 
\lVert D^2\widetilde{\varphi}_i\rVert_{\infty}\lesssim \ell(J_i)^{-2}.
\end{gather*}
Now, we set
\begin{equation*}
\varphi_i = \frac{\widetilde{\varphi}_i}{\sum_{j\in I}\widetilde{\varphi}_j}.
\end{equation*}
Clearly, the family $\{\varphi_i\}_{i\in I}$ is a partition of unity subordinated to sets $\{3J_i\}_{i\in I}$. Moreover, the inequalities above together with \lemref{lem:properties of Ji} imply that each $\varphi_i$ satisfies
\begin{gather*}
\lVert\nabla{\varphi}_i\rVert_{\infty}\lesssim \ell(J_i)^{-1},\\ 
\lVert D^2{\varphi}_i\rVert_{\infty}\lesssim \ell(J_i)^{-2}.
\end{gather*}

Recall that in \eqref{eq:def of F on RG} we defined $F(z)$ for $z\in \Pi_0(R_G)$. Concerning $L_0\setminus \Pi_0(R_G)$, by \lemref{lem:properties of Ji} (d) we have $L_0\setminus \Pi_0(R_G) = \bigcup_{i\in I}J_i = \bigcup_{i\in I}3J_i.$ Thus, for $z\in L_0\setminus \Pi_0(R_G)$ we may set
\begin{equation}\label{eq:extending F}
F(z)=\sum_{i\in I_0} \varphi_i(z) F_i(z).
\end{equation}

Using Lemmas \ref{lem:properties of Ji}--\ref{lem:Fi estimates}, one may follow the proofs of \cite[Lemma 7.24, Remark 7.26, Lemma 7.27]{tolsa2014analytic} to get the following.
\begin{lemma}\label{lem:properties of F}
	The function $F:L_0\rightarrow L_0^{\perp}$ is supported on $L_0\cap 1.9B_0$ and is $C\theta$-Lipschitz, where $C>0$ is an absolute constant. Furthermore, for $z\in 15J_i, i\in I,$
	\begin{equation}\label{eq:compare gradFi with grad F}
	|\nabla F(z)-\nabla F_i(z)|\lesssim\sqrt{\varepsilon_0},
	\end{equation}
	and
	\begin{equation*}
	|D^{2}F(z)|\lesssim \frac{\sqrt{\varepsilon_0}}{\ell(J_i)}.
	\end{equation*}
\end{lemma}

We denote the graph of $F$ as $\Gamma$, and we define a function $f: L_0\rightarrow\Gamma$ as
\begin{equation*}
f(x)=(x,F(x)).
\end{equation*}

We set also
\begin{equation*}
\sigma = \Hn{\Gamma}.
\end{equation*}

\begin{lemma}\label{lem:balls around f(Ji) in 1.5B0}
	Let $i\in I_0$. Then $B(f(z_{J_i}),2\diam(J_i))\subset 2.3B_0$.
\end{lemma}
\begin{proof}
	By the definition of $I_0$ we have $J_i\cap 1.5B_0\neq\varnothing$. We know by \lemref{lem:dist of Ji to B0} that $\diam(J_i)\le0.2r_0$, and so $z_{J_i}\in1.7B_0$. Moreover, since $F$ is supported on $L_0\cap 1.9B_0$ and is Lipschitz continuous with constant comparable to $\theta$, we have $\dist (f(z_{J_i}), z_{J_i})=|F(z_{J_i})|\lesssim\theta r_0$.
	
	It follows easily that $B(f(z_{J_i}),2\diam(J_i))\subset 2.3B_0$.
\end{proof}

We have defined a Lipschitz graph $\Gamma$, and a set $R_G\subset \Gamma\cap R_0$ such that $\restr{\mu}{R_G}\ll\mathcal{H}^n$. Clearly, measure $\restr{\mu}{R_G}$ is $n$-rectifiable. What remains to be shown is that $\mu(R_G)\ge 0.5\mu(R_0)$. Since $R_G$ contains $R_0\setminus\bigcup_{Q\in\Stop}Q$, it is enough to estimate the measure of the stopping cubes -- this is what we will do in the remaining part of the article.

\section{Small measure of cubes from \texorpdfstring{$\LD$}{LD}}\label{sec:LD}
In this section we will bound the measure of low density cubes. First, let us prove some additional estimates.
\subsection{\texorpdfstring{$\Gamma$}{Gamma} lies close to \texorpdfstring{$R_0$}{R0}}
\begin{lemma}\label{lem:x arbitrary close to Gamma}
	There exists a constant $C_1$ such that for any $x\in 3B_0$
	\begin{equation*}
	\dist(x,\Gamma)\le C_1\, d(x).
	\end{equation*}
\end{lemma}
\begin{proof}
	First, notice that if $x\in 3B_0\setminus 1.01B_0$, then $d(x)\gtrsim r_0$, and so the estimate $\dist(x,\Gamma)\le C_1\, d(x)$ is trivial. Now, assume $x\in 1.01B_0$.
	
	Let $\xi = \Pi_0(x)\in L_0,\ y=(\xi,F(\xi))\in\Gamma$. \lemref{lem:Lip property involving d(x)} gives us
	\begin{equation}\label{eq:dist x to Gamma estimated with dx and dy}
	\dist(x,\Gamma)\le |x-y| = |\Pi_0^{\perp}(x)-\Pi_0^{\perp}(y)|\lesssim d(x)+d(y).
	\end{equation}
	If $\xi\in \Pi_0(R_G)$, then $y\in R_G$, which means that $d(y)=0$ and we get $\dist(x,\Gamma) \lesssim d(x).$
	
	Now suppose $\xi\not\in \Pi_0(R_G)$. Let $i\in I$ be such that $\xi\in J_i$. Note that since $x\in 1.01B_0$, then by \eqref{eq:BR0 subset B0} $\xi\in 1.1B_0$, and so $J_i\cap 1.5B_0\not=\varnothing$. Hence, $i\in I_0$. Let $Q_i\in \Tree$ be the cube from \lemref{lem:Qi corresposnding to Ji} corresponding to $J_i$. It follows that
	\begin{equation}\label{eq:d(y) estimated with dist and ell}
	d(y)\le \dist(y, Q_i) + \ell(Q_i)\lesssim\dist(y,Q_i) + \ell(J_i).
	\end{equation}
	
	Now we will estimate $\dist(y,Q_i).$ Let $z = (\xi,F_i(\xi))\in L_{Q_i}$. We have
	\begin{multline*}
	|y-z| = |F(\xi) - F_i(\xi)| = \big\lvert \sum_{j\in I_0} \varphi_j(\xi)F_j(\xi) - F_i(\xi)\big\rvert = \big\lvert \sum_{j\in I_0} \varphi_j(\xi)(F_j(\xi) - F_i(\xi))\big\rvert\\
	\le \sum_{j\in I_0} \varphi_j(\xi) \big\lvert F_j(\xi) - F_i(\xi)\big\rvert.
	\end{multline*}
	Since $\varphi_j(\xi)\not= 0$ only for $j\in I_0$ such that $\xi\in 3J_j$, we get from \lemref{lem:Fi estimates} (b) that $\lvert F_j(\xi) - F_i(\xi)\rvert \lesssim \ell(J_i)$. Hence, 
	\begin{equation*}
	|y-z|\lesssim \ell(J_i).
	\end{equation*}
	
	We use the smallness of $\beta_{\mu,2}(3B_{Q_i})$ and Chebyshev inequality to find $p\in 2B_{Q_i},\ q\in L_{Q_i}$ such that $|p-q|\lesssim \ell(J_i)$. We know from \lemref{lem:Qi corresposnding to Ji} (b) that $|\Pi_0(p)-\xi|\lesssim\ell(J_i)$, and so $|\Pi_0(q)-\xi|\lesssim\ell(J_i).$ Together with the fact that both $q$ and $z$ belong to $L_{Q_i}$, and that $\measuredangle(L_0, L_{Q_i})\le\theta$ by \eqref{eq:notBS}, this implies
	\begin{equation*}
	|z-q|\lesssim\ell(J_i).
	\end{equation*}
	Thus,
	\begin{equation*}
	\dist(y, Q_i)\le |y-z|+|z-q|+|q-p|\lesssim\ell(J_i).
	\end{equation*}
	From this, \eqref{eq:d(y) estimated with dist and ell}, \lemref{lem:properties of Ji} (a), and the definition od $D$, we get
	\begin{equation*}
	d(y)\lesssim\ell(J_i)\approx D(\xi) \le d(x).		 
	\end{equation*}
	The estimate above together with \eqref{eq:dist x to Gamma estimated with dx and dy} conclude the proof.
\end{proof}
\begin{cor}\label{cor:doubling cubes close to Gamma}
	For every $Q\in\Tree$ we have
	\begin{equation*}
	\dist(Q,\Gamma)\lesssim \ell(Q).
	\end{equation*}
	Moreover, for $i\in I_0$ we have
	\begin{equation}\label{eq:Qi close to fJi}
	\dist(Q_i,f(J_i))\lesssim\ell(Q_i).
	\end{equation}
\end{cor}
\begin{proof}
	Since $Q\subset R_0\subset B_0$, the first inequality follows immediately by \lemref{lem:x arbitrary close to Gamma} and the definition of function $d$.
	
	The second inequality is implied by the first one, the fact that $\dist(\Pi_0(Q_i), J_i)\lesssim\ell(Q_i)$ by \lemref{lem:Qi corresposnding to Ji}, and that $\Gamma$ is a Lipschitz graph with a small Lipschitz constant.
\end{proof}

\begin{lemma}\label{lem:x from Gamma close to LQ}
	Let $C>0$. If $\varepsilon_0$ is chosen small enough, then for each $Q\in\Tree$ and $x\in \Gamma\cap CB_Q$
	\begin{equation*}
	\dist(x, L_Q)\lesssim_{A,\tau,C} \sqrt{\varepsilon_0} \ell(Q).
	\end{equation*}
\end{lemma}
\begin{proof}
	There are three cases to consider.
	\vspace{5pt}
	
	\textbf{Case 1.} $x\in R_G$, i.e. $d(x)=0$. 
	
	Fix some small $h>0$. Let $P\in \Tree$ be such that $(\dist(x,P)+\diam(B_P))\le h\ll\ell(Q)$. 
	Since $x\in \Gamma\cap CB_Q$, we have $\dist(P,Q)\lesssim \ell(Q)$. Setting $y=\Pi_{L_P}(x)$, we clearly have $|x-y|\lesssim h$, and in consequence $y\in L_P\cap C'B_{Q}$ with $C'\approx C$. Thus, we may apply \lemref{lem:planes close to each other for ancestors} to get
	\begin{equation*}
	\dist(y, L_Q)\lesssim_{A,\tau,C}\sqrt{\varepsilon_0}\ell(Q).
	\end{equation*}
	Thus, $\dist(x,L_Q)\lesssim_{A,\tau,C}\sqrt{\varepsilon_0}\ell(Q) + h$. Letting $h\rightarrow 0$ ends the proof in this case.
	\vspace{5pt}
	
	\textbf{Case 2.} $x=(\zeta, F(\zeta))$ for $\zeta\in L_0\setminus\Pi_0(R_G)$, and
	\begin{equation*}
	\sum_{i\in I_0}\varphi_i(\zeta) = 1.
	\end{equation*}
	Since $F(\zeta) = \sum_i \varphi_i(\zeta)F_i(\zeta)$, we get that $x$ is a convex combination of points $\{(\zeta, F_i(\zeta))\}_{i\in I_1}$, where $I_1\subset I_0$ consists of indices $i$ such that $\varphi_i(\zeta)\not = 0$. Thus, it suffices to show that for each $i\in I_1$
	\begin{equation*}
	\dist\big((\zeta, F_i(\zeta)), L_Q\big)\lesssim_{A,\tau,C}\sqrt{\varepsilon_0} \ell(Q).
	\end{equation*}
	First, note that since $x\in CB_Q$,
	\begin{equation*}
	D(\zeta)\le d(x)\lesssim_C \ell(Q).
	\end{equation*}
	Let $J_{i'}$ be the dyadic cube containing $\zeta$, $i'\in I_1$. Then
	\begin{equation}\label{eq:ell Jiprime smaller than ellQ}
	\diam(J_{i'})\le \frac{1}{20} D(\zeta)\lesssim_C \ell(Q).
	\end{equation}
	Moreover, as each $\varphi_i$ is supported in $3J_i$, we necessarily have $3J_i\cap J_{i'}\not= \varnothing$ for $i\in I_1$. Thus, by \lemref{lem:properties of Ji} (b) and by \lemref{lem:Qi corresposnding to Ji},
	\begin{equation}\label{eq:Qi smaller than Q}
	\ell(Q_{i'})\approx\ell(J_{i'})\approx \ell(J_i)\approx\ell(Q_i) \overset{\eqref{eq:ell Jiprime smaller than ellQ}}{\lesssim_C} \ell(Q).
	\end{equation}
	Furthermore, \lemref{lem:Fi estimates} (a) implies
	\begin{equation*}
	\dist(\Pi_0(Q_i),\Pi_0(Q_{i'}))\le \dist(Q_i,Q_{i'})\lesssim \ell(J_i).
	\end{equation*}
	Taking into account \lemref{lem:Qi corresposnding to Ji} and the fact that $\zeta\in J_{i'}\cap \Pi_0(CB_Q)$ we obtain
	\begin{equation*}
	\dist(\Pi_0(Q_{i'}),\Pi_0(Q))\le \dist(\Pi_0(Q_{i'}),J_{i'}) + \diam(J_{i'}) + \dist(\Pi_0(Q),J_{i'}) \lesssim_C \ell(J_{i'}) + \ell(Q).
	\end{equation*}
	The three estimates above yield
	\begin{align*}
	\dist(\Pi_0(Q_i),\Pi_0(Q))&\le\dist(\Pi_0(Q_i), \Pi_0(Q_{i'})) + \diam(\Pi_0(Q_{i'})) + \dist(\Pi_0(Q_{i'}),\Pi_0(Q))\\
	&\lesssim_{C} \ell(Q).
	\end{align*}
	Applying \lemref{lem:Lip property involving d(x)} to any $y_1\in Q_i, y_2\in Q$ gives us
	\begin{equation}\label{eq:Qi and Q are close}
	\dist(Q_i, Q)\lesssim_{A,\tau} \dist(\Pi_0(Q_i),\Pi_0(Q)) + \ell(Q) + \ell(Q_i) \lesssim_{C} \ell(Q).
	\end{equation}
	Note that $(\zeta, F_i(\zeta))\in L_{Q_i}\cap C'B_{Q_i}$ for some $C'=C'(n,d)>0$. Indeed:  $(\zeta, F_i(\zeta))\in L_{Q_i}$ by the definition of $F_i$; to see that $(\zeta, F_i(\zeta))\in  C'B_{Q_i}$ observe that $\varphi_i(\zeta)\not= 0$, and so $\zeta\in 3J_i$, which together with \lemref{lem:Qi corresposnding to Ji} gives $(\zeta, F_i(\zeta))\in  C'B_{Q_i}$.
	
	Due to the observation above and \eqref{eq:Qi smaller than Q}, \eqref{eq:Qi and Q are close}, we can use \lemref{lem:planes close to each other for ancestors} to get the desired inequality:
	\begin{equation*}
	\dist\big((\zeta,F_i(\zeta)), L_Q\big)\lesssim_{A,\tau,C} \sqrt{\varepsilon_0}\ell(Q).
	\end{equation*}
	\vspace{5pt}
	
	\textbf{Case 3.} $x=(\zeta, F(\zeta))$ for $\zeta\in L_0\setminus\Pi_0(R_G)$, and
	\begin{equation*}
	\sum_{i\in I_0}\varphi_i(\zeta) < 1.
	\end{equation*}
	It follows that there exists some $k\not\in I_0$ such that $\zeta\in 3J_k$. Hence, by \lemref{lem:dist of Ji to B0} (b)
	\begin{equation*}
	\ell(J_k)\approx\dist(\Pi_0(z_0),J_k)\gtrsim \ell(R_0).
	\end{equation*}
	Furthermore, if $J_{i'}$ is the cube containing $\zeta = \Pi_0(x)$, then using the definition of functions $d$ and $D$ yields
	\begin{equation*}
	\ell(J_{i'})\lesssim D(\Pi_0(x))\le d(x)\le \dist(x,Q) + \diam(B_Q) \lesssim \ell(Q)\le \ell(R_0).
	\end{equation*}
	Since $J_{i'}\cap 3J_k\not= \varnothing$, \lemref{lem:properties of Ji} (b) gives us $\ell(J_{i'})\approx\ell(J_k).$ Thus, 
	\begin{equation*}
	\ell(J_{i'})\approx\ell(Q)\approx\ell(R_0),
	\end{equation*}
	and again using  \lemref{lem:properties of Ji} (b) we get that $\ell(J_i)\approx\ell(R_0)$ for all $i\in I_1$, where $I_1\subset I_0$ are indices such that $\zeta\in 3J_i$. By the definition of cubes $Q_i$ in \lemref{lem:Qi corresposnding to Ji}, we also have $\ell(Q_i)\approx\ell(R_0)$.
	
	It is clear that $\dist(Q_i, R_0)=0$, and so the assumptions of \lemref{lem:planes close to each other for similar cubes} are satisfied for $Q_i$ and $R_0$. Since $\dist((\zeta, F_i(\zeta)), Q_i)\lesssim \ell(R_0)\approx\ell(Q_i)$, we get that
	\begin{equation*}
	|F_i(\zeta)| = \dist\big((\zeta, F_i(\zeta)), L_0\big)\lesssim_{A,\tau}\varepsilon_0\ell(R_0)\approx\varepsilon_0\ell(Q)
	\end{equation*}	
	for $i\in I_1$. Hence,
	\begin{equation*}
	\dist\big((\zeta, F(\zeta)), L_0\big) = |F(\zeta)|\le \sum_{i\in I_1} \varphi_i(\zeta)|F_i(\zeta)|\lesssim_{A,\tau} \varepsilon_0\ell(Q) \sum_{i\in I_1} \varphi_i(\zeta) \le \varepsilon_0\ell(Q).
	\end{equation*}
	At the same time, the planes $L_Q$ and $L_0$ are close to each other due to \lemref{lem:planes close to each other for similar cubes}, and so
	\begin{equation*}
	\dist\big((\zeta, F(\zeta)), L_Q\big)\lesssim_{A,\tau} \varepsilon_0\ell(Q).
	\end{equation*}
\end{proof}

\begin{cor}\label{cor:x from planes close to Gamma}
	Let $\theta$ and $\varepsilon_0$ be small enough. Suppose $Q\in\Tree$ satisfies $10B_Q\cap\Gamma\not= \varnothing$. Then for $y\in L_Q\cap 10B_Q$
	\begin{equation*}
	\dist(y,\Gamma)\lesssim_{A,\tau}\sqrt{\varepsilon_0}\ell(Q).
	\end{equation*}
\end{cor}
\begin{proof}
	
	Let $\widetilde{F}:L_Q\rightarrow L_Q^{\perp}$ be defined in such a way that $\Gamma$ is the graph of $\widetilde{F}$. 
	This definition makes sense because $\measuredangle(L_Q, L_0)\le \theta$. Moreover, $\lip(F)\lesssim \theta$ implies that $\lip(\widetilde{F})\lesssim \theta$.
	
	Let $x\in L_Q$ be such that $(x,\widetilde{F}(x))\in 10B_Q\cap\Gamma$. By the triangle inequality and \lemref{lem:x from Gamma close to LQ} we have for 
	$y\in L_Q\cap 10B_Q$
	\begin{equation*}
	|\widetilde{F}(y)|\le |\widetilde{F}(y)-\widetilde{F}(x)| + |\widetilde{F}(x)|\lesssim \theta\ell(Q) + C(A,\tau)\sqrt{\varepsilon_0}\ell(Q).
	\end{equation*}
	Thus, for $\theta$ and $\varepsilon_0$ small enough, we have $(y,\widetilde{F}(y))\in 11B_Q\cap\Gamma$ and we may use \lemref{lem:x from Gamma close to LQ} once again to conclude that
	\begin{equation*}
	\dist(y,\Gamma)\le |\widetilde{F}(y)| = \dist\big((y,\widetilde{F}(y)), L_Q \big) \lesssim_{A,\tau} \sqrt{\varepsilon_0}\ell(Q).
	\end{equation*}
\end{proof}
Recall that 
\begin{equation*}
\RFar = \{x\in 3B_0\ :\ \dist(x, L_Q)\ge\, \sqrt{\varepsilon_0}\ell(Q)\quad \text{for some $Q\in\Tree_0$ s.t. $x\in 3B_Q$} \}.
\end{equation*}
\begin{lemma}\label{lem:x outside RFar close to Gamma}
	For all $x\in 3B_0\setminus \RFar$
	\begin{equation*}
	\dist(x,\Gamma)\lesssim_{A,\tau} \sqrt{\varepsilon_0}\, d(x).
	\end{equation*}
\end{lemma}
\begin{proof}
	If $d(x)=0$, then $x\in R_G\subset\Gamma$ and we are done. Suppose that $d(x)>0$. Let $Q\in \Tree$ be such that
	\begin{equation*}
	\dist(x,Q)+\diam(B_Q)\le 2d(x).
	\end{equation*}
	Fix some $z\in Q$ and note that $|z-x|\le 2d(x)$.
	
	Let $C_1$ be the constant from \lemref{lem:x arbitrary close to Gamma}. 
	If we have $B(z, 2(C_1+2)d(x))\subset 3B_0$, then let $P\in \Tree$ be the smallest cube satisfying $B(z, 2(C_1+2)d(x))\subset 3B_{P}$; otherwise, set $P=R_0$.
	
	In both cases we have $\ell(P)\approx d(x)$, as well as $x\in 3B_P$. Moreover, we know from \lemref{lem:x arbitrary close to Gamma} that
	\begin{equation*}
	\dist(z,\Gamma)\le |z-x| + \dist(x,\Gamma)\le (2 + C_1)d(x).
	\end{equation*}
	Hence, $3B_P\cap\Gamma\neq \varnothing$ (for $P=R_0$ this is obvious, and for $P\subsetneq R_0$ it follows from the fact that $B(z, 2(C+2)d(x))\subset 3B_{P}$).
	
	The assumption $x\not\in \RFar$ gives us
	\begin{equation*}
	|x-\Pi_{L_P}(x)|\le \sqrt{\varepsilon_0}\ell(P),
	\end{equation*}
	and so $\Pi_{L_P}(x)\in 4B_P\cap L_P$. We apply Corollary \ref{cor:x from planes close to Gamma} to $\Pi_{L_P}(x)$ to get 
	\begin{equation*}
	\dist(\Pi_{L_P}(x), \Gamma)\lesssim_{A,\tau}\sqrt{\varepsilon_0}\ell(P).
	\end{equation*}
	The two inequalities above and the fact that $\ell(P)\approx d(x)$ imply 
	\begin{equation*}
	\dist(x,\Gamma)\lesssim_{A,\tau} \sqrt{\varepsilon_0}\, d(x).
	\end{equation*}		
\end{proof}
\begin{lemma}\label{lem:d and D comparable on graph}
	For every $x\in\Gamma$ we have $D(\Pi_0(x))\le d(x)\lesssim D(\Pi_0(x)).$
\end{lemma}
\begin{proof}
	The inequality $D(\Pi_0(x))\le d(x)$ follows directly from the definition of $D$ \eqref{eq:definition D}. 
	
	To see that $d(x)\lesssim D(\Pi_0(x))$, let $Q\in\Tree$ be such that
	\begin{equation}\label{eq:D(Pi(x)) in the prf of d(x)approx D(Pix)}
	\diam(B_Q)+\dist(\Pi_0(Q),\Pi_0(x))\le D(\Pi_0(x))+h
	\end{equation}
	for some small $h>0$. Take any $y\in 3B_Q\setminus \RFar$, then by \lemref{lem:x outside RFar close to Gamma} we have some $z\in\Gamma$ such that 
	\begin{equation*}
	\dist(y,\Gamma)=|y-z|\lesssim_{A,\tau}\sqrt{\varepsilon_0}d(y)\lesssim\sqrt{\varepsilon_0}\diam(B_Q)\overset{\eqref{eq:D(Pi(x)) in the prf of d(x)approx D(Pix)}}{\le} D(\Pi_0(x))+h.
	\end{equation*}
	Using the fact that $x,z\in\Gamma$, that $y\in 3B_Q$, and the inequality above, we have
	\begin{equation*}
	|x-z|\le 2|\Pi_0(x)-\Pi_0(z)|\le 2|\Pi_0(x)-\Pi_0(y)| + 2|\Pi_0(y)-\Pi_0(z)|\overset{\eqref{eq:D(Pi(x)) in the prf of d(x)approx D(Pix)}}{\lesssim} D(\Pi_0(x))+h,
	\end{equation*}
	and so
	\begin{equation*}
	|x-y|\le |x-z|+ |z-y| \lesssim D(\Pi_0(x))+h.
	\end{equation*}
	It follows that 
	\begin{equation*}
	d(x)\le d(y)+|x-y|\lesssim \diam(B_Q) + D(\Pi_0(x))+h \overset{\eqref{eq:D(Pi(x)) in the prf of d(x)approx D(Pix)}}{\lesssim} D(\Pi_0(x))+h.
	\end{equation*}
	Letting $h\rightarrow 0$ ends the proof.
\end{proof}
\subsection{Estimating the measure of \texorpdfstring{$\LD$}{LD}}
\begin{lemma}
	If $\varepsilon_0$ and $\tau$ are small enough, with $\varepsilon_0=\varepsilon_0(\tau)\ll\tau$, then
	\begin{equation}\label{eq:small measure of LD}
	\sum_{Q\in \LD} \mu(Q)\lesssim\tau\mu(R_0).
	\end{equation}
\end{lemma}
\begin{proof}
	Recall that by \lemref{lem:small measure of RFar} we have 
	$\mu(\RFar)\lesssim_{A,\tau}\sqrt{\varepsilon_0}\mu(R_0).$
	Hence, for $\varepsilon_0$ small enough we get 
	$ 
	\mu(\RFar)\le\tau\mu(R_0),
	$
	and so to show \eqref{eq:small measure of LD} it suffices to prove
	\begin{equation*}
	\mu(R_{\LD})\lesssim \tau\mu(R_0),
	\end{equation*}
	where $R_{\LD}=\bigcup_{Q\in\LD}Q\setminus\RFar$.
	
	We use Besicovitch covering theorem to find a countable collection of points $x_i\in R_{\LD}$ such that $x_i\in Q_i\setminus \RFar,\ Q_i\in \LD,$ and
	\begin{gather*}
	R_{\LD}\subset \bigcup_i B(x_i,r(Q_i)),\\
	\sum_i \one_{B(x_i, r(Q_i))}\le N,
	\end{gather*}
	where $N$ is a dimensional constant.
	
	Observe that $B(x_i, r(Q_i))\subset 1.5B_{Q_i}$. It follows that
	\begin{equation*}
	\mu(R_{\LD})\le \sum_i \mu(B(x_i, r(Q_i)))\le\sum_i \mu(1.5B_{Q_i})\lesssim \tau\sum_i r(Q_i)^n,
	\end{equation*}
	where the last inequality was obtained using the fact that $Q_i\in \LD$. Furthermore, since $x_i\not\in \RFar$ we may use \lemref{lem:x outside RFar close to Gamma} to get $\dist(x_i,\Gamma)\lesssim_{A,\tau} \sqrt{\varepsilon_0}\, d(x_i)$. Note also that $d(x_i)\lesssim r(Q_i)$. Hence,
	\begin{equation*}
	\dist(x_i,\Gamma)\lesssim_{A,\tau}\sqrt{\varepsilon_0} r(Q_i).
	\end{equation*}
	So, if $\varepsilon_0$ is small enough, $\Gamma$ passes close to the center of $B(x_i, r(Q_i))$. Since $\Gamma$ is a Lipschitz graph with small Lipschitz constant we get
	\begin{equation*}
	r(Q_i)^n\lesssim\mathcal{H}^n(\Gamma\cap B(x_i,r(Q_i))).
	\end{equation*}
	Thus,
	\begin{multline*}
	\mu(R_{\LD})\lesssim \tau\sum_i r(Q_i)^n\lesssim\tau\sum_i \mathcal{H}^n(\Gamma\cap B(x_i,r(Q_i)))\lesssim\tau \mathcal{H}^n\big(\Gamma\cap\bigcup_i B(x_i,r(Q_i))\big)\\
	\le \tau\mathcal{H}^n(\Gamma\cap 1.5B_0)\approx\tau \ell(R_0)^n.
	\end{multline*}
	We have $\ell(R_0)^n\approx \mu(3B_0)\approx \mu(R_0)$ because $\Theta_{\mu}(3B_0)=1$, see \remref{rem:WLOG density of BR0 is 1}, and $R_0$ is doubling. Hence,
	\begin{equation*}
	\mu(R_{\LD})\lesssim \tau\mu(R_0).
	\end{equation*}
\end{proof}

\section{Approximating measure \texorpdfstring{$\nu$}{nu}}\label{sec:approximating measure}
In order to estimate the measure of high density cubes, we need to introduce a measure $\nu$ supported on $\Gamma$ which will approximate $\mu$.
\subsection{Definition and properties of \texorpdfstring{$\nu$}{nu}}
Let $\eta<1/1000$ be a small dimensional constant which will be fixed in the proof of \lemref{lem:properties of Bk} (c). For every $i\in I$ (the set of indices from Section \ref{sec:extension of F}) consider a finite collection of points $\{z'_{k}\}_{k\in K_i}\subset J_i,\,\# K_i\lesssim_n 1,$ such that the balls $B(z'_{k},0.5\eta\ell(J_i))$ cover the whole $J_i$. We set $K=\bigcup_i K_i$, $K_0 = \bigcup_{i\in I_0} K_i$.

For $k\in K_i$ we define
\begin{align*}
z_k &= f(z_k')\in\Gamma,\\
r_k &=\eta\ell(J_i),\\
B_k &= B(z_{k}, r_k).
\end{align*}

The following lemma collects basic properties of $B_k$.
\begin{lemma}\label{lem:properties of Bk}
	We have the following:
	\begin{itemize}
		\item[(a)] For $k\in K_i$
		\begin{equation}\label{eq:Bk subset Ji}
		\Pi_0(3B_k)\subset 2J_i.
		\end{equation}
		\item[(b)] For $k\in K$ there exist at most $C=C(n)$ indices $k'\in K$ such that $\Pi_0(3B_k)\cap \Pi_0(3B_{k'})\not=\varnothing$ (in particular, there are at most $C$ indices $k'\in K$ such that $3B_k\cap 3B_{k'}
		\not=\varnothing$). Moreover, for all such $k'$ we have
		\begin{equation}\label{eq:radii comparable if Bk intersect}
		r_k\approx r_{k'}.
		\end{equation}
		\item[(c)] For $k\in K$ and $x\in 3B_k$ we have
		\begin{equation}\label{eq:rk comparable to d(x)}
		r_k\le d(x)\le \eta^{-3/2}r_k.
		\end{equation}
		\item[(d)] For $k\in K_0$
		\begin{equation}\label{eq:3Bk in 2.5B0}
		3B_k\subset 2.3B_0.
		\end{equation}
		\item[(e)] For $k\not\in K_0$
		\begin{equation}\label{eq:rk approx dist z0 for k not in K0}
		r_k\approx |z_k-z_0|\gtrsim\ell(R_0).
		\end{equation}
		If additionally $3B_k\cap 3B_0\not= \varnothing$, then
		\begin{equation}\label{eq:rk approx ell R0 for k not in K0}
		r_k \approx \ell(R_0).
		\end{equation}	
		\item[(f)] Finally,
		\begin{equation}\label{eq:Bk cover Gamma minus RG}
		\bigcup_{k\in K}B_k\cap\Gamma  = \bigcup_{k\in K}3B_k\cap\Gamma = \Gamma\setminus R_G.
		\end{equation}
	\end{itemize}
\end{lemma}
\begin{proof}
	(a) follows immediately by the definition of $B_k$.
	
	\vspace{5pt}
	Concerning (b), suppose $k\in K_i$ and $\Pi_0(3B_k)\cap \Pi_0(3B_{k'})\not=\varnothing$ for some $k'\in K_{i'}$. By (a) we know that $2J_i\cap 2J_{i'}\not= \varnothing$, and there are at most $N$ such indices $i'$, see \lemref{lem:properties of Ji}) (c). Since $\# K_i\lesssim_n 1$ by the definition, we get that there are at most $C(n,N)$ indices $k'$ satisfying $\Pi_0(3B_k)\cap \Pi_0(3B_{k'})\not=\varnothing$. The estimate $r_k\approx r_{k'}$ follows by \lemref{lem:properties of Ji} (b).
	\vspace{5pt}
	
	To prove (c), 
	recall that $z_k$ is the center of $B_k$. By the definition, $\Pi_0(z_k)\in J_i$ for $i\in I$ such that $r_k=\eta\ell(J_i)$. \lemref{lem:properties of Ji} (a) gives us $D(\Pi_0(z_k))\approx_{n} \ell(J_i)$. Hence, by \lemref{lem:d and D comparable on graph} we get  
	\begin{equation*}
	d(z_k)\approx \ell(J_i)=\eta^{-1} r_k.
	\end{equation*}
	Now, for an arbitrary $x\in 3B_k$ we have by the 1-Lipschitz property of function $d$ that 
	\begin{equation*}
	|d(x)- d(z_k)|\le |x-z_k|\le 3r_k,
	\end{equation*} 
	Since $d(z_k)\approx \eta^{-1} r_k,$ choosing $\eta$ small enough we arrive at $d(x)\approx \eta^{-1}r_k,$ and so for $\eta$ small enough $r_k\le d(x)\le \eta^{-3/2}r_k$.
	
	\vspace{5pt}
	Concerning (d), let $i\in I_0$ be such that $\Pi_0(z_k)\in J_i$. We know by \lemref{lem:balls around f(Ji) in 1.5B0} that $B(f(z_{J_i}),2\diam(J_i))\subset 2.3B_0$. Since  $3B_k\subset B(f(z_{J_i}),2\diam(J_i)),$ we get $3B_k\subset 2.3B_0$.
	
	\vspace{5pt}
	To show (e), let $k\in K\setminus K_0$. Let $i\in I\setminus I_0$ be such that $k\in K_i$, i.e. $\Pi_0(z_k)\in J_i$. By (c) and \lemref{lem:d and D comparable on graph} we have $d(z_k)\approx D(\Pi_0(z_k))\approx r_k$. At the same time, $|\Pi_0(z_k)-z_0|\approx \ell(J_i) \gtrsim\ell(R_0)$ by \lemref{lem:dist of Ji to B0} (b). Recall also that $\lVert F\rVert_{\infty}\lesssim \theta \ell(R_0)$ due to Lipschitz continuity and the fact that $\supp(F)\subset 1.9 B_0$, see \lemref{lem:properties of F}. It follows that
	\begin{equation*}
	|z_k-z_0|\le |z_k-\Pi_0(z_k)|+|\Pi_0(z_k)-z_0|\lesssim |F(z_k)| + \ell(J_i)\lesssim \theta\ell(R_0) + \ell(J_i)\lesssim \ell(J_i),
	\end{equation*}
	and on the other hand
	\begin{multline*}
	|z_k-z_0|\ge |\Pi_0(z_k)-z_0|-|z_k-\Pi_0(z_k)|\ge C\, \ell(J_i) - |F(z_k)|\ge C\, \ell(J_i) - C'\, \theta\ell(R_0)\\
	\ge C\, \ell(J_i) - C''\, \theta\ell(J_i)\gtrsim \ell(J_i),
	\end{multline*}
	for $\theta$ small enough. Hence, $|z_k-z_0|\approx \ell(J_i)\approx r_k\gtrsim\ell(R_0)$.
	
	Now, assume also $3B_k\cap 3B_0\not= \varnothing$, and suppose $x\in 3B_0\cap 3B_k$. We have $\Pi_0(x)\in 2J_i$ by \eqref{eq:Bk subset Ji}. Clearly, $D(\Pi_0(x))\le d(x)\lesssim\ell(R_0)$, and so $r_k\approx \ell(J_i)\approx D(\Pi_0(x))\lesssim \ell(R_0)$ by \lemref{lem:properties of Ji} (a).
	
	\vspace{5pt}
	Finally, to see \eqref{eq:Bk cover Gamma minus RG} note that by the definition of $B_k$ and by (a) we have 
	\begin{equation*}
	f(J_i)\subset\bigcup_{k\in K_i}B_k\cap\Gamma\subset\bigcup_{k\in K_i}3B_k\cap\Gamma\subset f(2J_i).
	\end{equation*}
	Together with  \lemref{lem:properties of Ji} (d) this implies \eqref{eq:Bk cover Gamma minus RG}.
\end{proof}
Since $\eta$ is a dimensional constant, we will usually not mention dependence on it in our further estimates.
%
%
	
Due to bounded superposition of $3B_k$ (\lemref{lem:properties of Bk} (b)) we may define a partition of unity $\{h_k\}_{k\in K}$ such that $0\le h_k\le 1,\ \supp h_k\subset 3B_k,\ \lip(h_k)\approx \ell(J_i)^{-1},$ and
\begin{equation}\label{eq:definition of h}
h=\sum_{k\in K}h_k \equiv 1\quad \text{on $\bigcup_{k\in K} 2B_k$}.
\end{equation}	
Again, by the bounded superposition of $3B_k$ we may assume 
\begin{equation}\label{eq:hk comparable to 1}
h_k(x)\approx 1,\quad x\in B_k.
\end{equation}
Recall that $\sigma=\Hn{\Gamma},$ and that $c_0$ is a constant minimizing $\alpha_{\mu}(3B_0)$. We set
\begin{equation}\label{eq:ck definition}
c_k =\begin{cases}
\frac{\int h_k\ d\mu}{\int h_k\ d\sigma}\quad &\text{for $k\in K_0$},\\
c_0\quad &\text{for $k\not\in K_0$}.
\end{cases}
\end{equation}
We define the approximating measure as
\begin{equation}\label{eq:definition of nu}
\nu = \restr{\mu}{R_G} + \sum_k c_k h_k \sigma.
\end{equation}
Note that, since $\restr{\mu}{R_G}\ll\sigma$ by \lemref{lem:mu on RG absolutely continuous wrto Hn}, we also have $\nu\ll\sigma$. To simplify the notation, we introduce
\begin{align*}
\mu_G &= \restr{\mu}{R_G},\\
\mu_B &= \mu - \mu_G,\\
\nu_B &= \nu - \mu_G = \sum_k c_k h_k d\sigma.
\end{align*}

Note that by \lemref{lem:properties of Bk} (d), \eqref{eq:Bk cover Gamma minus RG}, and the fact that $R_G\subset B_0$, we get 
\begin{equation*}
\Gamma\setminus( 2.3B_0)=L_0\setminus (2.3B_0)\subset \Gamma\cap\bigcup_{k\not\in K_0}B_k,
\end{equation*}
and so by the definition of $\nu$ we have
\begin{equation}\label{eq:nu flat outside 1.5B0}
\restr{\nu}{(2.3B_0)^c}=c_0\Hn{L_0\setminus (2.3B_0)}.
\end{equation}

\begin{lemma}
	For each $k\in K_0$ there exists $P_k\in\Tree$ such that $3B_k\subset 2.5B_{P_k}$, and $\ell(P_k)\approx r_k$.
\end{lemma}
\begin{proof}		
	We know by \eqref{eq:3Bk in 2.5B0} that $3B_k\subset 2.3B_0$. Thus, we may define $P_k$ as the smallest cube in $\Tree$ such that $3B_k\subset 2.5B_{P_k}$. We have $\ell(P_k)\approx r_k$ due to \eqref{eq:rk comparable to d(x)}.		
	%
	%
\end{proof}

We will write for $k\in K_0$
\begin{align}\label{eq:def of widetilde{B}_k}
\widetilde{B}_k &= 2.5B_{P_k},\\
\widetilde{c}_k&= c_{P_k},\notag\\
L_k &= L_{P_k},\notag
\end{align}
and for $k\not\in K_0$ set $\widetilde{B}_k=2.5B_0,\ \widetilde{c}_k=c_0$, and $L_k=L_0$.

Note that for every $k\in K$
\begin{equation}\label{eq:f(zk) close to Lk}
\dist(z_k,L_k)\lesssim_{A,\tau}\sqrt{\varepsilon_0}r_k.
\end{equation}
Indeed, for $k\in K_0$, it follows by \lemref{lem:x from Gamma close to LQ} applied to $z_k$ and $P_k$. For $k\not\in K_0$, but such that $z_k\in 1.9B_0$, again it follows by \lemref{lem:x from Gamma close to LQ} applied to $z_k$ and $R_0$. Finally, for $k\not\in K_0$ such that $z_k\not\in 1.9B_0$ this is trivially true because $\Gamma\setminus (1.9B_0)=L_0\setminus (1.9B_0)$, and so $\dist(z_k,L_0)=0$.


\begin{lemma}\label{lem:locally a good graph}
	For $k\in K$ the set $\Gamma\cap 3B_k$ is a Lipschitz graph over $L_k$, with a Lipschitz constant  at most $C\sqrt{\varepsilon_0}$.
\end{lemma}
\begin{proof}
	Suppose $k\in K_0$, i.e. that $k\in K_i$ for some $i\in I_0$. We know by \eqref{eq:compare gradFi with grad F} that $f(15J_i)$ is a $C\sqrt{\varepsilon_0}$-Lipschitz graph over $L_{Q_i}$ (recall that $F_i$ is an affine function whose graph is $L_{Q_i}$). 
	
	At the same time, since $P_k$ satisfies $\dist(P_k,Q_i)\lesssim \ell(Q_i)$ (see \eqref{eq:Qi close to fJi} and the definition of $P_k$) and $\ell(P_k)\approx r_k\approx \ell(Q_i),$ we can apply \lemref{lem:planes close to each other for similar cubes} to get $\measuredangle(L_{Q_i}, L_k)\lesssim_{A,\tau}\varepsilon_0$. It follows that $f(15J_i)$ is a $C\sqrt{\varepsilon_0}$-Lipschitz graph over $L_k$. The same is true for $k\not\in K_0$: since $L_k=L_0=\graph({F_i})$, it follows immediately by \eqref{eq:compare gradFi with grad F}.  We conclude by noting that
	\begin{equation*}
	\Gamma\cap 3B_k\overset{\eqref{eq:Bk subset Ji}}{\subset}f(15J_i).
	\end{equation*}
\end{proof}

\lemref{lem:locally a good graph} and \eqref{eq:f(zk) close to Lk} imply that for every $k\in K$
\begin{equation}\label{eq:sigma close to HnLk in 3Bk}
F_{3B_k}(\sigma, \Hn{L_k})\lesssim_{A,\tau}\sqrt{\varepsilon_0}r_k^{n+1}.
\end{equation}
Furthermore, by $\eqref{eq:sum of F containing x estimate}$ we have
\begin{equation}\label{eq:mu close to HnLk in tildeBk}
F_{\widetilde{B}_k}(\mu, \widetilde{c}_{k}\Hn{L_k})\lesssim_{A,\tau}\varepsilon_0 r_k^{n+1}.
\end{equation}

\begin{lemma}\label{lem:ck and ctildeBk close}
	For $k\in K$ we have 
	\begin{equation}\label{eq:ck and ctildeBk close}
	\left|c_k-\widetilde{c}_{k}\right|\lesssim_{A,\tau}\sqrt{\varepsilon_0}.
	\end{equation}
\end{lemma}
\begin{proof}
	For $k\not\in K_0$ we have $c_k=c_0=\widetilde{c}_k$, so the claim is trivially true. Suppose $k\in K_0$. Recall that $h_k\approx 1$ in $B_k$ \eqref{eq:hk comparable to 1}, $\lip(h_k)\approx r_k^{-1}$, and $\widetilde{c}_k\approx_{A,\tau}1$  by \eqref{eq:cBQ estimate}. It follows that
	\begin{multline*}
	\left\lvert c_k-\widetilde{c}_k\right\rvert r_k^n \overset{\eqref{eq:hk comparable to 1}}{\approx}\left\lvert c_k-\widetilde{c}_k\right\rvert \int h_k\ d\sigma = \left\lvert \int h_k\ d\mu - \int h_k \widetilde{c}_k\ d\sigma\right\rvert \\
	\le \left\lvert \int h_k\ d\mu - \int h_k \widetilde{c}_k\ d\Hn{L_k}\right\rvert + \widetilde{c}_k\left\lvert \int h_k\ d\Hn{L_k} - \int h_k \ d\sigma\right\rvert\\
	\le F_{\widetilde{B}_k}(\mu, \widetilde{c}_k\Hn{L_k})r_k^{-1} + \widetilde{c}_k F_{3B_k}(\sigma, \Hn{L_k})r_k^{-1}\overset{\eqref{eq:sigma close to HnLk in 3Bk},\eqref{eq:mu close to HnLk in tildeBk}}{\lesssim_{A,\tau}}\sqrt{\varepsilon_0}r_k^n.
	\end{multline*}
\end{proof}
An immediate corollary of \eqref{eq:cBQ estimate} and the lemma above is that for $k\in K$
\begin{equation}\label{eq:ck comparable to 1}
c_k\approx_{A,\tau}1.
\end{equation}
\begin{lemma}\label{lem:nu is AD regular}
	The measure $\nu$ is $n$-AD-regular, that is, for $x\in\Gamma,\ r>0$
	\begin{equation*}
	\nu(B(x,r))\approx_{A,\tau}r^n
	\end{equation*}
\end{lemma}
\begin{proof}
	We know by
	\eqref{eq:Bk cover Gamma minus RG}, the definition of $h$ \eqref{eq:definition of h}, and
	\eqref{eq:ck comparable to 1} that 
	\begin{equation*}
	d\restr{\sigma}{\Gamma\setminus R_G} = \sum_k h_kd\sigma\approx_{A,\tau} \sum_k c_k h_kd\sigma.
	\end{equation*}
	Together with \lemref{lem:mu on RG absolutely continuous wrto Hn} this gives
	\begin{equation*}
	d\nu = d\mu_G + \sum_k c_k h_k d\sigma\approx_{A,\tau} d\sigma.
	\end{equation*}
\end{proof}
\begin{lemma}\label{lem:ck close to cj for close Bk and Bj}
	If $k,j\in K$ satisfy $3B_k\cap 3B_j\not = \varnothing$, then
	\begin{equation*}
	|c_k - c_j|\lesssim_{A,\tau}\sqrt{\varepsilon_0}.
	\end{equation*}
\end{lemma}
\begin{proof}
	If  $3B_k\cap 3B_j\not = \varnothing$, then by \eqref{eq:Bk subset Ji} and \lemref{lem:properties of Ji} (b) it follows that 
	\begin{equation*}
	r_k\approx r_j.
	\end{equation*}		
	%
	
	Now, since $3B_k\cap 3B_j\not = \varnothing$ and $r_k\approx r_j$, we get that there exists $R\in\Tree$ such that $2.5B_R\supset \widetilde{B}_k\cup \widetilde{B}_j$ and $\ell(R)\approx r_k$. Hence, we may use \lemref{lem:cBQ similar for similar cubes} and \lemref{lem:ck and ctildeBk close} to obtain
	\begin{equation*}
	|c_k - c_j| \le |c_k - \widetilde{c}_k| + |\widetilde{c}_k - c_R| + |c_R - \widetilde{c}_j| + | \widetilde{c}_j - c_j|\lesssim_{A,\tau}\sqrt{\varepsilon_0}.
	\end{equation*}
\end{proof}

\subsection{\texorpdfstring{$\nu$}{nu} approximates \texorpdfstring{$\mu$}{mu} well}
\begin{lemma}\label{lem:h=1 outside RFar}
	We have 
	\begin{equation*}
	3B_0\setminus(R_G\cup \RFar)\subset \bigcup_{k\in K} 2B_k.
	\end{equation*}
	In consequence, for every $x\in 3B_0\setminus(R_G\cup \RFar)$ we have $h(x)=1.$
\end{lemma}
\begin{proof}
	Let $x\in 3B_0\setminus(R_G\cup \RFar)$. We will find $k\in K$ such that $x\in 2B_k$.
	
	By \lemref{lem:x outside RFar close to Gamma} we have $y\in \Gamma$ such that 
	\begin{equation} \label{eq:x close to y in prf of h=1 outside RFar}
	|x-y|\lesssim_{A,\tau}\sqrt{\varepsilon_0}d(x).
	\end{equation}
	Since $x\not\in R_G$, we have $d(x)>0$. Moreover, since $d(x)\le d(y)+|x-y|\le d(y) + 0.5 d(x)$, we get that $0<d(x)\le 2d(y)$. In particular, $y\not\in R_G$ and by \eqref{eq:Bk cover Gamma minus RG} there exists $k\in K$ such that $y\in B_k\cap\Gamma$. It follows by \lemref{lem:properties of Bk} (c) that
	\begin{equation*}
	d(x)\le 2d(y)\approx r_k.
	\end{equation*}
	Together with \eqref{eq:x close to y in prf of h=1 outside RFar} this gives $|x-y|\le r_k/2$, for $\varepsilon_0$ small enough. Since $y\in B_k$, we get that $x\in 2B_k$.
	
\end{proof}
\begin{lemma}
	Suppose that $x\in 2.5B_0$, $r\ge Cd(x)$ for some $C>0$, and that $B(x,r)\subset 3B_0$. Then,
	\begin{equation*}
	F_{B(x,r)}(\mu_B, h\mu)\lesssim_{A,C}\varepsilon_0^{1/4}r^{n+1}.
	\end{equation*}
\end{lemma}
\begin{proof}
	Since $B(x,r)\subset 3B_0$, and $r\ge Cd(x)$, there exists a cube $Q\in \Tree$ such that $B(x,r)\subset 3B_{Q}$ and $\ell(Q)\approx_C r$. In consequence, using the properties of $\Tree$ yields
	\begin{equation*}
	\mu(B(x,r)\cap\RFar)\le \mu(3B_{Q}\cap \RFar)\overset{\eqref{eq:notF}}{\le}\varepsilon_0^{1/4}\mu(3B_{Q})\overset{\eqref{eq:notHD}}{\lesssim_{A,C}}\varepsilon_0^{1/4}r^n.
	\end{equation*}
	Thus, given any $\phi\in\lip_1(B(x,r))$ we have
	\begin{equation*}
	\left\lvert \int\phi\ d\mu_B - \int\phi\ d\restr{\mu_B}{(\RFar)^c}\right\rvert \le r\mu (B(x,r)\cap \RFar)\lesssim_{A,C}\varepsilon_0^{1/4}r^{n+1},
	\end{equation*}
	and so $F_{B(x,r)}(\mu_B, \restr{\mu_B}{(\RFar)^c})\lesssim_{A,C}\varepsilon_0^{1/4}r^{n+1}.$ Similarly, $F_{B(x,r)}(h\mu, h\restr{\mu}{(\RFar)^c})\lesssim_{A,C}\varepsilon_0^{1/4}r^{n+1}.$ 
	
	Now, observe that $h\mu = h\mu_B$ by the definition of $h$. Moreover, inside $B(x,r)$ we have 
	\begin{equation*}
	h\restr{\mu_B}{(\RFar)^c} = \restr{\mu_B}{(\RFar)^c}
	\end{equation*}
	because $h\equiv1$ on $3B_0\setminus(R_G\cup \RFar)$ by \lemref{lem:h=1 outside RFar}. Thus, the triangle inequality yields 
	\begin{equation*}
	F_{B(x,r)}(\mu_B, h\mu)\le F_{B(x,r)}(\mu_B, \restr{h\mu}{(\RFar)^c}) + F_{B(x,r)}(h\mu, h\restr{\mu}{(\RFar)^c})\lesssim_{A,C}\varepsilon_0^{1/4}r^{n+1}.	
	\end{equation*}
\end{proof}
%
\begin{lemma}\label{lem:nuB and hmu close}
	If $x\in 2.5B_0$ and $r>0$ satisfy $B(x,r)\subset 2.5B_0$, then
	\begin{equation}\label{eq:nuB and hmu close}
	F_{B(x,r)}(\nu_B,h\mu)\lesssim_{A,\tau}\sqrt{\varepsilon_0}\sum_{3B_k\cap B(x,r)\not=\varnothing} r_k^{n+1}.
	\end{equation}
\end{lemma}
\begin{proof}
	Since $\nu_B =\sum_k c_k h_k \sigma$, our aim is estimating $F_{B(x,r)}(\sum_k c_k h_k \sigma,h\mu)$.
	Set
	\begin{equation*}
	K(x,r)=\{k\in K\ :\ 3B_k\cap B(x,r)\not=\varnothing\}.
	\end{equation*}
	
	First, we will deal with $k\in K(x,r)\setminus K_0$. For such $k$ by \eqref{eq:rk approx ell R0 for k not in K0} we have 
	\begin{equation}\label{eq:rk approx ellR0 outside K0}
	r_k\approx \ell(R_0).
	\end{equation} 
	
	In particular, $r\lesssim r_k,$ and so given $\phi\in\lip_1(B(x,r))$ we have $\lip(\phi h_k)\lesssim 1,\ \supp(\phi h_k)\subset B(x,r)\cap 3B_k\subset 2.5B_0$. 
	
	Moreover, recall that
	\begin{equation}\label{eq:mu close to HnL0 in 5BR0}
	F_{2.5B_0}(\mu,c_0\Hn{L_0}) \overset{\eqref{eq:sum of F containing x estimate},\eqref{eq:notHD}}{\lesssim_{A}}\varepsilon_0\ell(R_0)^{n+1}\overset{\eqref{eq:rk approx ellR0 outside K0}}{\approx} \varepsilon_0 r_k^{n+1}.
	\end{equation}
	In consequence, since $c_k=c_0$ by \eqref{eq:ck definition}, we have for any $\phi\in\lip_1(B(x,r))$
	\begin{multline*}
	\left|\sum_{k\in K(x,r)\setminus K_0}\left( \int\phi h_k c_0 \ d\sigma - \int\phi h_k\ d\mu\right)\right|\\
	\le \sum_{k\in K(x,r)\setminus K_0}\left( \left| \int\phi h_k c_0 \ d\Hn{L_0} - \int\phi h_k\ d\mu\right|+ c_0\left| \int\phi h_k \ d\Hn{L_0} - \int\phi h_k\ d\sigma\right|\right)\\
	\le \sum_{k\in K(x,r)\setminus K_0} \left( F_{2.5B_0}(\mu, c_0\Hn{L_0}) + c_0F_{3B_k}( \sigma, \Hn{L_0})\right) \\
	\overset{\eqref{eq:mu close to HnL0 in 5BR0},\eqref{eq:sigma close to HnLk in 3Bk}, \eqref{eq:cBQ estimate}}{\lesssim_{A,\tau}}\sum_{k\in K(x,r)\setminus K_0} \sqrt{\varepsilon_0} r_k^{n+1}.
	\end{multline*}
	
	Now, we turn our attention to $k\in K_0(x,r)=K(x,r)\cap K_0$. For any $\phi\in\lip_1(B(x,r))$ we have
	\begin{multline*}
	\left|\sum_{k\in K_0(x,r)}\left( \int\phi c_k h_k\ d\sigma - \int\phi h_k\ d\mu\right)\right|\\
	\le \left|\sum_{k\in K_0(x,r)}\left( \int(\phi-\phi(z_k)) c_k h_k\ d\sigma -\int(\phi-\phi(z_k)) h_k\ d\mu\right)\right| \\
	+ \left| \sum_{k\in K_0(x,r)}\phi(z_k)\left(\int c_k h_k\ d\sigma - \int h_k\ d\mu\right)\right| =: I_1 + I_2.
	\end{multline*}
	We start by estimating $I_1$. Observe that setting $\Phi_k=(\phi-\phi(z_k))h_k$ we have $\lip(\Phi_k)\lesssim 1$ and $\supp\Phi_k\subset 3B_k$. Hence,
	\begin{multline*}
	I_1 = \left|\sum_{k\in K_0(x,r)}\left( \int c_k \Phi_k\ d\sigma -\int\Phi_k\ d\mu\right)\right|\\
	\overset{\eqref{eq:sigma close to HnLk in 3Bk},\eqref{eq:ck comparable to 1}}{\le} \sum_{k\in K_0(x,r)}\left(\left | \int c_k \Phi_k d\Hn{L_k} -\int\Phi_k\ d\mu\right| + C(A,\tau)\sqrt{\varepsilon_0}r_k^{n+1} \right)\\
	\overset{\eqref{eq:ck and ctildeBk close}}{\le} \sum_{k\in K_0(x,r)}\left(\left | \int \widetilde{c}_k \Phi_k d\Hn{L_k} -\int\Phi_k\ d\mu\right| + C(A,\tau)\sqrt{\varepsilon_0}r_k^{n+1} \right)\\
	\overset{\eqref{eq:mu close to HnLk in tildeBk}}{\lesssim_{A,\tau}}\sum_{k\in K_0(x,r)}\sqrt{\varepsilon_0}r_k^{n+1}.
	\end{multline*}
	
	Concerning $I_2$, note that for $k\in K_0(x,r)$ we have by the definition of $c_k$ \eqref{eq:ck definition}
	\begin{equation*}
	\int c_k h_k\ d\sigma - \int h_k\ d\mu = 0,
	\end{equation*}
	and so 
	\begin{equation*}
	I_2 = 0.
	\end{equation*}
	Putting together the estimates for $k\in K(x,r)\setminus K_0$ and for $k\in K_0(x,r)$, and taking supremum over $\phi\in\lip_1(B(x,r))$, we finally get
	\begin{equation*}
	F_{B(x,r)}(\nu_B,h\mu)\lesssim_{A,\tau}\sqrt{\varepsilon_0}\sum_{k\in K(x,r)} r_k^{n+1}.
	\end{equation*}
\end{proof}
The previous two lemmas, and the fact that $F_B(\nu,\mu) = F_B(\nu_B,\mu_B)$, imply the following:
\begin{lemma}
	For $x\in 2.5B_0$ and $r\gtrsim d(x)$ such that $B(x,r)\subset 2.5B_0$ we have
	\begin{equation}\label{eq:nu close to mu}
	F_{B(x,r)}(\nu,\mu)\lesssim_{A,\tau}\varepsilon_0^{1/4}r^{n+1}+\sqrt{\varepsilon_0}\sum_{3B_k\cap B(x,r)\not=\varnothing} r_k^{n+1}.
	\end{equation}
\end{lemma}
In particular, we have
\begin{equation}\label{eq:nu close to mu in 5BR0}
F_{2.5B_0}(\nu,\mu)\lesssim_{A,\tau}\varepsilon_0^{1/4}\ell(R_0)^{n+1}.
\end{equation}
\begin{lemma}
	For $x\in\Gamma$ and $r\gtrsim \ell(R_0)$ such that $B(x,r)\cap 3B_0\not=\varnothing$ we have
	\begin{equation}\label{eq:nu close to flat on big balls}
	F_{B(x,r)}(\nu,c_0\Hn{L_0})\lesssim_{A,\tau} \varepsilon_0^{1/4} r\ell(R_0)^n.
	\end{equation}
\end{lemma}
\begin{proof}
	Recall that by \eqref{eq:nu flat outside 1.5B0} we have
	\begin{equation*}
	\restr{\nu}{(2.3B_0)^c}=c_0\Hn{L_0\cap (2.3B_0)^c}.
	\end{equation*}
	To take advantage of this equality, we define an auxiliary function $\psi$ such that $\psi\equiv 1$ on $2.3B_0$, $\supp(\psi)\subset 2.5B_0$, and $\lip(\psi)\lesssim \ell(R_0)^{-1}$. Then,
	\begin{multline}\label{eq: dist of nu to flat}
	\left| \int\psi\ d\nu - c_0\int\psi\ d\Hn{L_0}\right|\\
	\overset{\eqref{eq:nu close to mu in 5BR0}}{\lesssim}\left| \int\psi\ d\mu - c_0\int\psi\ d\Hn{L_0}\right| + \varepsilon_0^{1/4}\ell(R_0)^{n}\\
	\overset{\eqref{eq:sum of F containing x estimate}}{\le} C(A,\tau) \varepsilon_0\ell(R_0)^{n} + \varepsilon_0^{1/4}\ell(R_0)^{n}\lesssim \varepsilon_0^{1/4} \ell(R_0)^{n}.
	\end{multline}
	
	Recall that $z_0=z_{R_0}$. It follows that for $\phi\in\lip_1(B(x,r))$ we have
	\begin{multline*}
	\left\lvert\int\phi\ (d\nu-c_0d\Hn{L_0})\right\rvert = \left\lvert \int \left( (\phi-\phi(z_0))\psi + \phi(z_0)\psi + \phi(1-\psi)\right)(d\nu -c_0d\Hn{L_0})\right\rvert\\
	\overset{\eqref{eq:nu flat outside 1.5B0}}{\le}F_{2.5B_0}(\nu, \mu)+ F_{2.5B_0}(c_0\Hn{L_0}, \mu) + |\phi(z_0)|	\left\lvert\int\psi\ (d\nu-c_0d\Hn{L_0})\right\rvert + 0\\
	\overset{\eqref{eq:nu close to mu in 5BR0},\eqref{eq:sum of F containing x estimate}}{\lesssim_{A,\tau}}
	\varepsilon_0^{1/4}\ell(R_0)^{n+1} + \varepsilon_0\ell(R_0)^{n+1}+ |\phi(z_0)|	\left\lvert\int\psi\ (d\nu-c_0d\Hn{L_0})\right\rvert\\
	\overset{\eqref{eq: dist of nu to flat}}{\lesssim} \varepsilon_0^{1/4}\ell(R_0)^{n+1} + r\varepsilon_0^{1/4}\ell(R_0)^{n}\lesssim \varepsilon_0^{1/4} r\ell(R_0)^n.
	\end{multline*}
\end{proof}

\section{Small measure of cubes from \texorpdfstring{$\HD$}{HD}}\label{sec:HD}
For brevity of notation let us denote by $\Pi_*\nu$ the image measure of $\nu$ by $\Pi_0$, that is the measure such that $\Pi_*\nu(A) = \nu(\Pi_0^{-1}(A))$. Set
\begin{equation*}
f=\frac{d\Pi_*\nu}{d\Hn{L_0}}.
\end{equation*}
The key estimate necessary to bound the measure of high density cubes is the following.
\begin{lemma}\label{lem:L2 norm estimate of nu density}
	We have 
	\begin{equation}\label{eq:L2 norm estimate of nu density}
	\lVert f-c_0\rVert_{L^2(\Hn{L_0})}^2 \lesssim_{A,\tau} \varepsilon_0^{1/8}\mu(R_0).
	\end{equation}
\end{lemma}
We postpone the proof of the above lemma to the next subsection. Let us show now how we can use it to estimate the measure of cubes in $\HD$.
\begin{lemma}
	We have
	\begin{equation}\label{eq:small measure of HD}
	\sum_{Q\in\HD}\mu(Q)\lesssim_{A,\tau}\varepsilon_0^{1/8}\mu(R_0).
	\end{equation}
\end{lemma}
\begin{proof}
	Recall that by \lemref{lem:small measure of RFar} we have 
	$\mu(\RFar)\lesssim_{A,\tau}\sqrt{\varepsilon_0}\mu(R_0).$
	Thus, to show \eqref{eq:small measure of HD} it suffices to prove
	\begin{equation*}
	\mu(R_{\HD})\lesssim \varepsilon_0^{1/8}\mu(R_0),
	\end{equation*}
	where $R_{\HD}=\bigcup_{Q\in\HD}Q\setminus\RFar$.
	
	For every $x\in R_{\HD}$ we define $B_x = B(x, r(Q_x)/100)$, where $Q_x\in \HD$ is such that $x\in Q_x$. We use the $5r$-covering theorem to choose $\{x_j\}_{j\in J}$ such that all $B_{x_j}$ are pairwise disjoint and $\bigcup_j 5B_{x_j}$ covers $\bigcup_{x\in R_{\HD}} B_x$. Observe that $5B_{x_j}\subset 3B_{Q_{x_j}}$, and so by \eqref{eq:notHD}
	\begin{equation}\label{eq:mass of 5Bxj bounded}
	\mu(5B_{x_j})\lesssim_A r(B_{x_j})^n.
	\end{equation} 
	
	For every $j$ set $B_j = \frac{1}{2}B_{x_j}, Q_j = Q_{x_j}$, and let $P_j\in\Tree$ be the parent of $Q_j$. We have $\ell(P_j)\approx\ell(Q_j)\approx r(B_j).$ Since $x_j\not\in\RFar$, we can use \lemref{lem:x outside RFar close to Gamma} to obtain
	\begin{equation*}
	\dist(x_j,\Gamma)\lesssim_{A,\tau}\sqrt{\varepsilon_0}d(x_j)\lesssim\sqrt{\varepsilon_0}\ell(P_j)\approx_{A,\tau}\sqrt{\varepsilon_0} r(B_j).
	\end{equation*}
	Since $2B_j$ are disjoint, the centers of $B_j$ are close to $\Gamma$, and $\Gamma$ is a graph of function $F$ with $\lip(F)\lesssim\theta\ll 1$, it follows that $\Pi_0(B_j)$ are disjoint as well.
	
	%
	We use the above to get
	\begin{equation}\label{eq:estimating RHD}
	\mu(R_{\HD})\le \sum_{j\in J}\mu(5B_{x_j})\overset{\eqref{eq:mass of 5Bxj bounded}}{\lesssim_{A}}\sum_{j\in J}r(B_{x_j})^n \\
	\approx \sum_{j\in J}\mathcal{H}^n(\Pi_0(B_j))= \mathcal{H}^n\big(\bigcup_{j\in J}\Pi_0(B_j)\big).
	\end{equation}
	We claim that 
	\begin{equation}\label{eq:subset of BM}
	\bigcup_{j\in J}\Pi_0(B_j)\subset \mathcal{BM},
	\end{equation}
	where 
	\begin{equation*}
	\mathcal{BM} = \{ x\in L_0\ : \ \mathcal{M}(f-c_0)>1 \},
	\end{equation*}
	and $\mathcal{M}$ is the Hardy-Littlewood maximal function on $L_0$. $\mathcal{BM}$ stands for ``big $\mathcal{M}$''. Before we prove \eqref{eq:subset of BM}, note that due to the weak type $(2,2)$ estimate for $\mathcal{M}$ we have
	\begin{equation*}
	\mathcal{H}^n(\mathcal{BM})\lesssim \lVert f-c_0\rVert^2_{L^2(\Hn{L_0})}.
	\end{equation*}
	Putting this together with \eqref{eq:estimating RHD}, \eqref{eq:subset of BM}, and our key estimate from \lemref{lem:L2 norm estimate of nu density}, we get that
	\begin{equation*}
	\mu(R_{\HD})\lesssim_{A,\tau}\varepsilon_0^{1/8}\mu(R_0).
	\end{equation*}
	Therefore, all that remains is to show \eqref{eq:subset of BM}.
	
	Let $j\in J,\ y\in \Pi_0(B_j)$. Since $|y-\Pi_0(x_j)|\le r(B_j)\le r(B_{Q_j})$ and $\Pi_0(x_j)\in \Pi_0(B_{Q_j})$, we have $B(y,25r(B_{Q_j}))\supset \Pi_0(10B_{Q_j})$. Clearly, for some $C=C(n)>0$
	\begin{multline}\label{eq:MF estimate}
	\mathcal{M}(f-c_0)(y)\ge \frac{C}{r(B_{Q_j})^n}\Pi_*\nu(B(y,25r(B_{Q_j}))) - c_0\\
	\ge \frac{C}{r(B_{Q_j})^n}\Pi_*\nu(\Pi_0(10B_{Q_j})) - c_0\ge \frac{C}{r(B_{Q_j})^n}\nu(10B_{Q_j}) - c_0.
	\end{multline}		
	Recall that by \eqref{eq:c smaller then mu over Hn} and \remref{rem:WLOG density of BR0 is 1} we have
	\begin{equation*}
	c_0\lesssim 1.
	\end{equation*}
	Thus, 
	if we show that $\nu(10B_{Q_j})\gtrsim Ar(B_{Q_j})^n,$ for $A$ big enough we will have $\mathcal{M}(f-c_0)(y)>1$, and so we will be done.
	
	Let us define
	\begin{equation*}
	\lambda(z)=(r(10B_{Q_j})-|z-z_{Q_j}|)_+.
	\end{equation*}
	Note that $\lambda$ is $1$-Lipschitz and that $\supp(\lambda)\subset 10B_{Q_j}\subset 2.5B_0$. Moreover,
	\begin{equation*}
	7r({B_{Q_j}})\one_{3B_{Q_j}}\le \lambda\le 10r({B_{Q_j}})\one_{10B_{Q_j}}.
	\end{equation*}
	Note that $r(B_{Q_j})\gtrsim d(z_{Q_j})$. We get that
	\begin{align*}
	r({B_{Q_j}})\nu(10B_{Q_j})\gtrsim& \int \lambda(z)\ d\nu(z)\\
	\overset{\eqref{eq:nu close to mu}}{\ge}& \int \lambda(z)\ d\mu(z) - C(A,\tau)\big( \varepsilon_0^{1/4}r({B_{Q_j}})^{n+1} + \varepsilon_0^{1/2} \sum_{3B_k\cap 10B_{Q_j}\not=\varnothing} r_k^{n+1}\big)\\
	\ge& \ 7r({B_{Q_j}}) \mu(3B_{Q_j}) - C(A,\tau)\big( \varepsilon_0^{1/4}r({B_{Q_j}})^{n+1} + \varepsilon_0^{1/2} \sum_{3B_k\cap 10B_{Q_j}\not=\varnothing} r_k^{n+1}\big).
	\end{align*}
	Note that for all $k$ such that $3B_k\cap 10B_{Q_j}\not=\varnothing$ we have $r_k\lesssim_{A,\tau} r(B_{Q_j})$. Indeed, for $x\in 10B_{Q_j}$ it holds that $d(x)\lesssim_{A,\tau} r(B_{Q_j})$, and for $x\in 3B_k$ we have $r_k\le d(x)$ by \lemref{lem:properties of Bk} (c). Moreover, since the balls $\Pi_0(3B_k)$ are of bounded intersection by \lemref{lem:properties of Bk} (b), we get
	\begin{equation*}
	\sum_{3B_k\cap 10B_{Q_j}\not=\varnothing} r_k^{n}\le \sum_{\Pi_0(3B_k)\cap \Pi_0(10B_{Q_j})\not=\varnothing} r_k^{n}\lesssim_{A,\tau} r(B_{Q_j})^n.
	\end{equation*}
	Hence, using the above and the fact that $Q_j\in\HD$
	\begin{align*}
	r({B_{Q_j}})\nu(10B_{Q_j})\gtrsim& \ r({B_{Q_j}}) \mu(3B_{Q_j}) - C(A,\tau)\varepsilon_0^{1/4} r(B_{Q_j})^{n+1}\\
	\gtrsim& Ar({B_{Q_j}})^{n+1} - C(A,\tau)\varepsilon_0^{1/4} r(B_{Q_j})^{n+1} \gtrsim Ar({B_{Q_j}})^{n+1},
	\end{align*}
	for $\varepsilon_0$ small enough. Thus, $\nu(10B_{Q_j})\gtrsim Ar({B_{Q_j}})^n$ and by \eqref{eq:MF estimate} we get $\mathcal{M}(f-c_0)(y)>1$.
\end{proof}
\subsection{\texorpdfstring{$\Lambda$}{Lambda}-estimates}
The aim of this section is to prove the crucial estimate from \lemref{lem:L2 norm estimate of nu density}, i.e.
\begin{equation}\label{eq:L2 norm estimate of nu density reiterated}
\lVert f-c_0\rVert_{L^2(\Hn{L_0})}^2 \lesssim_{A,\tau} \varepsilon_0^{1/8}\mu(R_0).
\end{equation}
From now on we will denote by $\phi:\R^d\rightarrow\R$ a radial $C^{\infty}$ function such that $\phi\equiv 1$ on $B(0,1/2)$, $\supp(\phi)\subset B(0,1)$, and
\begin{equation*}
\phi_r(x)=r^{-n}\phi\left(\frac{x}{r}\right).
\end{equation*}
We also set 
\begin{equation*}
\psi_r(x) = \phi_{r}(x)-\phi_{2r}(x).
\end{equation*}


A classical result of harmonic analysis (see \cite[Sections I.6.3, I.8.23]{stein1993harmonic}) states that
\begin{equation}\label{eq:classical harmonic analysis result}
\lVert f-c_0\rVert_{L^2(\Hn{L_0})}^2\approx \int_{L_0}\int_0^{\infty}|\psi_r\ast \Pi_*\nu(z)|^2\frac{dr}{r}\ d\mathcal{H}^n(z),
\end{equation}
and so we will work with the latter expression.

For $r>0$ and $x\in\R^d$ let us define
\begin{equation*}
\widetilde{\psi}_r(x) = \psi_r\circ \Pi_0(x)\cdot \phi\left(\frac{x}{5r}\right).
\end{equation*}	
Given a measure $\lambda$ on $\R^d$ we set
\begin{equation*}
\Lambda_{\lambda}(x,r) = \left\lvert \widetilde{\psi}_r\ast\lambda(x)\right\rvert.
\end{equation*}

\begin{lemma}
	We have for all $x\in\Gamma$
	\begin{equation}
	\Lambda_{\nu}(x,r) = |\psi_r\ast \Pi_*\nu(\Pi_0(x))|.
	\end{equation}
\end{lemma}
\begin{proof}
	By the definition of $\Lambda_{\nu}$, it suffices to show that for all $x,y\in\Gamma$ we have 
	\begin{equation*}
	\widetilde{\psi}_r(x-y) = \psi_r(\Pi_0(x)-\Pi_0(y)).
	\end{equation*}
	Hence, by the definition of $\widetilde{\psi}_r$, we need to check that $ \phi((5r)^{-1}(x-y))=1$ whenever $\psi_r(\Pi_0(x)-\Pi_0(y))\not= 0$. 
	
	Since $\supp(\psi_r)\subset B(0,2r)$, we get that $|\Pi_0(x)-\Pi_0(y)|\le 2r$. Thus, due to the fact that $\Gamma$ is a $C\theta$-Lipschitz graph, we have
	\begin{equation*}
	|x-y|\le 2(1+C\theta)r\le \frac{5}{2}r.
	\end{equation*}
	Hence, $y\in B(x,5r/2)$, which gives $\phi((5r)^{-1}(x-y))=1$.
\end{proof}

The lemma above and the fact that $\Pi_0$ is bilipschitz between $\Gamma$ and $L_0$ imply that
\begin{align*}
\int_{L_0}\int_0^{\infty}|\psi_r\ast \Pi_*\nu(z)|^2\frac{dr}{r}\ d\mathcal{H}^n(z) &\approx \int_{L_0}\int_0^{\infty}|\psi_r\ast \Pi_*\nu(z)|^2\frac{dr}{r}\ d\Pi_*\sigma(z)\\
&= \int_{\Gamma}\int_0^{\infty}\Lambda_{\nu}(x,r)^2\frac{dr}{r}\ d\sigma(x).
\end{align*}

In consequence of \eqref{eq:classical harmonic analysis result} and the above, to prove \lemref{lem:L2 norm estimate of nu density} it suffices to show that
\begin{equation}\label{eq:lambda estimate to prove}
\int_{\Gamma}\int_0^{\infty}\Lambda_{\nu}(x,r)^2\frac{dr}{r}\ d\sigma(x) \lesssim_{A,\tau} \varepsilon_0^{1/8}\mu(R_0).
\end{equation}

We start with the following simple calculation.

\begin{lemma}
	For $x\in\Gamma$ we have 
	\begin{equation}\label{eq:lambdanu bdd by alpha}
	\Lambda_{\nu}(x,r)\lesssim_{A,\tau} \alpha_{\nu}(x,2r).
	\end{equation}
	Moreover, for $x\in\Gamma\cap 2.5B_0$ and $d(x)\lesssim r<\eta r_0$ we have
	\begin{equation}\label{eq:lambdamu bdd by alpha}
	\Lambda_{\mu}(x,r)\lesssim_{A,\tau} \alpha_{\mu}(3B_Q),
	\end{equation}
	for some $Q\in\Tree$ such that $B(x,5r)\subset 3B_Q$ and $r\approx \ell(Q)$.
\end{lemma}
\begin{proof}
	First, we will prove \eqref{eq:lambdanu bdd by alpha}. Let $B=B(x,2r)$, and $L_B, c_B$ be the minimizing plane and constant for $\alpha_{\nu}(B)$. Using the fact that $\Lambda_{\nu}(x,r) = |\psi_r\ast \Pi_*\nu(\Pi_0(x))|$ we get
	\begin{multline*}
	\Lambda_{\nu}(x,r) = \left| \int \psi_r(\Pi_0(x)-\Pi_0(y))\ d\nu(y)\right|\\
	\le \left| \int \psi_r(\Pi_0(x)-\Pi_0(y)) \ d(\nu-c_B\Hn{L_B})(y)\right|
	+ \left| \int \psi_r(\Pi_0(x)-\Pi_0(y))\ d(c_B\Hn{L_B})(y)\right|\\
	\lesssim r^{-(n+1)}F_B(\nu, c_B\Hn{L_B}) + 0.
	\end{multline*}
	Hence, by $n$-AD-regularity of $\nu$ we arrive at 
	\begin{equation*}
	\Lambda_{\nu}(x,r)\lesssim_{A,\tau} \alpha_{\nu}(x,2r).
	\end{equation*}
	
	Now, let us look at \eqref{eq:lambdamu bdd by alpha}. Since $x\in\Gamma\cap 2.5B_0$ and $d(x)\lesssim r<\eta r_0$, we may find $Q\in\Tree$ such that $B(x,5r)\subset 3B_Q$ and $\ell(Q)\approx_{A,\tau} r$. We use the fact that $|\nabla\widetilde{\psi}_r|\lesssim r^{-n-1}$ and $\supp \widetilde{\psi}_r\subset B(x,5r)$ to get
	\begin{align*}
	\Lambda_{\mu}(x,r)\le& \left| \int \widetilde{\psi}_r(x-y)d(\mu-c_Q\Hn{L_Q}) \right| + c_Q\left| \int \widetilde{\psi}_r(x-y)d\Hn{L_Q} \right|\\
	& \le {r^{-(n+1)}}F_{B(x,5r)}(\mu, c_Q\Hn{L_Q}) + c_Q\left| \int \widetilde{\psi}_r(x-y)d\Hn{L_Q} \right|
	\\ & \lesssim_{A,\tau} \alpha_{\mu}(3B_Q) + c_Q\left| \int \widetilde{\psi}_r(x-y)d\Hn{L_Q} \right|.
	\end{align*}
	We claim that the last integral above is equal to $0$. To prove this, it suffices to show that for $x\in\Gamma$ and $y\in L_Q$ we have $\widetilde{\psi}_r(x-y) = \psi_r(\Pi_0(x)-\Pi_0(y))$, because
	\begin{equation*}
	\int \psi_r(\Pi_0(y)-\Pi_0(x))\ d(\Hn{L_Q})(y) = 0.
	\end{equation*}
	Since $\widetilde{\psi}_r(x-y) = \psi_r(\Pi_0(y)-\Pi_0(x))\phi((5r)^{-1}(x-y))$, we only have to check that $\phi((5r)^{-1}(x-y))=1$ for $\Pi_0(y)-\Pi_0(x)\in\supp \psi_r$. In other words, knowing that $|\Pi_0(y)-\Pi_0(x)|\le 2r$, we expect that $|x-y|\le \frac{5}{2}r$.
	
	Indeed, the fact that $\Gamma$ is a $C\theta$-Lipschitz graph, that $\measuredangle(L_Q,L_0)\le \theta$, $|\Pi_0(y)-\Pi_0(x)|\le 2r$, and \lemref{lem:x from Gamma close to LQ}, imply
	\begin{equation*}
	|\Pi_0^{\perp}(y) - \Pi_0^{\perp}(x)|\lesssim_{A,\tau}\theta r.
	\end{equation*}
	Hence,
	\begin{equation*}
	|x-y|\le2r + C(A,\tau)\theta r \le \frac{5}{2}r,
	\end{equation*}
	as expected.
\end{proof}
Before we proceed, let us state the following auxiliary result. Recall that given a ball $B$, $z(B)$ denotes the center of $B$.
\begin{lemma}[{\cite[Lemma 6.11]{azzam2018characterization}}]\label{lem:small alphas of fsigma}
	Let $B$ be a ball centered on an $\varepsilon$-Lipschitz graph $\Gamma$, and $f$ a function such that
	\begin{equation*}
	\lVert f-f(z(B))\rVert_{L^{\infty}(3B\cap\Gamma)}\lesssim \varepsilon,
	\end{equation*}
	and $f(x)\approx 1$ uniformly for $x\in 3B\cap\Gamma$. Then
	\begin{equation*}
	\int_B \int_0^{r(B)} \alpha_{f\sigma}(x,r)^2 \frac{dr}{r}d\sigma(x)\lesssim \varepsilon^2 r(B)^n,
	\end{equation*}
	where $\sigma$ denotes the surface measure on $\Gamma$.
\end{lemma}

We split the area of integration from \eqref{eq:lambda estimate to prove} into several pieces. We will estimate each of them separately.

\begin{lemma}\label{lem:estimate for A4}
	For every $k\in K$ we have
	\begin{equation*}
	\int_{B_k}\int_0^{\eta^2d(x)}|\Lambda_{\nu}(x,r)|^2\ \frac{dr}{r}d\sigma(x)\lesssim_{A,\tau} \varepsilon_0 r_k^n.
	\end{equation*}
\end{lemma}
\begin{proof}
	By \lemref{lem:properties of Bk} (c) we know that for $x\in B_k$ we have $\eta^{2}d(x)\le \eta^{1/2}r_k$. Hence,
	\begin{equation*}
	\int_{B_k}\int_0^{\eta^2d(x)}|\Lambda_{\nu}(x,r)|^2\ \frac{dr}{r}d\sigma(x) \le \int_{B_k}\int_0^{\eta^{1/2}r_k}|\Lambda_{\nu}(x,r)|^2\ \frac{dr}{r}d\sigma(x).
	\end{equation*}
	
	Let $g(x)=\sum_{j\in K} c_{j} h_{j}(x)$. Note that for $x\in 3B_k\cap\Gamma$ we have $h(x)=1$, due to \eqref{eq:Bk cover Gamma minus RG} and the definition of $h$ \eqref{eq:definition of h}. Thus, by \lemref{lem:ck close to cj for close Bk and Bj},
	\begin{equation*}
	|g(x)-c_{k}| = \big| \sum_{j\in K} (c_j-c_k)h_j(x) \big| \lesssim_{A,\tau}\sqrt{\varepsilon_0}\sum_{j\in K}h_j(x) = \sqrt{\varepsilon_0}.
	\end{equation*}
	Hence, by \eqref{eq:ck comparable to 1}, $g(x)\approx_{A,\tau} 1$. Since $\restr{\nu}{3B_k} = g\restr{\sigma}{3B_k}$, and $\Gamma\cap 3B_k$ is a $C\sqrt{\varepsilon_0}$-Lipschitz graph by \lemref{lem:locally a good graph}, we can apply \lemref{lem:small alphas of fsigma} and get
	\begin{equation}\label{eq:alphanu small on Bk}
	\int_{B_k}\int_0^{\eta^{1/2}r_k}|\Lambda_{\nu}(x,r)|^2\ \frac{dr}{r}d\sigma(x)\overset{\eqref{eq:lambdanu bdd by alpha}}{\lesssim_{A,\tau}}\int_{B_k}\int_0^{\eta^{1/2}r_k}|\alpha_{\nu}(x,2r)|^2\ \frac{dr}{r}d\sigma(x)\lesssim_{A,\tau} \varepsilon_0 r_k^n.
	\end{equation}
\end{proof}

Let $M(\R^d)$ denote the space of finite Borel measures on $\R^d$.
\begin{lemma}[{\cite[Lemma 8.2]{azzam2018characterization}}]\label{lem:boundedness of T}
	For $\lambda\in M(\R^d)$ we define
	\begin{equation*}
	T\lambda(x) = \left(\int_0^{\infty}\Lambda_{\lambda}(x,r)^2\ \frac{dr}{r}\right)^{1/2},
	\end{equation*}
	and for $f\in L^2(\sigma)$ set $T_{\sigma}f = T(f\sigma)$. Then $T_{\sigma}$ is bounded in $L^p(\sigma)$ for $1<p<\infty$, and T is bounded from $M(\R^d)$ to $L^{1,\infty}(\sigma)$. Furthermore, the norms $\lVert T_{\sigma}\rVert_{L^p(\sigma)\rightarrow L^p(\sigma)}$ and $\lVert T\rVert_{M(\R^d)\rightarrow L^{1,\infty}(\sigma)}$ are bounded above by some absolute constants depending only on $p,n$ and $d$.
\end{lemma}

\begin{lemma}\label{lem:estimate for A3}
	We have
	\begin{equation*}
	\int_{\Gamma\cap 2.4B_0}\int_{\eta^2d(x)}^{\eta r_0}\Lambda_{\nu}(x,r)^2\ \frac{dr}{r} d\sigma(x)\lesssim_{A,\tau} \varepsilon_0^{1/8}\ell(R_0)^n.
	\end{equation*}
\end{lemma}
\begin{proof}
	Since $\nu= (\nu_B -h\mu) + (h\mu -\mu_B) + \mu$, for each $x\in \Gamma\cap 2.4B_0$ we split
	\begin{multline}\label{eq:split int of lambda nu}
	\int_{\eta^2d(x)}^{\eta r_0}\Lambda_{\nu}(x,r)^2\ \frac{dr}{r}\\
	\lesssim \int_{\eta^2d(x)}^{\eta r_0}(\Lambda_{\nu_B}(x,r)-\Lambda_{h\mu}(x,r))^2\ \frac{dr}{r} + \int_{\eta^2d(x)}^{\eta r_0}\Lambda_{\mu_B-h\mu}(x,r)^2\ \frac{dr}{r} + \int_{\eta^2d(x)}^{\eta r_0}\Lambda_{\mu}(x,r)^2\ \frac{dr}{r}.
	\end{multline}
	Let
	\begin{equation*}
	H = \bigg\{ x\in \Gamma\cap 2.4B_0\ :\ \int_{\eta^2d(x)}^{\eta r_0}\Lambda_{\mu_B-h\mu}(x,r)^2\ \frac{dr}{r} > \varepsilon_0^{1/4} \bigg\}.
	\end{equation*}
	We divide our area of integration into two parts:
	\begin{multline}\label{eq:divide area of int to I1 I2}
	\int_{\Gamma\cap 2.4B_0}\int_{\eta^2d(x)}^{\eta r_0}\Lambda_{\nu}(x,r)^2\ \frac{dr}{r} d\sigma(x) =\\ \int_{H}\int_{\eta^2d(x)}^{\eta r_0}\Lambda_{\nu}(x,r)^2\ \frac{dr}{r} d\sigma(x) + \int_{\Gamma\cap 2.4B_0\setminus H}\int_{\eta^2d(x)}^{\eta r_0}\Lambda_{\nu}(x,r)^2\ \frac{dr}{r} d\sigma(x) \eqqcolon I_1 + I_2.
	\end{multline}
	In order to estimate $I_1$, note that for $x\in 2.4B_0$ and $r<\eta r_0$ we have $B(x,5r)\subset 2.5B_0$. Since $\supp\widetilde{\psi}_r\subset 5B(0,r)$ we get $\Lambda_{\mu_B-h\mu}(x,r) = \Lambda_{\restr{(\mu_B-h\mu)}{2.5B_0}}(x,r)$. Hence, by \lemref{lem:boundedness of T} applied to $\lambda= \restr{(\mu_B-h\mu)}{2.5B_0}$
	\begin{equation*}
	\sigma(H)\le \sigma\big(\big\{x\in\Gamma\ :\ T\big(\restr{(\mu_B-h\mu)}{2.5B_0}\big)>\varepsilon_0^{1/8} \big\} \big)\lesssim \varepsilon_0^{-1/8}(\mu_B-h\mu)(2.5B_0).
	\end{equation*}
	Since $h=1$ on $3B_0\setminus(\RFar\cup R_G)$ by \lemref{lem:h=1 outside RFar}, $\mu_B(R_G)=(h\mu)(R_G)=0$ by their definition and \eqref{eq:Bk cover Gamma minus RG}, and $\mu(\RFar)$ is small by \lemref{lem:small measure of RFar}, we have 
	\begin{equation*}
	(\mu_B-h\mu)(2.5B_0)\le\mu_B(\RFar)=\mu(\RFar)\lesssim_{A,\tau}\varepsilon_0^{1/2}\ell(R_0)^n.
	\end{equation*}
	Thus, for $\varepsilon_0$ small enough
	\begin{equation*}
	\sigma(H)\le C(A,\tau)\varepsilon_0^{-1/8}\varepsilon_0^{1/2}\ell(R_0)^n\le \varepsilon_0^{1/4} \ell(R_0)^n.
	\end{equation*}
	Now, consider the density $q=\frac{d\restr{\nu}{2.5B_0}}{d\sigma}$. Arguing as before we see that for $x\in 2.4B_0$ and $r<\eta r_0$ we have $\Lambda_{\nu}(x,r)=\Lambda_{q\sigma}(x,r)$. By $n$-AD-regularity of $\nu$ (\lemref{lem:nu is AD regular}) we get $\lVert q\rVert_{L^4(\sigma)}^4\lesssim_{A,\tau}\sigma(2.5B_0)\approx \ell(R_0)^n$. Using the $L^4(\sigma)$ boundedness of $T_{\sigma}$ yields
	\begin{equation}\label{eq:estimate I1}
	I_1\le \int_H |T_{\sigma}q(x)|^2\ d\sigma(x)\le \sigma(H)^{1/2}\lVert T_{\sigma}(q)\rVert_{L^4(\sigma)}^2\lesssim_{A,\tau} \varepsilon_0^{1/8} \ell(R_0)^n.
	\end{equation}
	
	We move on to estimating $I_2$. Observe that by the definition of $H$ we have
	\begin{equation*}
	\int_{\Gamma\cap 2.4B_0\setminus H}\int_{\eta^2d(x)}^{\eta r_0}\Lambda_{\mu_B-h\mu}(x,r)^2\ \frac{dr}{r} d\sigma(x)\lesssim \varepsilon_0^{1/4} \ell(R_0)^n.		
	\end{equation*}
	Thus, by \eqref{eq:split int of lambda nu},
	\begin{multline}\label{eq:divide I2}
	I_2 = \int_{\Gamma\cap 2.4B_0\setminus H}\int_{\eta^2d(x)}^{\eta r_0}\Lambda_{\nu}(x,r)^2\ \frac{dr}{r} d\sigma(x)\\
	\lesssim \int_{\Gamma\cap 2.4B_0}\int_{\eta^2d(x)}^{\eta r_0}(\Lambda_{\nu_B}(x,r)-\Lambda_{h\mu}(x,r))^2\ \frac{dr}{r} d\sigma(x)	+\varepsilon_0^{1/4}\ell(R_0)^n\\
	+ \int_{\Gamma\cap 2.4B_0}\int_{\eta^2d(x)}^{\eta r_0}\Lambda_{\mu}(x,r)^2\ \frac{dr}{r} d\sigma(x) =: I_{21} + \varepsilon_0^{1/4}\ell(R_0)^n+ I_{22}.
	\end{multline}
	
	To handle $I_{22}$, we use \eqref{eq:lambdamu bdd by alpha} to get for $x\in \Gamma\cap 2.4B_0$.
	\begin{equation*}
	\int_{\eta^2d(x)}^{\eta r_0}\Lambda_{\mu}(x,r)^2\ \frac{dr}{r} \lesssim_{\eta} \sum_{\substack{Q\in\Tree\\ x\in 3B_Q}}\alpha_{\mu}(3B_Q)^2\overset{\eqref{eq:sum of alphas containing x estimate}}{\lesssim_{A,\tau}}\varepsilon_0^2.
	\end{equation*}
	Hence, 
	\begin{equation}\label{eq:estimate I22}
	I_{22}= \int_{\Gamma\cap 2.4B_0}\int_{\eta^2d(x)}^{\eta r_0}\Lambda_{\mu}(x,r)^2\ \frac{dr}{r} d\sigma(x)\lesssim \varepsilon_0^2 \ell(R_0)^n.
	\end{equation}
	
	Finally, we deal with the integral $I_{21}$. Observe that, since $\Lambda_{\nu_B}(x,r)-\Lambda_{h\mu}(x,r)= \widetilde{\psi}_r\ast \nu_B(x) - \widetilde{\psi}_r\ast h\mu(x)$, and $|\nabla\widetilde{\psi}_r|\lesssim r^{-n-1}$, we have
	\begin{equation*}
	|\Lambda_{\nu_B}(x,r)-\Lambda_{h\mu}(x,r)|^2\lesssim \left(r^{-n-1}F_{B(x,5r)}(\nu_B, h\mu) \right)^2 \overset{\eqref{eq:nuB and hmu close}}{\lesssim_{A,\tau}}{\varepsilon_0}\Bigg(\sum_{3B_k\cap B(x,5r)\not=\varnothing} \frac{r_k^{n+1}}{r^{n+1}}\Bigg)^2.
	\end{equation*}
	Note that for $k\in K$ such that $3B_k\cap B(x,5r)\not=\varnothing, $ for $\eta^2d(x)<r<\eta r_0$, and for any $y\in 3B_k\cap B(x,5r)$, we have
	\begin{equation}\label{eq:rk smaller than r}
	r_k\overset{\eqref{eq:rk comparable to d(x)}}{\le}d(y)\le d(x) + 5r \le (\eta^{-2} + 5)r.
	\end{equation}
	Thus, $r_k\le \eta^{-3}r$, and for some big $C'=C'(A,\tau)$ we have $B_k\subset C'B_0$. It follows by the Cauchy-Schwarz inequality, the fact that $B_k$ are centered on $\Gamma$, and that they are of bounded intersection, that 
	\begin{align}\label{eq:schwarz applied to rk/r}
	\Bigg(\sum_{3B_k\cap B(x,5r)\not=\varnothing} \frac{r_k^{n+1}}{r^{n+1}}\Bigg)^2\le& \Bigg(\sum_{3B_k\cap B(x,5r)\not=\varnothing} \frac{r_k^{n+2}}{r^{n+2}}\Bigg)\Bigg(\sum_{3B_k\cap B(x,5r)\not=\varnothing} \frac{r_k^{n}}{r^{n}}\Bigg)\\
	&\lesssim \sum_{3B_k\cap B(x,5r)\not=\varnothing} \frac{r_k^{n+2}}{r^{n+2}}.
	\end{align}
	Together with the fact that $r_k\le \eta^{-3} r$ this implies
	\begin{align*}
	I_{21} &= \int_{\Gamma\cap 2.4B_0}\int_{\eta^2d(x)}^{\eta r_0}(\Lambda_{\nu_B}(x,r)-\Lambda_{h\mu}(x,r))^2\ \frac{dr}{r} d\sigma(x)\\
	&\lesssim \varepsilon_0 \int_{\Gamma\cap 2.4B_0}\int_{\eta^2d(x)}^{\eta r_0} \sum_{3B_k\cap B(x,5r)\not=\varnothing} r_k^{n+2} \frac{dr}{r^{n+3}}d\sigma(x)\\
	&= \varepsilon_0 \sum_{k\in K} r_k^{n+2}  \int_{\Gamma\cap 2.4B_0}\int_{\eta^2d(x)}^{\eta r_0}  \one_{B(x,5r)\cap 3B_k\not=\varnothing}(x) \frac{dr}{r^{n+3}}d\sigma(x)\\
	&\le \varepsilon_0\sum_{B_k\subset C'B_0}r_k^{n+2} \int_{\Gamma\cap 2.4B_0}\int_{\eta^3 r_k}^{\eta r_0}  \one_{B(x,5r)\cap 3B_k\not=\varnothing}(x) \frac{dr}{r^{n+3}}d\sigma(x).
	\end{align*}
	Now, note that if $B(x,5r)\cap 3B_k\not=\varnothing$, then
	\begin{equation*}
	x\in B(z_k, 5r+3r_k)\overset{\eqref{eq:rk smaller than r}}{\subset}B(z_k,\eta^{-3}r).
	\end{equation*}
	Hence,
	\begin{multline*}
	I_{21}\lesssim\varepsilon_0 \sum_{B_k\subset C'B_0}r_k^{n+2}\int_{\eta^3 r_k}^{\eta r_0}\int_{B(z_k,\eta^{-3}r)}\ d\sigma(x)  \frac{dr}{r^{n+3}}\\
	\lesssim \varepsilon_0 \sum_{B_k\subset C'B_0}r_k^{n+2}\int_{\eta^3 r_k}^{\eta r_0} \frac{dr}{r^{3}}\\
	\lesssim \varepsilon_0 \sum_{B_k\subset C'B_0}r_k^{n}\lesssim \varepsilon_0\sigma(C'B_0)\lesssim \varepsilon_0\ell(R_0)^n.
	\end{multline*}
	Together with \eqref{eq:divide area of int to I1 I2}, \eqref{eq:estimate I1}, \eqref{eq:divide I2}, and \eqref{eq:estimate I22}, this concludes the proof.
\end{proof}

We are finally ready to complete the proof of \eqref{eq:lambda estimate to prove}. Let us split the area of integration into four subsets:
\begin{align*}
A_1 &= \{(x,r)\ :\ B(x,2r)\cap 2.3B_0 = \varnothing\},\\
A_2 &= \{(x,r)\ :\ B(x,2r)\cap 2.3B_0 \not=\varnothing,\ r>\eta r_0  \},\\
A_3 &= \{(x,r)\ :\ B(x,2r)\cap 2.3B_0 \not=\varnothing,\ \eta^2d(x)<r\le \eta r_0\},\\
A_4 &= \{(x,r)\ :\ B(x,2r)\cap 2.3B_0 \not=\varnothing,\ 0< r\le \min(\eta^2d(x),\, \eta r_0)\},
\end{align*}
we also set
\begin{equation*}
I_i =\iint_{A_i} \Lambda_{\nu}(x,r)^2\ \frac{dr}{r}d\sigma(x).
\end{equation*}

Since $\restr{\nu}{(2.3B_0)^c} = c_0\Hn{L_0\cap (2.3B_0)^c}$ by \eqref{eq:nu flat outside 1.5B0}, for $(x,r)\in A_1$ we have
\begin{equation*}
\Lambda_{\nu}(x,r)=c_0\Lambda_{\Hn{L_0}}(x,r)=0,
\end{equation*}
and so $I_1=0$.

Now let $(x,r)\in A_2$. Since $B(x,2r)\cap 2.3B_0\not=\varnothing,\ r>\eta r_0$, we have
\begin{equation*}
|x-z_0|\le 2r + 2.3r_0<\eta^{-2}r,
\end{equation*}
so that $r\ge \max(\eta r_0,\, \eta^2|x-z_0|)$. It follows that 
\begin{align*}
I_2&\le \int_{\Gamma}\int_{\max(\eta r_0,\, \eta^2|x-z_0|)}^{\infty}	\Lambda_{\nu}(x,r)^2\ \frac{dr}{r}d\sigma(x)\\
&\overset{\eqref{eq:lambdanu bdd by alpha}}{\lesssim_{A,\tau}} \int_{\Gamma}\int_{\max(\eta r_0,\, \eta^2|x-z_0|)}^{\infty}	\alpha_{\nu}(x,2r)^2\ \frac{dr}{r}d\sigma(x)\\
&\overset{\eqref{eq:nu close to flat on big balls}}{\lesssim_{A,\tau}} \varepsilon_0^{1/2}\ell(R_0)^{2n}\int_{\Gamma}\int_{\max(\eta r_0,\, \eta^2|x-z_0|)}^{\infty}	\ \frac{dr}{r^{2n+1}}d\sigma(x)\\
&\approx \varepsilon_0^{1/2}\ell(R_0)^{2n}\int_{\Gamma}	\frac{1}{\max(r_0,\, \eta|x-z_0|)^{2n}}\ d\sigma(x)\\
& \approx \varepsilon_0^{1/2}\ell(R_0)^{2n}\left(\int_{\Gamma\cap 1.9B_0}	\frac{1}{r_0^{2n}}\ d\sigma(x) + \int_{\Gamma\setminus 1.9B_0}	\frac{1}{|x-z_0|^{2n}}\ d\sigma(x) \right) \approx \varepsilon_0^{1/2}\ell(R_0)^n,
\end{align*}
where we used in the last line that $\Gamma\setminus 1.9B_0 = L_0\setminus 1.9B_0$, see \lemref{lem:properties of F}.

Concerning $(x,r)\in A_3$, note that necessarily $x\in 2.4B_0$, and so by \lemref{lem:estimate for A3}
\begin{equation*}
I_3\le \int_{\Gamma\cap 2.4B_0}\int_{\eta^2d(x)}^{\eta r_0}\Lambda_{\nu}(x,r)^2\ \frac{dr}{r} d\sigma(x)\lesssim_{A,\tau} \varepsilon_0^{1/8}\ell(R_0)^n.
\end{equation*}

Finally, for $(x,r)\in A_4$, we only need to consider $x$ such that $d(x)>0$ and $x\in 2.4B_0\cap\Gamma$, and since all such $x$ are contained in some $B_k$ we get
\begin{multline*}
I_4\le \int_{\Gamma\cap 2.4B_0}\int_0^{\eta^2d(x)}\Lambda_{\nu}(x,r)^2\ \frac{dr}{r} d\sigma(x) \le \sum_{B_k\cap 2.4B_0\not=\varnothing} \int_{B_k}\int_0^{\eta^2d(x)}\Lambda_{\nu}(x,r)^2\ \frac{dr}{r} d\sigma(x)\\
\overset{\text{Lemma}\ \ref{lem:estimate for A4}}{\lesssim_{A,\tau}} \sum_{B_k\cap 2.4B_0\not=\varnothing} \varepsilon_0r_k^n\approx \sum_{B_k\cap 2.4B_0\not=\varnothing} \varepsilon_0 \sigma(B_k)\overset{\text{Lemma}\ \ref{lem:properties of Bk}}{\lesssim} \varepsilon_0\sigma(CB_0)\approx \varepsilon_0\ell(R_0)^n.
\end{multline*}
Putting together all the estimates above finishes the proof of \eqref{eq:lambda estimate to prove} and \lemref{lem:L2 norm estimate of nu density}.

\section{Small measure of cubes from \texorpdfstring{$\BA$}{BA}}\label{sec:BS}
We know by \lemref{lem:small measure of RFar} that $\mu(\RFar)\lesssim_{A,\tau} \sqrt{\varepsilon_0}\mu(R_0).$ Thus, in order to estimate the measure of $\bigcup_{Q\in\BA}Q$, it suffices to bound the measure of
\begin{equation*}
R_{\BA} = \bigcup_{Q\in\BA}Q\setminus \RFar.
\end{equation*}
\begin{lemma}\label{lem:estimating with grad F}
	We have
	\begin{equation*}
	\mu(R_{\BA})\lesssim_{A}\theta^{-2}\lVert\nabla F\rVert_{L^2}^2.
	\end{equation*}
\end{lemma}
\begin{proof}
	For every $x\in R_{\BA}$ we define $B_x = B(x, r(Q_x)/100)$, where $Q_x\in \BA$ is such that $x\in Q_x$. We use the $5r$-covering theorem to choose $\{x_i\}_{i\in J}$ such that all $B_{x_i}$ are pairwise disjoint and $\bigcup_i 5B_{x_i}$ covers $\bigcup_{x\in R_{\BA}} B_x$. Observe that
	\begin{equation}\label{eq:5Bxi zawiera sie w 5BQxi}
	5B_{x_i}\subset 3B_{Q_{x_i}}.
	\end{equation}
	
	Set $B_i = \frac{1}{2}B_{x_i}, Q_i = Q_{x_i}$, and let $P_i\in\Tree$ be the parent of $Q_i$. We have $\ell(P_i)\approx\ell(Q_i)\approx r(B_i).$ Since $x_i\not\in\RFar$, we can use \lemref{lem:x outside RFar close to Gamma} to obtain
	\begin{equation*}
	\dist(x_i,\Gamma)\lesssim_{A,\tau}\sqrt{\varepsilon_0}d(x_i)\lesssim\sqrt{\varepsilon_0}\ell(P_i)\approx\sqrt{\varepsilon_0} r(B_i).
	\end{equation*}
	Hence, for small $\varepsilon_0$, we get that $\frac{1}{4}B_i\cap\Gamma\not=\varnothing$. It follows that for each $i\in J$ we can choose balls $B_{i,1},\, B_{i,2}\subset B_i$ centered at $\Gamma$, with $r(B_{i,1})\approx r(B_{i,2})\approx r(B_i)$, and such that $\dist(B_{i,1},B_{i,2})\gtrsim r(B_i)$. Then, for any points $y_k\in B_{i,k}\cap\Gamma,\, k=1,2,$ we have
	\begin{equation}\label{eq:yk well separated}
	r(B_i)\lesssim |y_1-y_2|\lesssim |\Pi_0(y_1) - \Pi_0(y_2)|.
	\end{equation}
	Since $y_1,y_2\in \Gamma\cap B_i\subset \Gamma\cap B_{P_i},$ we have by \lemref{lem:x from Gamma close to LQ}
	\begin{equation*}
	\dist(y_k, L_{P_i})\lesssim_{A,\tau}\sqrt{\varepsilon_0}\ell(P_i),\qquad k=1,2.
	\end{equation*}
	Let $w_k = \Pi_{L_{P_i}}(y_k)$. By the estimate above we have
	$
	|y_k-w_k|\lesssim_{A,\tau}\sqrt{\varepsilon_0}\ell(P_i).
	$
	Moreover, it is easy to see that $w_k\in B_{P_i}$. 
	
	Since $\measuredangle(L_{Q_i},L_0)>\theta$ and $Q_i\in \Tree_0$ by the definition of $\BA$ \eqref{eq:def of BS}, $\ell(Q_i)\approx\ell(P_i)$, and $\dist(Q_i,P_i)=0,$ we may use \lemref{lem:planes close to each other for similar cubes} with $Q_i, P_i$ to get
	\begin{equation*}
	\measuredangle(L_{P_i},L_0)\ge \measuredangle(L_{Q_i},L_0)-\measuredangle(L_{P_i},L_{Q_i})\ge \theta - C(A,\tau)\varepsilon_0 \gtrsim\theta.
	\end{equation*}
	Thus,
	\begin{multline*}
	|F(\Pi_0(y_1)) - F(\Pi_0(y_2))| = |\Pi_0^{\perp}(y_1) - \Pi_0^{\perp}(y_2)|\\
	\ge |\Pi_0^{\perp}(w_1) - \Pi_0^{\perp}(w_2)| - \sum_{k=1}^2 |y_k-w_k| \gtrsim  \theta|\Pi_0(w_1) - \Pi_0(w_2)| - \sum_{k=1}^2 |y_k-w_k|\\
	\ge \theta|\Pi_0(y_1) - \Pi_0(y_2)| - 2\sum_{k=1}^2 |y_k-w_k|\\ \overset{\eqref{eq:yk well separated}}{\gtrsim}\theta r(B_i) - c(A,\tau)\sqrt{\varepsilon_0}r(B_i)\gtrsim\theta r(B_i),
	\end{multline*}
	for $\varepsilon_0$ small enough. 
	
	Now, denoting by $m_i$ the mean of $F$ over the ball $\Pi_0(B_i)$, we have
	\begin{equation*}
	|F(\Pi_0(y_1)) - F(\Pi_0(y_2))|\le |F(\Pi_0(y_1)) -  m_i|+ |F(\Pi_0(y_2)) - m_i|\le 2\max_{k=1,2} |F(\Pi_0(y_k)) -  m_i|.
	\end{equation*}
	Hence, the estimates above give us for some $k\in\{1,2\}$
	\begin{equation}\label{eq:F far from the mean}
	|F(\Pi_0(y_k)) -  m_i|\gtrsim \theta r(B_i).
	\end{equation}	
	Since the estimate above holds for all points $y_k\in B_{i,k}\cap\Gamma$, and $\Pi_0(B_{i,k}\cap\Gamma)\approx r(B_i)^n$, we can use Poincaré's inequality to get
	\begin{equation*}
	r(B_i)^2\int_{\Pi_0(B_i)}|\nabla F(\xi)|^2\ d\mathcal{H}^n(\xi) \gtrsim \int_{\Pi_0(B_i)}|F(\xi) - m_i|^2\ d\mathcal{H}^n(\xi) \gtrsim\theta^2 r(B_i)^{n+2}
	\end{equation*}
	for all $i\in J$. 
	
	We claim that the $n$-dimensional balls $\{\Pi_0(B_i)\}_{i\in J}$ are pairwise disjoint. This follows easily by the fact that $2B_i = B_{x_i}$ are pairwise disjoint, $\frac{1}{4}B_i\cap\Gamma \not = \varnothing$, and $\Gamma$ is a graph of a Lipschitz function with a small Lipschitz constant.
	
	Hence, we may sum the inequality above over all $i\in J$ to finally get
	\begin{multline*}
	\lVert \nabla F\rVert_{L^2}^2 \ge \sum_{i\in J} \int_{\Pi_0(B_i)}|\nabla F|^2\ d\mathcal{H}^n \gtrsim \sum_{i\in J} \theta^2 r(B_i)^n\\
	\overset{\eqref{eq:notHD}}{\gtrsim} A^{-1}\theta^2 \sum_{i\in J} \mu(3B_{Q_i})\overset{\eqref{eq:5Bxi zawiera sie w 5BQxi}}{\gtrsim_A} \theta^2 \sum_{i\in J} \mu(5B_{x_i}) \ge \theta^2\mu(R_{\BA}).
	\end{multline*}
\end{proof}

To estimate $\lVert\nabla F\rVert_{L^2}$ we will use a well-known theorem due to Dorronsoro. We reformulate it slightly for the sake of convenience. In what follows, $\D_{\Gamma}$ denotes the dyadic grid on $\Gamma$, i.e. the image of the standard dyadic grid on $L_0$ under $f(x) = (x,F(x)).$ For $Q\in\D_{\Gamma}$ we set $B_Q=B(z_Q, \diam(Q))$ and $r_Q=\diam(Q)\approx \ell(Q)$.

\begin{theorem}[{\cite[Theorem 2]{dorronsoro1985characterization}}]\label{thm:dorronsoro}
	Let $F:\R^n\rightarrow\R^{d-n}$ be an $L$-Lipschitz function, with $L$ sufficiently small, and let $\Gamma\subset \R^d$ be the graph of $F$, and $\sigma=\Hn{\Gamma}$. Then
	\begin{equation*}
	\int_{\Gamma}\int_0^{\infty} \beta_{\sigma,1}(x,r)^2\ \frac{dr}{r}d\sigma \approx \lVert\nabla F\rVert_{L^2}^2.
	\end{equation*}
\end{theorem}

To estimate the integral above we split the area of integration into four subfamilies:
\begin{align*}
A_1 &= \{(x,r)\ :\  B(x,r)\cap 1.9B_0 =\varnothing\},\\
A_2 &=\{(x,r)\ :\ B(x,r)\cap 1.9B_0 \not=\varnothing,\ r>0.1 r_0 \},\\
A_3 &=\{(x,r)\ :\  B(x,r)\cap 1.9B_0 \not=\varnothing,\ \eta^2 d(x)\le r <0.1 r_0 \},\\
A_4 &= \{(x,r)\ :\ B(x,r)\cap 1.9B_0 \not=\varnothing,\ r< \min(\eta^2 d(x), 0.1 r_0) \},
\end{align*}
we also set
\begin{equation*}
I_i =\iint_{A_i} \beta_{\sigma,1}(x,r)^2\ \frac{dr}{r}d\sigma(x).
\end{equation*}
Firstly, note that for $(x,r)\in A_1$ we have $B(x,r)\cap \Gamma = B(x,r)\cap L_0$ because $\supp(F)\subset 1.9B_0$, and so
\begin{equation}\label{eq:sum of new D1 is 0}
I_1=0.
\end{equation}

\begin{lemma}\label{lem:sum of new D2}
	We have 
	\begin{equation*}
	I_2\lesssim_{A,\tau} \varepsilon_0^{1/2}\ell(R_0)^n.
	\end{equation*}
\end{lemma}
\begin{proof}
	Let $(x,r)\in A_2$. Observe that since $\sigma\approx_{A,\tau}\nu$, we have
	\begin{equation*} 
	\beta_{\sigma,1}(x,r)\approx_{A,\tau}\beta_{\nu,1}(x,r)\overset{\eqref{eq:beta1 alpha estimate}}{\lesssim} \alpha_{\nu}(x,2r)\overset{\eqref{eq:nu close to flat on big balls}}{\lesssim_{A,\tau}}\varepsilon_0^{1/4}\frac{\ell(R_0)^n}{r^n}.
	\end{equation*}
	Note that if $B(x,r)\cap 1.9B_0\not=\varnothing$, then necessarily $x\in B(z_0,1.9r_0 + r)\subset B(z_0,20r)$. Hence,
	\begin{multline*}
	I_2 \le \int_{0.1 r_0}^{\infty} \int_{B(z_0,20r)} \beta_{\sigma,1}(x,r)^2\ d\sigma\frac{dr}{r} \lesssim_{A,\tau}\varepsilon_0^{1/2} \int_{0.1 r_0}^{\infty} \int_{B(z_0,20r)} \frac{\ell(R_0)^{2n}}{r^{2n+1}}\ d\sigma dr\\
	\lesssim\varepsilon_0^{1/2} \int_{0.1 r_0}^{\infty} \frac{\ell(R_0)^{2n}}{r^{n+1}}\ dr\approx \varepsilon_0^{1/2}\ell(R_0)^n.
	\end{multline*}
	
\end{proof}

\begin{lemma}\label{lem:sum of new D4}
	We have
	\begin{equation*}
	I_3\lesssim_{A,\tau} \varepsilon_0\ell(R_0)^n.
	\end{equation*}
\end{lemma}
	\begin{proof}
		Let $(x,r)\in A_3$. Since  $B(x,r)\cap 1.9B_0 \not=\varnothing$ and $\eta^2 d(x)\le r <0.1 r_0$, it is clear that $B(x,2r)\subset 2.1B_0$ and we may find a cube $P=P(x,r)\in\Tree$ such that $B(x,2r)\subset 3B_P$ and $r\approx \ell(P)$. We will estimate the average distance of $B(x,r)\cap\Gamma$ to $L_P$.
		
		Bounding the part corresponding to $B(x,r)\cap R_G\subset3B_P\cap R_G$ is straightforward: \lemref{lem:mu on RG absolutely continuous wrto Hn} states that $d\restr{\mu}{R_G}=g\,d\Hn{R_G}$ with $g\approx_{A,\tau} 1$, and so
		\begin{multline}\label{eq:estimate new on BQcapRG}
		\int_{B(x,r)\cap R_G}\frac{\dist(y,L_P)}{r}\ d\sigma(y) \lesssim_{A,\tau}\int_{3B_P\cap R_G}\frac{\dist(y,L_P)}{\ell(P)}\ d\mu(y)\\
		\lesssim_{A,\tau} \left(\int_{3B_P\cap R_G}\left(\frac{\dist(y,L_P)}{\ell(P)}\right)^2\ d\mu(y) \right)^{1/2}\ell(P)^{n/2}\lesssim_{\tau} \beta_{\mu,2}(3B_P)\ell(P)^{n}.
		\end{multline}	
		
		Dealing with the part outside of $R_G$ is a bit more delicate. By \eqref{eq:Bk cover Gamma minus RG} and the definition of functions $h_k$ \eqref{eq:definition of h},
		\begin{multline*}
		\int_{B(x,r)\setminus R_G}\frac{\dist(y,L_P)}{r}\ d\sigma(y) = \sum_{k\in K} \int_{B(x,r)}\frac{\dist(y,L_P)}{r}\ h_k(y)\ d\sigma(y)\\
		\overset{\eqref{eq:ck comparable to 1}}{\approx_{A,\tau}}\sum_{k\in K} \int_{B(x,r)}\frac{\dist(y,L_P)}{r}\ c_k h_k(y)\ d\sigma(y) =  \int_{B(x,r)}\frac{\dist(y,L_P)}{r}\ d\nu_B(y).
		\end{multline*}
		Consider the 1-Lipschitz function $\Phi(y) = \psi(y) \dist(y,L_P),$ where $\psi$ is $r^{-1}$-Lipschitz, $\psi\equiv 1$ on $B(x,r)$, $|\psi|\le1$, and $\supp(\psi)\subset B(x,2r)$.
		\begin{multline*}
		\int_{B(x,r)}\frac{\dist(y,L_P)}{r}\ d\nu_B(y)\le \int_{B(x,2r)}\frac{\psi(y) \dist(y,L_P)}{r}\ d\nu_B(y)
		\\\le \int_{B(x,2r)}\frac{\psi(y) \dist(y,L_P)}{r}\ h(y)\ d\mu(y) 
		+ r^{-1}\left|\int_{B(x,2r)}\Phi(y)\ d(\nu_B-h\mu)(y)\right|
		\end{multline*}
		
		Since $|\psi|,|h|\le 1$, the first term on the right hand side above can be bounded by $\beta_{\mu,2}(3B_P)\ell(P)^{n}$, just as in \eqref{eq:estimate new on BQcapRG}. Concerning the second term,
		\begin{equation*}
		r^{-1}\left|\int_{B(x,2r)}\Phi(y)\ d(\nu_B-h\mu)(y)\right|\overset{\eqref{eq:nuB and hmu close}}{\lesssim_{A,\tau}}\sqrt{\varepsilon_0}r^{-1}\sum_{3B_k\cap B(x,2r)\not=\varnothing} r_k^{n+1}
		\end{equation*}
		
		Gathering all the calculations above we get that
		\begin{equation}\label{eq:estimating beta1sigma with beta2mu and a sum}
		\beta_{\sigma,1}(x,r)^2\lesssim_{A,\tau} \beta_{\mu,2}(3B_P)^2 + \varepsilon_0\left(\sum_{3B_k\cap B(x,2r)\not=\varnothing} \frac{r_k^{n+1}}{r^{n+1}}\right)^2.
		\end{equation}
		Integrating the first term over $A_3$, since each $P(x,r)$ has sidelength comparable to $r$ and $\dist(P(x,r),x)\lesssim_{A,\tau} r$, it is easy to see that 
		\begin{equation*}
		\iint_{A_3} \beta_{\mu,2}(3B_{P(x,r)})^2\ \frac{dr}{r}d\sigma \lesssim_{A,\tau} \sum_{P\in\Tree} \beta_{\mu,2}(3B_{P})^2\ell(P)^n\overset{\eqref{eq:beta sum estimate}}{\lesssim_{A,\tau}}\varepsilon_0^2\ell(R_0)^n.
		\end{equation*}
		
		Moving on to the second term from \eqref{eq:estimating beta1sigma with beta2mu and a sum}, note that if $y\in 3B_k\cap B(x,2r)\not=\varnothing$, then by \eqref{eq:rk comparable to d(x)} we have $r_k\le d(y)\le 2r+d(x)\le (2+\eta^2) r.$ Thus, following calculations from the proof of \lemref{lem:estimate for A3} (more precisely \eqref{eq:rk smaller than r} and onwards), we get that 
		\begin{equation*}
		\varepsilon_0\iint_{A_3} \left(\sum_{3B_k\cap B(x,2r)\not=\varnothing} \frac{r_k^{n+1}}{r^{n+1}}\right)^2 \frac{dr}{r}d\sigma \lesssim_{A,\tau} \varepsilon_0\ell(R_0)^n.
		\end{equation*}
		Hence, $I_3\lesssim_{A,\tau}  \varepsilon_0\ell(R_0)^n$. 
	\end{proof}
	
	\begin{lemma}\label{lem:sum of new D3}
		We have
		\begin{equation*}
		I_4\lesssim \varepsilon_0\ell(R_0)^n.
		\end{equation*}
	\end{lemma}
	\begin{proof}
		Let $(x,r)\in A_4$, so that $\eta^2d(x)\ge r>0$. It follows by \eqref{eq:Bk cover Gamma minus RG} that $x\in B_k$ for some $k\in K$. Then,
		\begin{equation*}
		r\le \eta^2d(x)\overset{\eqref{eq:rk comparable to d(x)}}{\le}\eta^{1/2}r_k.
		\end{equation*}
		Note also that $x\in 2B_0$. Thus,
		\begin{multline*}
		I_3 \le \sum_{B_k\cap 2B_0\not=\varnothing} \int_{B_k}\int_{0}^{\eta^{1/2}r_k} \beta_{\sigma,1}(x,r)^2\ \frac{dr}{r}d\sigma(x)\\
		\overset{\text{Lemma}\ \ref{lem:nu is AD regular}}{\approx_{A,\tau}}\sum_{B_k\cap 2B_0\not=\varnothing} \int_{B_k}\int_{0}^{\eta^{1/2}r_k} \beta_{\nu,1}(x,r)^2\
		\frac{dr}{r}d\sigma(x)\\
		\overset{\eqref{eq:beta1 alpha estimate}}{\lesssim} \sum_{B_k\cap 2B_0\not=\varnothing} \int_{B_k}\int_{0}^{\eta^{1/2}r_k} \alpha_{\nu}(x,2r)^2\
		\frac{dr}{r}d\sigma(x)
		\overset{\eqref{eq:alphanu small on Bk}}{\lesssim_{A,\tau}} \sum_{B_k\cap 2B_0\not=\varnothing}\varepsilon_0 r_k^n\lesssim\varepsilon_0\ell(R_0)^n.
		\end{multline*}
	\end{proof}	
Putting together the estimates for $I_1,\ I_2,\ I_3$ and $I_4$, we get that
\begin{equation*}
	\int_{\Gamma}\int_0^{\infty} \beta_{\sigma,1}(x,r)^2\ \frac{dr}{r}d\sigma\lesssim_{A,\tau}\sqrt{\varepsilon_0}\ell(R_0)^n\approx \sqrt{\varepsilon_0}\mu(R_0).
\end{equation*}
Thus, \lemref{lem:estimating with grad F} and \thmref{thm:dorronsoro} give us
\begin{equation*}
\mu(R_{\BA})\lesssim_{A,\tau,\theta}\sqrt{\varepsilon_0}\mu(R_0).
\end{equation*}
Taking into account the estimates for other stopping cubes, we arrive at
\begin{equation*}
\mu\left(\bigcup_{Q\in\Stop}Q\right)< \frac{\mu(R_0)}{2}.
\end{equation*}
Thus, $\mu(R_G)\ge 0.5\mu(R_0)$, and since $R_G$ is a subset of the Lipschitz graph $\Gamma$ and $\restr{\mu}{R_G}$ is $n$-rectifiable, the proof of \lemref{lem:main lemma} is finished.


\end{document}